\documentclass[12pt, final]{l4dc2024}

\usepackage{comment}
\usepackage{siunitx}
\usepackage{relsize}
\usepackage{ifthen}
\usepackage{caption}

\usepackage{color}
\usepackage[T1]{fontenc}
\usepackage{psfrag}
\usepackage{booktabs}
\usepackage{mathtools}
\usepackage{xstring}
\usepackage{multirow}
\usepackage{xcolor}
\usepackage{prettyref}
\usepackage{flexisym}
\usepackage{bigdelim}
\usepackage{breqn} 
\usepackage{listings}

\usepackage{xspace}
\usepackage{bm}
\graphicspath{{./figures/}}
\usepackage{tikz}
\usetikzlibrary{matrix,calc}

\newtheorem{assumption}{Assumption}
\newcommand{\cf}{\emph{cf.}\xspace}

\newcommand{\bdmath}{\begin{dmath}}
\newcommand{\edmath}{\end{dmath}}
\newcommand{\beq}{\begin{equation}}
\newcommand{\eeq}{\end{equation}}
\newcommand{\bdm}{\begin{displaymath}}
\newcommand{\edm}{\end{displaymath}}
\newcommand{\bea}{\begin{eqnarray}}
\newcommand{\eea}{\end{eqnarray}}
\newcommand{\beal}{\beq \begin{array}{ll}}
\newcommand{\eeal}{\end{array} \eeq}
\newcommand{\beas}{\begin{eqnarray*}}
\newcommand{\eeas}{\end{eqnarray*}}
\newcommand{\ba}{\begin{array}}
\newcommand{\ea}{\end{array}}
\newcommand{\bit}{\begin{itemize}}
\newcommand{\eit}{\end{itemize}}
\newcommand{\ben}{\begin{enumerate}}
\newcommand{\een}{\end{enumerate}}

\newcommand{\calE}{{\cal E}}

\newcommand{\calI}{{\cal I}}

\newcommand{\calN}{{\cal N}}

\newcommand{\calR}{{\cal R}}
\newcommand{\calS}{{\cal S}}

\newcommand{\calZ}{{\cal Z}}

\newcommand{\eg}{\emph{e.g.,}\xspace}
\newcommand{\ie}{\emph{i.e.,}\xspace}

\renewcommand{\boldsymbol}[1]{{\bm #1}}

\newcommand{\hide}[1]{}

\newcommand{\hiddenText}{{\color{gray} hidden text.}}
\newcommand{\hideWithText}[1]{\hiddenText}

\newcommand{\subject}{\text{ subject to }}

\DeclareMathOperator*{\argmin}{arg\,min}

\newcommand{\norm}[1]{\left\| #1 \right\|}

\newcommand{\tran}{^{\mathsf{T}}}

\newcommand{\trace}[1]{\mathrm{tr}\left(#1\right)}

\newcommand{\inv}{^{-1}}

\newcommand{\zero}{{\mathbf 0}}
\newcommand{\eye}{{\mathbf I}}

\newcommand{\Real}[1]{ { {\mathbb R}^{#1} } }

\newcommand{\SEthree}{\ensuremath{\mathrm{SE}(3)}\xspace}

\newcommand{\SOthree}{\ensuremath{\mathrm{SO}(3)}\xspace}

\newcommand{\scenario}[1]{{\smaller \sf#1}\xspace}

\newcommand{\inprod}[2]{\left\langle #1, #2 \right\rangle}
\newcommand{\bmat}{\left[ \begin{array}}
    \newcommand{\emat}{\end{array}\right]}
\newcommand{\blue}[1]{{\color{blue}#1}}

\newcommand{\linkToPdf}[1]{\href{#1}{\blue{(pdf)}}}
\newcommand{\linkToPpt}[1]{\href{#1}{\blue{(ppt)}}}
\newcommand{\linkToCode}[1]{\href{#1}{\blue{(code)}}}
\newcommand{\linkToWeb}[1]{\href{#1}{\blue{(web)}}}
\newcommand{\linkToVideo}[1]{\href{#1}{\blue{(video)}}}
\newcommand{\linkToMedia}[1]{\href{#1}{\blue{(media)}}}
\newcommand{\award}[1]{\xspace} 

\newcommand{\bbZ}{\mathbb{Z}}
\newcommand{\mle}{\mathrm{MLE}}
\newcommand{\tldtheta}{\tilde{\theta}}
\newcommand{\cbrace}[1]{\left\{ #1 \right\}}
\newcommand{\poly}[1]{\mathbb{R}[#1]}
\newcommand{\sym}[1]{\mathbb{S}^{#1}}

\newcommand{\pd}[1]{\sym{#1}_{++}}
\newcommand{\sos}[1]{\mathrm{SOS}[#1]}
\renewcommand{\deg}[1]{\mathrm{deg}(#1)}
\newcommand{\vol}[1]{\mathrm{Vol}(#1)}

\newcommand{\gray}[1]{{\color{gray}#1}}

\newcommand{\Rgt}{R_{\mathrm{gt}}}
\newcommand{\tgt}{t_{\mathrm{gt}}}
\newcommand{\lmo}{\scenario{LM-O}}
\newcommand{\mvn}{\scenario{MVN}}

\newcommand{\damping}{\mathfrak{b}}
\newcommand{\gravity}{\mathfrak{g}}
\newcommand{\mass}{\mathfrak{m}}
\newcommand{\length}{\mathfrak{l}}

\usepackage{enumitem}
\usepackage{algorithmicx}
\usepackage[vlined,ruled,linesnumbered]{algorithm2e}
\usepackage{adjustbox}
\usepackage{amsfonts}

\title[]{Uncertainty Quantification of Set-Membership Estimation in \\ Control and Perception: Revisiting the Minimum Enclosing Ellipsoid}
\usepackage{times}

\author{%
 \Name{Yukai Tang} \Email{yt3846@princeton.edu}\\
 \addr Department of Operations Research and Financial Engineering, Princeton University, USA
 \AND
 \Name{Jean-Bernard Lasserre} \Email{lasserre@laas.fr}\\
 \addr LAAS-CNRS and Toulouse School of Economics (TSE), Toulouse, France
 \AND
 \Name{Heng Yang} \Email{hankyang@seas.harvard.edu}\\
 \addr School of Engineering and Applied Sciences, Harvard University, USA%
}

\begin{document}

\maketitle
\vspace{-8mm}

\begin{abstract}%
    \emph{Set-membership estimation} (SME) outputs a set estimator that guarantees to cover the groundtruth. Such sets are, however, defined by (many) abstract (and potentially nonconvex) constraints and therefore difficult to manipulate. We present tractable algorithms to compute simple and tight over-approximations of SME in the form of \emph{minimum enclosing ellipsoids} (MEE). We first introduce the \emph{hierarchy of enclosing ellipsoids} proposed by \cite{nie05jgo-minimum}, based on \emph{sums-of-squares relaxations}, that asymptotically converge to the MEE of a basic semialgebraic set.
    This framework, however, struggles in modern control and perception problems due to computational challenges.
    We contribute three computational enhancements to make this framework practical, namely constraints pruning, generalized relaxed Chebyshev center, and handling non-Euclidean geometry. We showcase numerical examples on system identification and object pose estimation.

\end{abstract}

\begin{keywords}%
    Set-Membership Estimation, Minimum Enclosing Ellipsoid, Semidefinite Relaxations
\end{keywords}
   
\vspace{-4mm}
\section{Introduction}
\label{sec:intro}\vspace{-2mm}
Model estimation and learning from measurements is a central task in numerous disciplines~\citep{stengel94book-optimal,barfoot17book-state}. Let $\theta \in \Theta \subseteq \Real{n}$ be an unknown model, and $\bbZ = \{z_i\}_{i=1}^N \in \calZ^N$ be a set of $N$ measurements such that, if $z_i$ is a perfect (noise-free) measurement, it holds 
$
r(\theta,z_i) = 0, i=1,\dots,N,
$
with $r: \Theta \times \calZ \rightarrow \Real{m}$ a given \emph{residual} function that is typically designed from first principles describing how $z_i$ is generated from $\theta$ (see Examples~\ref{ex:sme-system-id}-\ref{ex:sme-object-pose} below). 

{\bf Maximum Likelihood Estimation}. The most popular approach to estimate $\theta$ from $\bbZ$ is \emph{maximum likelihood estimation} (MLE). When the measurements are noisy, the residual is adjusted to \vspace{-6mm}
\bea 
\label{eq:noise-distribution}
r(\theta,z_i) = \epsilon_i, \quad i=1,\dots,N,
\eea
with $\epsilon_i \in \Real{m}$ some small noise. The most common distributional assumption is that $\epsilon_i \sim \calN(0,\Sigma)$ follows a Gaussian distribution, leading to the MLE estimator that is the solution of a \emph{nonlinear least squares} problem~\citep{nocedal99book-numerical,pineda22neurips-theseus}\footnote{Or maximum \emph{a posteriori} estimation (MAP) if there is prior knowledge about the distribution of $\theta$.}\vspace{-4mm}
\bea \vspace{-4mm}
\label{eq:mle-least-squares}
\theta^\star_{\mle} \in \argmin_{\theta \in \Theta} \sum_{i=1}^N \Vert r(\theta,z_i) \Vert_{\Sigma\inv}^2.
\eea 
When the measurement set $\bbZ$ contains \emph{outliers}, the least squares objective in \eqref{eq:mle-least-squares} can be replaced by a \emph{robust loss}~\citep{huber04book-robust,antonante21tro-outlier}. Despite significant interests and progress in the literature, two shortcomings of the MLE framework exist. On one hand, the Gaussian assumption is questionable. In fact, in Appendix \ref{app:sec:non-Gaussian} we show noises generated by modern neural networks on a computer vision example deviate far from a Gaussian distribution.\footnote{One can replace the Gaussian assumption with more sophisticated distributional assumptions, but the resulting optimization often becomes more difficult to solve.} On the other hand, in safety-critical applications, \emph{provably correct uncertainty quantification} of a given estimator is desired (\eg how close is $\theta^\star_\mle$ to the groundtruth). The uncertainty of the MLE estimator \eqref{eq:mle-least-squares}, however, is nontrivial to quantify due to the potential nonlinearity in the residual function $r$.

{\bf Set-Membership Estimation}. An alternative framework, known as \emph{set-membership estimation} (SME, or unknown-but-bounded estimation)~\citep{milanese91automatica-optimal}, seeks to resolve the two shortcomings of MLE. Instead of making a distributional assumption on the measurement noise~\eqref{eq:noise-distribution}, SME only requires the noise to be bounded
\bea
\label{eq:sme-noise-bound}
\Vert \epsilon_i \Vert = \Vert r(\theta,z_i) \Vert \leq \beta_i,\quad i=1,\dots,N,
\eea 
where $\Vert \cdot \Vert$ indicates the $\ell_2$ vector norm (one can choose $\ell_1$ or $\ell_\infty$ norm and our algorithm would still apply). Characterizing the bound of the noise is easier than characterizing the distribution of the noise and can be conveniently done using, \eg conformal prediction with a calibration dataset~\citep{angelopoulos21arxiv-conformal}. Given \eqref{eq:sme-noise-bound}, SME returns a \emph{set estimation} of the model
\begin{equation}
\label{eq:sme-set}
\calS = \left\{ \theta \in \Theta \mid \Vert r(\theta,z_i) \Vert \leq \beta_i, i=1,\dots,N  \right\}, \tag{SME}
\end{equation} 
\ie $\calS$ contains all models \emph{compatible} with the measurements $\bbZ$ under assumption \eqref{eq:sme-noise-bound}. Clearly, the groundtruth must belong to $\calS$, and the ``size'' of $\calS$ informs the uncertainty of the estimated model.

{\bf Challenges}. It is almost trivial to write down the set $\calS$ as in \eqref{eq:sme-set}, which, however, turns into a highly nontrivial object to manipulate. The reasons are (a) each of the constraints ``$\Vert r(\theta,z_i) \Vert\leq \beta_i$'' may be a nonconvex constraint, and/or (b) the number of constraints $N$ may be very large. We use two examples in control and perception, respectively, to illustrate the challenges.

\begin{example}[System Identification~\citep{kosut92tac-set,li23arxiv-learning}]
    \label{ex:sme-system-id}
    Consider a discrete-time dynamical system with state $x \in \Real{n_x}$, control $u \in \Real{n_u}$, and unknown system parameters $\theta \in \Real{n}$\vspace{-2mm}
    \bea \label{eq:dynamical-system}
    x_{k+1} = \phi_0(x_k, u_k) + \sum_{i=1}^n \theta_i \phi_i(x_k, u_k) + \epsilon_k = \Phi(x_k,u_k) \tldtheta + \epsilon_k, \quad k=0,\dots,N-1,
    \eea
    where $\Phi(x_k,u_k) = [\phi_0(x_k,u_k),\dots,\phi_n(x_k,u_k)] \in \Real{n_x \times (n+1)}$ is a (nonlinear) activation function, $\epsilon_k\in\Real{n_x},k=0,\dots,N-1$ are unknown noise vectors assumed to satisfy \eqref{eq:sme-noise-bound},\footnote{Note that our algorithm can allow $\theta$ to appear nonlinearly in the dynamics, but it is sufficient to restrict to dynamics linear in $\theta$ as in \eqref{eq:dynamical-system} for system identification tasks in many real applications.} and $\tldtheta = [1,\theta\tran]\tran$. Given a system trajectory $\{x_0,u_0,\dots,x_{N-1},u_{N-1},x_N \}$, the SME of the parameters $\theta$ is 
    \bea \label{eq:sme-system-id}
    \calS = \left\{ \theta \in \Real{n} \mid \Vert x_{k+1} - \Phi(x_k,u_k) \tldtheta \Vert \leq \beta_k,k=0,\dots,N-1\right\}.
    \eea 
\end{example}

\begin{example}[Object Pose Estimation~\citep{hartley03book-multiple,yang2023object}]\label{ex:sme-object-pose}
Consider a 3D point cloud $\{Y_i \}_{i=1}^N$ and a camera, at an unknown rotation and translation $\theta = (R,t) \in \SEthree$,\footnote{$\SEthree := \SOthree \times \Real{3}$ with $\SOthree := \{ R \in \Real{3\times 3} \mid R\tran R = R R\tran =\eye, \det{R} = +1 \}$.}  observing the point cloud as a set of 2D image keypoints
    \bea \label{eq:camera-measurement-model}
    y_i = \Pi(R Y_i + t) + \epsilon_i, \quad i=1,\dots,N,
    \eea 
    where $\Pi: \Real{3} \rightarrow \Real{2}, \Pi(v) = [v_1/v_3,v_2/v_3]\tran$ denotes the projection of a 3D point onto the 2D image plane and $\epsilon_i \in \Real{2}$ denotes measurement noise satisfying \eqref{eq:sme-noise-bound}. Given pairs of 3D-2D correspondences $\{y_1,Y_1,\dots,y_N, Y_N \}$, the SME of the camera pose $\theta$ is 
    \bea \label{eq:sme-object-pose}
    \calS = \left\{ \theta \in \SEthree \mid \Vert y_i - \Pi(R Y_i + t) \Vert \leq \beta_i,i=1,\dots,N\right\}.
    \eea  
\end{example}

The SME \eqref{eq:sme-system-id} is convex --defined by quadratic inequality constraints-- but $N$ can be large given a long trajectory. The SME \eqref{eq:sme-object-pose}, shown by~\cite[Proposition 3]{yang2023object} to be defined by $N$ quadratic inequalities and $N$ linear inequalities, is unfortunately nonconvex despite that $N<10$.

{\bf Contributions}. 
We propose tractable algorithms based on semidefinite programming (SDP) to simplify the set-membership estimator \eqref{eq:sme-set} while maintaining tight uncertainty quantification. Specifically, we seek to find the \emph{minimum enclosing ellipsoid} (MEE) of \eqref{eq:sme-set}, \ie the ellipsoid that contains $\calS$ with minimum volume. Such an ellipsoid allows us to (i) use the center of the ellipsoid as a point estimator, (ii) provide a (minimum) worst-case error bound for the point estimator, and (iii) generate samples in $\calS$ via straightforward rejection sampling (\ie sample inside the ellipsoid and accept the sample if inside $\calS$).\footnote{Clearly, this would also allow us to obtain an estimate of the volume of $\calS$ by counting the acceptance rate.} Our algorithms are based on the \emph{sums-of-squares} (SOS) relaxation framework proposed in \cite{nie05jgo-minimum}, revisited in \cite{kojima13mp-enclosing}, for computing \emph{a hierarchy of enclosing ellipsoids} that \emph{asymptotically} converge to the MEE of a \emph{basic semialgebraic set}, \ie a set defined by finitely many polynomial (in-)equalities, to which both \eqref{eq:sme-system-id} and \eqref{eq:sme-object-pose} belong.\footnote{Surprisingly, the SOS-based MEE framework seems unexploited in the context of Examples \ref{ex:sme-system-id}-\ref{ex:sme-object-pose}.}
When the SME is convex, together with the celebrated \emph{L{\"o}wner-John's ellipsoid theorem}~\citep{henk12-lowner}, we give an algorithm that can provide a \emph{certificate} of convergence when the MEE has been attained by the SOS hierarchy. We show this vanilla algorithm already outperforms the confidence set of least-squares estimation~\citep{abbasi11colt-regret} in a simple instance of Example~\ref{ex:sme-system-id}.
Unfortunately, applying this algorithm to Example \ref{ex:sme-system-id} with large $N$ and Example \ref{ex:sme-object-pose} encounters three challenges. (C1) A long system trajectory (\eg $N=1000$) in Example~\ref{ex:sme-system-id} renders high-order SOS relaxations overly expensive. We therefore introduce a \emph{preprocessing} algorithm to prune redundant constraints in the set \eqref{eq:sme-system-id} (\eg over $900$ constraints are deemed redundant). (C2) We empirically found for \eqref{eq:sme-set} sets that are nonconvex or defined by many constraints, the SDP solver would encounter serious numerical issues and simply fail. To circumvent this, we propose a two-step approach, where step one draws random samples from \eqref{eq:sme-set} to approximate the \emph{shape matrix} of the enclosing ellipsoid, and step two minimizes the size of the ellipsoid with a fixed shape. The second step, coincidentally, becomes a strict generalization of the \emph{relaxed Chebyshev center} method proposed in \cite{eldar2008minimax}. (C3) The last challenge is to handle the non-Euclidean geometry when enclosing an \eqref{eq:sme-set} of $\SOthree$. By leveraging the special geometry of \emph{unit quaternions}, we show a reduction of computing the MEE on a Riemannian manifold to computing the MEE in Euclidean space. With these computational enhancements, we conduct experiments on system identification (including those that are hard to learn~\citep{tsiamis2021linear}) and object pose estimation that were not possible to perform in existing literatures.

{\bf Paper Organization}. We defer a review of related works to Appendix \ref{sec:related-work} and first present the SOS-based MEE framework in Section \ref{sec:mee-sos}. We describe in Section~\ref{sec:improvement} our computational enhancements. We present numerical experiments in Section~\ref{sec:experiments} and conclude in Section~\ref{sec:conclusion}. Preliminaries on the moment-SOS hierarchy~\citep{lasserre2001global,parrilo03mp-semidefinite} are presented in Appendix~\ref{app:sec:moment-sos-pop}. 


\section{Minimum Enclosing Ellipsoid by SOS Relaxations}
\label{sec:mee-sos}

Consider a basic semialgebraic set defined by a finite number of polynomial constraints
\bea\label{eq:basic-semialgebraic-set}
\calS = \cbrace{\theta \in \Real{n} \mid g_i(\theta) \geq 0, i=1,\dots,l_g, h_j(\theta) = 0, j=1,\dots,l_h}
\eea 
with $g_i,h_j \in \poly{\theta}$ real polynomials in $\theta$. We use the same notation $\calS$ here as in \eqref{eq:sme-set}, \eqref{eq:sme-system-id}, \eqref{eq:sme-object-pose} because the membership sets we consider are all basic semialgebraic sets.  Let $\xi = (\theta_{i_1},\dots,\theta_{i_d}) = P\theta$ be a $d$-dimensional subvector of $\theta$ (with $P$ a selection matrix), $\calS_\xi := \{ P \theta \mid \theta \in \calS \}$ be the restriction of $\calS$ on $\xi$, and consider a $d$-dimensional ellipsoid
\bea\label{eq:ellipsoid}
\calE = \{ \xi \in \Real{d} \mid  1 - (\xi - \mu)^T E (\xi - \mu) \geq 0 \},
\eea
with $\mu \in \Real{d}$ the center and $E \in \pd{d}$ the shape matrix ($E$ is positive definite). We want to find the ellipsoid $\calE$ with minimum volume that encloses $\calS_\xi$
\begin{equation}\label{eq:mee}
    \max_{\mu \in \Real{d},E \in \pd{d}} \cbrace{ \log\det E \mid \xi \in \calE, \forall  \xi \in \calS_\xi}, \tag{MEE}
\end{equation} 
where $\det E$ is inversely proportional to the volume of $\calE$ (\ie \eqref{eq:mee} minimizes the volume of $\calE$). We remark that considering an ellipsoid in the subvector $\xi$ is general and convenient as (i) $P = I$ recovers the $n$-dimensional ellipsoid, (ii) in Example \ref{ex:sme-object-pose} it is desired to enclose the rotation $R$ and translation $t$ separately {because $R$ lives in $\SOthree$ and $t$ lives in $\Real{3}$}, and (iii) having $\xi = \theta_i \in \Real{}$ for some dimension $i$ allows the ellipsoid $\calE_i$ to be a line segment that encloses $\theta_i$ and hence $\calE_1 \times \dots \times \calE_n$ forms an enclosing box of $\calS$.

Problem \eqref{eq:mee} is generally intractable even when $\calS$ is convex.\footnote{When $\calS$ is a set of points, then \eqref{eq:mee} is easy to solve~\citep{gartner99esa-fast,magnani05cdc-tractable,moshtagh05copt-minimum}.} However, denoting $e(\xi) := 1 - (\xi - \mu)^T E (\xi - \mu)$ in \eqref{eq:ellipsoid} as the polynomial in $\xi$, we observe the constraint in \eqref{eq:mee} simply asks $e(\xi)$ to be \emph{nonnegative on the set $\calS$}. This observation allows a hierarchy of convex relaxations of \eqref{eq:mee} based on sums-of-squares (SOS) programming~\citep{lasserre2001global,blekherman12siam-semidefinite}.

\begin{theorem}[MEE Approximation by SOS Programming]
    \label{thm:mee-sos}
    Assume $\calS$ is Archimedean,\footnote{The definition of Archimedeanness is given in \cite[Def. 3.137]{blekherman12siam-semidefinite}. One can make the Archimedean condition trivially hold by adding a constraint $M_\theta - \theta\tran \theta \geq 0$ to \eqref{eq:basic-semialgebraic-set}, which is easy for Examples \ref{ex:sme-system-id}-\ref{ex:sme-object-pose}.} consider the SOS program with an integer $\kappa$ such that $2\kappa \geq \max\{2, \{\deg{g_i}\}_{i=1}^{l_g}, \{\deg{h_j}\}_{j=1}^{l_h} \}$
    \begin{subequations}\label{eq:mee-sos}
        \bea 
    \max_{E,b,c,\sigma_i,\lambda_j} & \log\det E \label{eq:mee-sos-max-logdet} \\
    \subject & 1 - (\xi\tran E \xi + 2 b\tran \xi + c) = \sum_{i=0}^{l_g} \sigma_i(\theta) g_i(\theta) + \sum_{j=1}^{l_h} \lambda_j(\theta) h_j(\theta) \label{eq:mee-sos-equality} \\
     & \sigma_i(\theta) \in \sos{\theta}, \ \ \deg{\sigma_i g_i} \leq 2\kappa, \quad i=0,\dots,l_g \label{eq:mee-sos-sos-multiplier} \\
     & \lambda_j(\theta) \in \poly{\theta}, \ \ \deg{\lambda_j h_j} \leq 2\kappa, \quad j=1,\dots,l_h \label{eq:mee-sos-poly-multiplier} \\
     & \begin{bmatrix}
        E & b \\ b\tran & c 
     \end{bmatrix} \succeq 0 \label{eq:mee-sos-lifting}
        \eea
    \end{subequations}
    where $g_0(\theta) := 1$, $\sos{\theta}$ is the set of SOS polynomials in $\theta$, and $\deg{\cdot}$ denotes the degree of a polynomial. Let $(E_\star, b_\star, c_\star)$ be an optimal solution of \eqref{eq:mee-sos}, then,
    \begin{enumerate}[label = (\roman*)]
        \item for any $\kappa$, we have $\xi\tran E_\star \xi + 2 b_\star\tran \xi + c_\star = (\xi - \mu_\star)\tran E_\star (\xi - \mu_\star)$ with $\mu_\star = -E_\star^{-1}b_\star$, and the ellipsoid $\calE_\kappa = \{\xi \in \Real{d} \mid (\xi - \mu_\star)\tran E_\star (\xi - \mu_\star) \leq 1\}$ encloses $\calS_\xi$;
        
        \item $\vol{\calE_\kappa}$ decreases as $\kappa$ increases, and $\calE_\kappa$ tends to the solution of \eqref{eq:mee} as $\kappa \rightarrow \infty$.
    \end{enumerate}
\end{theorem}

{Although it seems that Equation~\eqref{eq:mee-sos-equality} is hard to satisfy because the LHS only contains a subset of variables comparing to the RHS, we claim that it is not the case. This is because the inequalities $\{g_i\}$ and equalities $\{h_j\}$ contain all variables. Thus, monomial cancellation is easy to happen.} The proof of Theorem~\ref{thm:mee-sos} is given in Appendix~\ref{app:sec:proof-mee-sos} and is inspired by \cite{nie05jgo-minimum,magnani05cdc-tractable}. Problem~\eqref{eq:mee-sos} is convex and can be readily implemented by YALMIP~\citep{lofberg04yalmip} and solved by MOSEK~\citep{aps19mosek}. Its intuition is simple: \eqref{eq:mee-sos-equality}-\eqref{eq:mee-sos-poly-multiplier} ensures $1 - (\xi\tran E \xi + 2 b\tran \xi + c)$ is nonnegative on $\calS$ and hence every feasible solution is an enclosing ellipsoid, \eqref{eq:mee-sos-lifting} uses a lifting technique to convexify the bilinearity of $\mu$ and $E$ in the original ellipsoid parametrization \eqref{eq:ellipsoid}, and the objective \eqref{eq:mee-sos-max-logdet} seeks to minimize the volume. The convergence of the hierarchy follows from Putinar's Positivstellensatz~\citep{putinar1993positive}.

{\bf Certifying Convergence}. Unlike applying SOS relaxations to polynomial optimization, where a certificate of convergence is known~\citep{henrion05-detecting}, detecting the convergence of \eqref{eq:mee-sos} is in general difficult~\citep{lasserre15mp-generalization}. However, when the set $\calS$ is convex, it is possible to leverage L{\"o}wner-John's ellipsoid theorem~\citep{xie2016thesis} to derive a simple convergence certificate. We present such a result in Appendix~\ref{app:sec:proof-certificate-mee} and a numerical example in Appendix~\ref{app:sec:experiments}.


\vspace{-5mm}
\section{Computational Enhancement}
\label{sec:improvement}

The SOS-based algorithm~\eqref{eq:mee-sos} works very well on simple examples. In this section, we describe three challenges of applying \eqref{eq:mee-sos} and present three techniques to enhance its performance.

\vspace{-5mm}
\subsection{Pruning Redundant Constraints}
\label{sec:prune-redundant}

The first challenge arises when the number of inequalities $l_g$ is very large, in which case the convex optimization \eqref{eq:mee-sos} has $l_g+1$ positive semidefinite (PSD) variables (whose sizes grow quickly with the relaxation order $\kappa$). In system identification (Example~\ref{ex:sme-system-id}), a large $l_g = N$ is common as practitioners often collect a long system trajectory to accurately identify the system. 

Nevertheless, since the dimension $n$ of the parameter $\theta$ is usually much smaller than the trajectory length $N$, one would expect many constraints in \eqref{eq:sme-system-id} to be \emph{redundant}. The difficultly lies in how to \emph{identify and prune} the redundant constraints. When all the constraints are linear, algorithms from linear programming can identify redundancy~\citep{caron89jota-degenerate,telgen83ms-identifying,paulraj10mpe-comparative,cotorruelo20arxiv-elimination}. To handle the quadratic constraints in \eqref{eq:sme-system-id}, we propose Algorithm~\ref{alg:prune-redundant-constraints}.

\vspace{-4mm}
\setlength{\textfloatsep}{0pt}%
\begin{algorithm}
    \SetAlgoLined
    \caption{Prune Redundant Constraints in \eqref{eq:sme-system-id} \label{alg:prune-redundant-constraints}}

    \textbf{Input:} Constraint set $\{g_k(\theta) \}_{k \in [N-1]}$ with $g_k(\theta) := \beta_k^2 - \Vert x_{k+1} - \Phi(x_k,u_k)\tldtheta \Vert^2$ (\cf \eqref{eq:sme-system-id}) 

    {\bf Output:} Pruned set of constraints $\{g_k(\theta) \}_{k \in \calI }$

    Initialize $\calI = [N-1]$

    \For{ $k \gets 0$ to $N-1$ }{

        Solve $\underline{g}_k \leq g_k^\star = \min_{\theta} \cbrace{ g_k(\theta) \mid g_i(\theta) \geq 0, i \in \calI\backslash \{k\} }$ \hfill $\triangleleft$ {\smaller \gray{First-order moment-SOS hierarchy}} \label{line:alg-solve-pop}

        \lIf{ $\underline{g}_k^\star \geq 0$ }{
            $\calI \leftarrow \calI \backslash \{k\}$ 
        }
    }

    \textbf{ Return:} $\calI$
\end{algorithm}
\vspace{-5mm}

Intuitively, line \ref{line:alg-solve-pop} of Algorithm \ref{alg:prune-redundant-constraints} seeks to minimize $g_{k}(\theta)$ when the rest of the constraints hold. If $g_k^\star \geq 0$, then $g_k(\theta)$ is redundant as it is implied by the rest of the constraints. The optimization in line \ref{line:alg-solve-pop}, however, is nonconvex because $g_k(\theta)$ is a concave polynomial. Therefore, we use the first-order moment-SOS hierarchy to obtain a lower bound $\underline{g}_k \leq g_k^\star$: if $\underline{g}_k \geq 0$, then $g^\star_k \geq 0$ must hold and $g_k(\theta)$ is deemed redundant. In practice, the first-order moment-SOS hierarchy is very efficient and easily scales to $N=1000$. 
In a pendulum system identification experiment presented in Section~\ref{sec:exp:system-id:pendulum-long-trajectory}, we show Algorithm~\ref{alg:prune-redundant-constraints} effectively prunes over $90\%$ constraints.

\vspace{-3mm}
\subsection{Generalized Relaxed Chebyshev Center}
\label{sec:general-rcc}

The second challenge comes from the ``$\log\det E$'' objective in \eqref{eq:mee-sos}. Maximizing the ``$\log\det$'' of $E$ is convex~\citep{vandenberghe98simaa-determinant}, yet it cannot be written as a standard linear SDP.\footnote{SDPT3~\citep{toh99oms-sdpt3} natively supports $\log\det$ but empirically we found it performs worse than MOSEK.} Consequently, it is often replaced by the geometric mean~\citep{lofberg04yalmip} and modelled with an exponential cone constraint~\citep{aps19mosek}, causing serious numerical issues in our experiments.\footnote{Particularly, MOSEK returns a solution with large duality gap.} Even after pruning redundant constraints using Algorithm~\ref{alg:prune-redundant-constraints}, we found problem~\eqref{eq:mee-sos} still difficult to solve. Similar observations have been reported in~\cite[page 950]{lasserre23crm-pell}.

This motivates solving the following \emph{generalized Chebyshev center} (GCC) problem 
\begin{equation}\label{eq:gcc}
    \eta^\star = \min_{\mu \in \Real{d}} \max_{\xi \in \calS_\xi} \quad  \Vert \mu - \xi \Vert_Q^2 \tag{GCC}
\end{equation}
with a given $Q \succ 0$ (recall $\xi = P \theta$ is a subvector of $\theta$). When $Q = I$, \eqref{eq:gcc} reduces to the usual Chebyshev center problem~\citep{milanese91automatica-optimal,eldar2008minimax}. In \eqref{eq:gcc}, given any $\mu$, the inner ``$\max_\xi$'' computes the maximum $Q$-weighted distance from $\mu$ to the set $\calS_\xi$, denoted as $\eta^\star_\mu$. Therefore, the ellipsoid $\calE_{Q,\mu} = \{ \xi \mid \Vert \xi - \mu \Vert_Q^2 \leq \eta^\star_\mu  \}$ must enclose the set $\calS_\xi$. Via the outer ``$\min_\mu$'', \eqref{eq:gcc} seeks the $\mu$ such that $\calE_{Q,\mu}$ is the smallest enclosing ellipsoid with a given shape matrix $Q$. We found \eqref{eq:gcc} to be much easier to solve than \eqref{eq:mee-sos} due to removing the ``$\log\det E$'' objective.

{\bf Estimating $Q$ from Samples}. Intuitively, \eqref{eq:gcc} may be more conservative than \eqref{eq:mee} because \eqref{eq:gcc} assumes $Q$ is given while \eqref{eq:mee} optimizes the shape matrix. However, if we can uniformly draw samples $(\xi^1,\dots,\xi^{N_s})$ from $\calS_\xi$, then a good estimate of $Q$ can be obtained as
\vspace{-2mm}
\bea \label{eq:estimate-Q}
Q = \frac{1}{N_s}\sum_{i=1}^{N_s} (\xi^i - \bar{\xi}) (\xi^i - \bar{\xi})\tran, \quad \bar{\xi} = \frac{1}{N_s} \sum_{i=1}^{N_s} \xi^i. \vspace{-3mm}
\eea 
We can use three different algorithms to uniformly sample $\calS_\xi$. (i) When $\calS$ is convex, we can use the \emph{hit-and-run} algorithm~\citep{belisle93mor-hit}, where each iteration involves solving a convex optimization. (ii) For Example \ref{ex:sme-object-pose}, \cite{yang2023object} proposed a RANSAG algorithm, where each iteration involves solving a geometry problem. (iii) We can first solve \eqref{eq:gcc} with $Q = I$ to obtain an {enclosing ball}, then perform \emph{rejection sampling} inside the enclosing ball. We use (iii) in our experiments as it is general and does not require convexity.

{\bf Convex Relaxations for \eqref{eq:gcc}}. Problem~\eqref{eq:gcc} is still nonconvex, but we can design a hierarchy of SDP relaxations that asymptotically converges to $\eta^\star$.
\begin{theorem}[Generalized Relaxed Chebyshev Center]
    \label{thm:grcc}
    Let $\calS$ in \eqref{eq:basic-semialgebraic-set} be Archimedean, $\kappa$ be an integer such that $2\kappa \geq \max \{2,\{\deg{g_i}\}_{i=1}^{l_g},\{\deg{h_j} \}_{j=1}^{l_h}\}$, and $s_n(d):=\left( \substack{n+d \\ d} \right)$, consider the following convex quadratic SDP
    \begin{subequations}\label{eq:grcc-relax}
        \bea 
        - \eta_\kappa^\star = \min_{z \in \Real{s_n(2\kappa)}} & 
        L_z(\theta)\tran P\tran Q P L_z(\theta) - \langle C, M_\kappa(z) \rangle \label{eq:gcc-relax-cost}\\
        \subject & M_\kappa(z) \succeq 0, \quad M_0(z) = 1 \label{eq:gcc-relax-1}\\
        & M_{\kappa - \lceil \deg{g_i}/2 \rceil}(g_i z) \succeq 0, \quad i = 1,\dots,l_g \label{eq:gcc-relax-2} \\
        & L_z(h_j [\theta]_{2\kappa-\deg{h_j}}) = 0, \quad j=1,\dots,l_h \label{eq:gcc-relax-3}
        \eea 
        where $z \in \Real{s_n(2\kappa)}$ is the pseudomoment vector in $\theta$ of degree up to $2\kappa$, $L_z(u)$ and $M_\kappa(uz)$ are both linear functions of $z$ given coefficients of a certain polynomial $u(\theta)$ (\cf Appendix~\ref{app:sec:moment-sos-pop}), and $C$ is a constant matrix whose expression is given in Appendix~\ref{app:sec:proof-grcc}. Let $z_{\star,\kappa}$ be an optimal solution to \eqref{eq:grcc-relax} and $\mu_{\star,\kappa} = L_{z_{\star,\kappa}}(\theta)$, then,
        \begin{enumerate}[label = (\roman*)]
            \item for any $\kappa$, the ellipsoid $\calE_{Q,\kappa}:= \{ \Vert \xi -  \mu_{\star,\kappa} \Vert_Q^2 \leq  \eta_\kappa^\star\}$ encloses $\calS_\xi$;
            \item $\eta^\star_\kappa \geq \eta^\star$ for any $\kappa$ and $\eta^\star_\kappa$ converges to $\eta^\star$ as $\kappa \rightarrow \infty$.
        \end{enumerate}
    \end{subequations}
\end{theorem}

The proof of Theorem~\eqref{thm:grcc}, together with a lifting technique to write \eqref{eq:grcc-relax} as a standard linear SDP, is given in Appendix~\ref{app:sec:proof-grcc}. 
The basic strategy is to first apply the moment-SOS hierarchy, with order $\kappa$, to relax the inner ``$\max_{\xi \in \calS_\xi}$'' (a polynomial optimization) as a convex SDP whose decison variable is $z \in \Real{s_n(2\kappa)}$, the pseudomoment vector constrained by \eqref{eq:gcc-relax-1}-\eqref{eq:gcc-relax-3}. Then, by invoking Sion's minimax theorem~\citep{sion1958general}, one can switch ``$\min_{\mu}$'' and ``$\max_{z}$'', after which the ``$\min_{\mu}$'' admits a closed-form solution $\mu = L_z(\theta)$ and leads to a single-level optimization with cost \eqref{eq:gcc-relax-cost}. It is worth noting that letting $\kappa=1$ (and $Q=I$) in \eqref{eq:grcc-relax} recovers the relaxed Chebyshev center in \cite{eldar2008minimax} for bounded error estimation. For this reason, we call $\mu_{\star,\kappa}$ the \emph{generalized relaxed Chebyshev center} of order $\kappa$ (GRCC-$\kappa$). As we will see in Section~\ref{sec:exp:system-id:rand-linear}, the enclosing balls at GRCC-$\kappa$ with $\kappa\geq 2$ is significantly smaller than those of the RCC.

\vspace{-3mm}
\subsection{Handling the non-Euclidean Geometry of $\SOthree$}
\label{sec:handle-so3}

With the \eqref{eq:gcc} formulation, we can approximate the minimum ellipsoid that encloses $\calS_t$, the translation part of the SME in Example~\ref{ex:sme-object-pose}. The last challenge lies in enclosing $\calS_R$, the rotation part of~\eqref{eq:sme-object-pose}. Since the concept of an ellipsoid is not well defined on a Riemannian manifold such as $\SOthree$, it is natural to seek an enclosing \emph{geodesic ball}. This is closely related to the problem of finding the \emph{Riemannian minimax center} of $\calS_R$, defined as~\citep{arnaudon13cg-approximating}
\bea\label{eq:SO3minimax-center}
\mu_{R,\star} = \argmin_{\mu \in \SOthree} \max_{R \in \calS_R} d_{\SOthree}(\mu,R),
\eea 
where, given two rotations $\mu,R \in \SOthree$, $d_{\SOthree}(\mu,R)$ denotes the geodesic distance on $\SOthree$ defined by $d_{\SOthree}(\mu,R) = \arccos((\trace{\mu\tran R} - 1)/2)$. Clearly, if we can find $\mu_{R,\star}$, then we have found the best point estimate from which the worst-case error bound is the smallest.

Unfortunately, problem~\eqref{eq:SO3minimax-center} cannot be directly solved using Theorem~\ref{thm:grcc} because (i) $\mu \in \SOthree$ is constrained (instead of in \eqref{eq:gcc} $\mu \in \Real{d}$ is unconstrained), and (ii) the objective ``$d_{\SOthree}(\mu,R)$'' is not a polynomial. The next result resolves these issues by leveraging unit quaternions.


\begin{theorem}[Minimum Enclosing Ball for Rotations]
    \label{thm:minimum-ball-SO3}
    Suppose there exists a 3D rotation $\bar{R}$ such that $d_{\SOthree}(\bar{R},R) \leq \frac{\pi}{2}, \forall R \in \calS_R$, and let $\bar{q} \in S^3:= \{ q \in \Real{4} \mid \norm{q} = 1 \}$ (or $-\bar{q}$) be any one of the two unit quaternions associated with $\bar{R}$. Consider the following Chebyshev center problem in $\Real{4}$
    \bea \label{eq:quaternion-minimax}
    \mu_\star = \argmin_{\mu \in \Real{4}}  \max_{q \in \Real{4}}  \left\{ \Vert \mu - q \Vert^2 \mid q \in S^3, \calR(q) \in \calS_R, \bar{q}\tran q \geq 0\right\},
    \eea
    where $\calR: S^3 \rightarrow \SOthree$ maps a unit quaternion to a 3D rotation matrix (whose expression is given in Appendix~\ref{app:sec:proof-ball-so3}). Then the projection of $\mu_\star$ onto $S^3$ is the solution to \eqref{eq:SO3minimax-center}, \ie $\mu_{R,\star} = \mu_\star/\norm{\mu_\star}$.
\end{theorem}

The proof of Theorem~\ref{thm:minimum-ball-SO3} is given in Appendix~\ref{app:sec:proof-ball-so3}. Theorem \ref{thm:minimum-ball-SO3} states that we can solve the original Riemannian minimax problem \eqref{eq:SO3minimax-center} by solving \eqref{eq:quaternion-minimax}, which is easily verified to be an instance of \eqref{eq:gcc} and hence can be approximated by the SDP relaxations in Theorem~\ref{thm:grcc}. Theorem~\ref{thm:minimum-ball-SO3} requires the existence of $\bar{R}$ such that all rotations in $\calS_R$ are reasonably close to $\bar{R}$ and similar assumptions are needed in \cite{arnaudon13cg-approximating}. In other words, $\calS_R$ needs to have relatively small uncertainty. This assumption can be numerically verified by first using the RANSAG algorithm in \cite{yang2023object} to compute $\bar{R}$ and then compute the maximum distance between $\bar{R}$ and $\calS_R$, which we found empirically to be below $\frac{\pi}{2}$ for a majority of the test problems.

For a summary of the whole workflow, see Appendix~\ref{app:sec:summary}.

\vspace{-5mm}

\section{Experiments}
\label{sec:experiments}

\vspace{-1mm}
\begin{figure}[h]
	\begin{center}
		\begin{tabular}{cccc}
            \hspace{-9mm}	
			\begin{minipage}{0.25\textwidth}
				\centering
				\includegraphics[width=\textwidth]{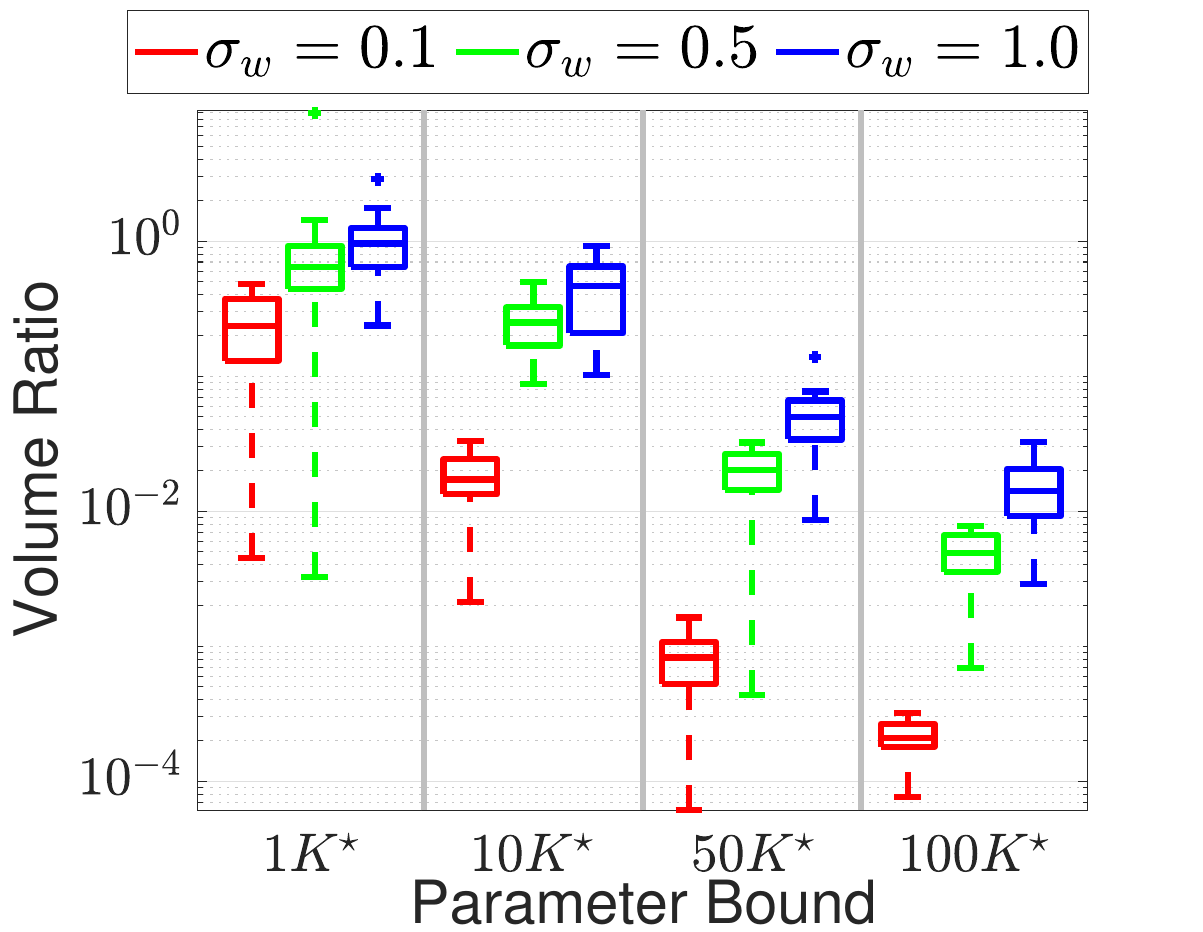}
			\end{minipage}
			& \hspace{-7mm}
			\begin{minipage}{0.25\textwidth}
				\centering
				\includegraphics[width=\textwidth]{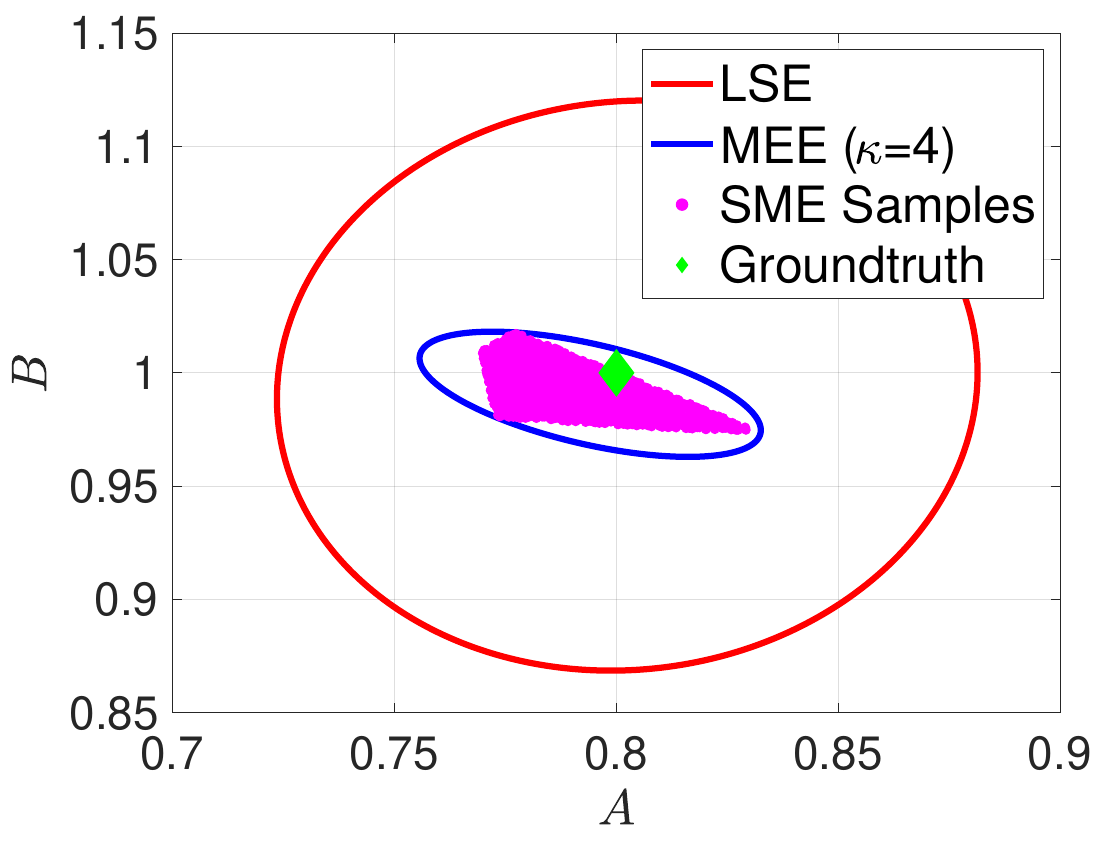}
			\end{minipage}
			&  \hspace{-3mm}
			\begin{minipage}{0.25\textwidth}
				\centering
				\includegraphics[width=\textwidth]{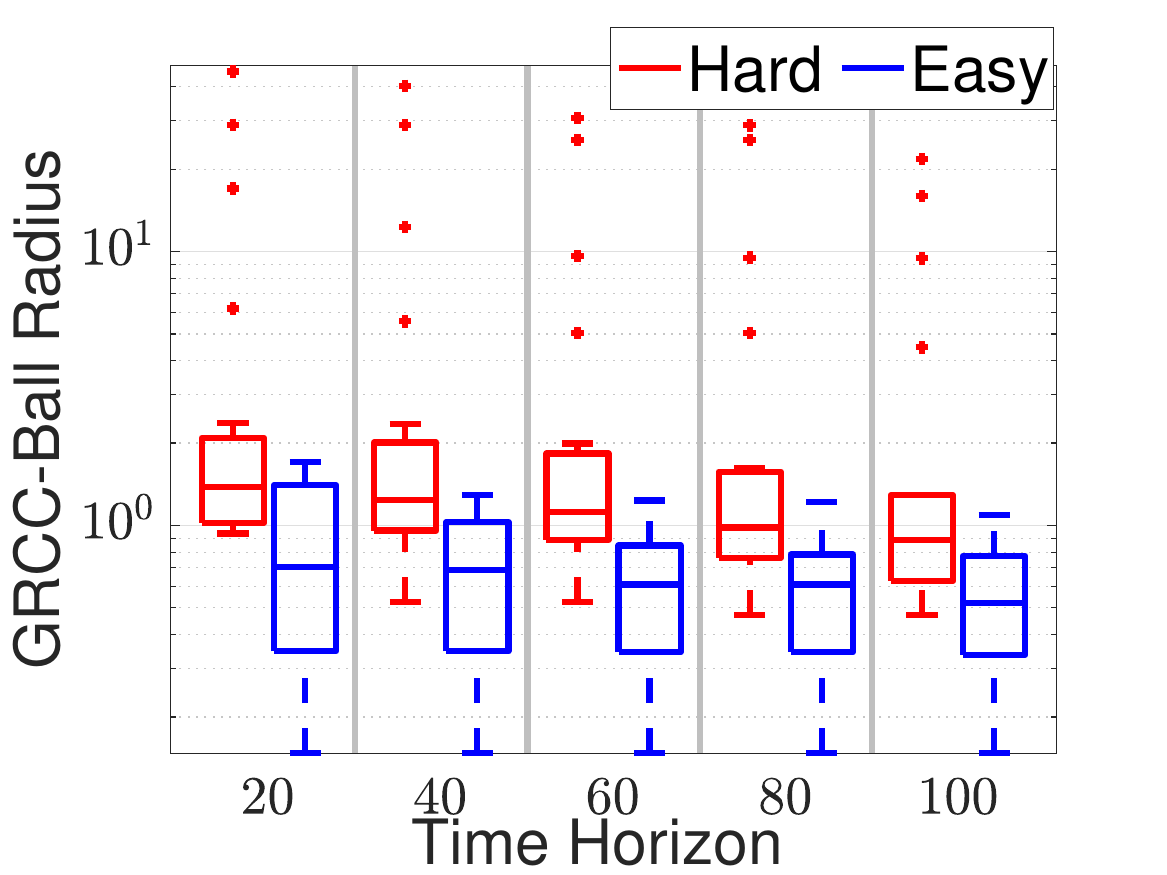}
			\end{minipage}
            &  \hspace{-3mm}
			\begin{minipage}{0.25\textwidth}
				\centering
				\includegraphics[width=\textwidth]{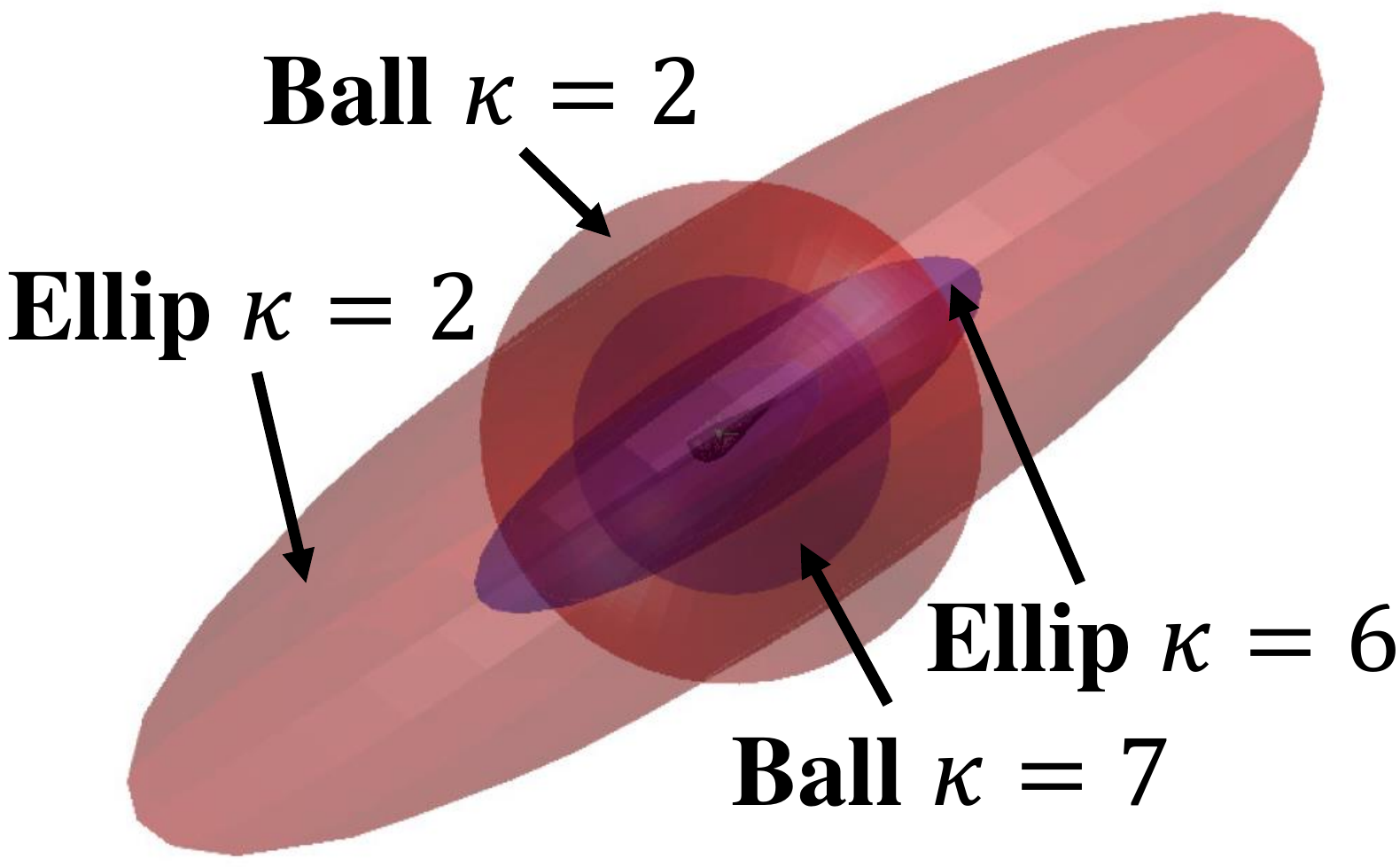}
			\end{minipage} \\
			\multicolumn{2}{c}{\smaller (a) Comparison to least squares estimation} & {\smaller (c) Hard-to-learn systems} & {\smaller (d) Constraints pruning}
        \end{tabular}
        
		\begin{tabular}{ccc}
            \hspace{-10mm}	
            \begin{minipage}{0.34\textwidth}
				\centering
				\includegraphics[width=\textwidth]{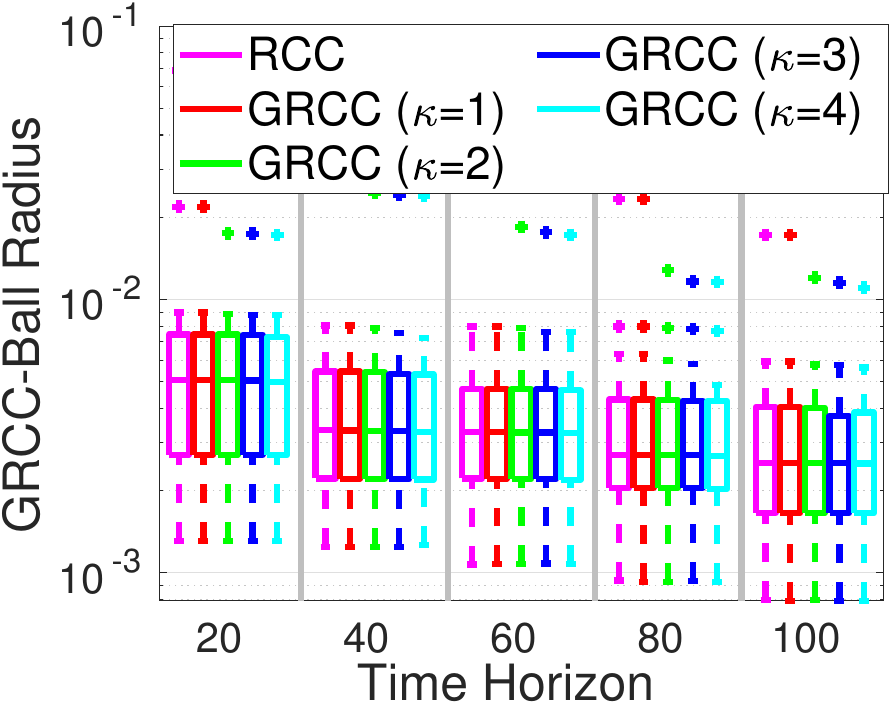}
			\end{minipage}
			&\hspace{-3mm}
			\begin{minipage}{0.34\textwidth}
				\centering
				\includegraphics[width=\textwidth]{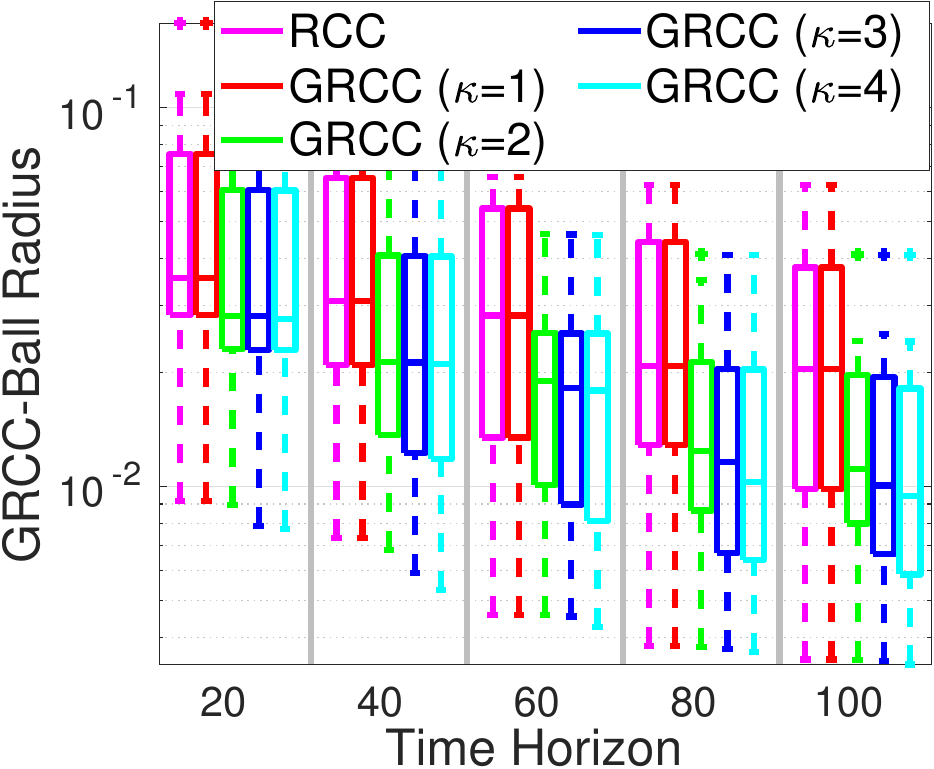}
			\end{minipage}
			&\hspace{-3mm}
			\begin{minipage}{0.34\textwidth}
				\centering
				\includegraphics[width=\textwidth]{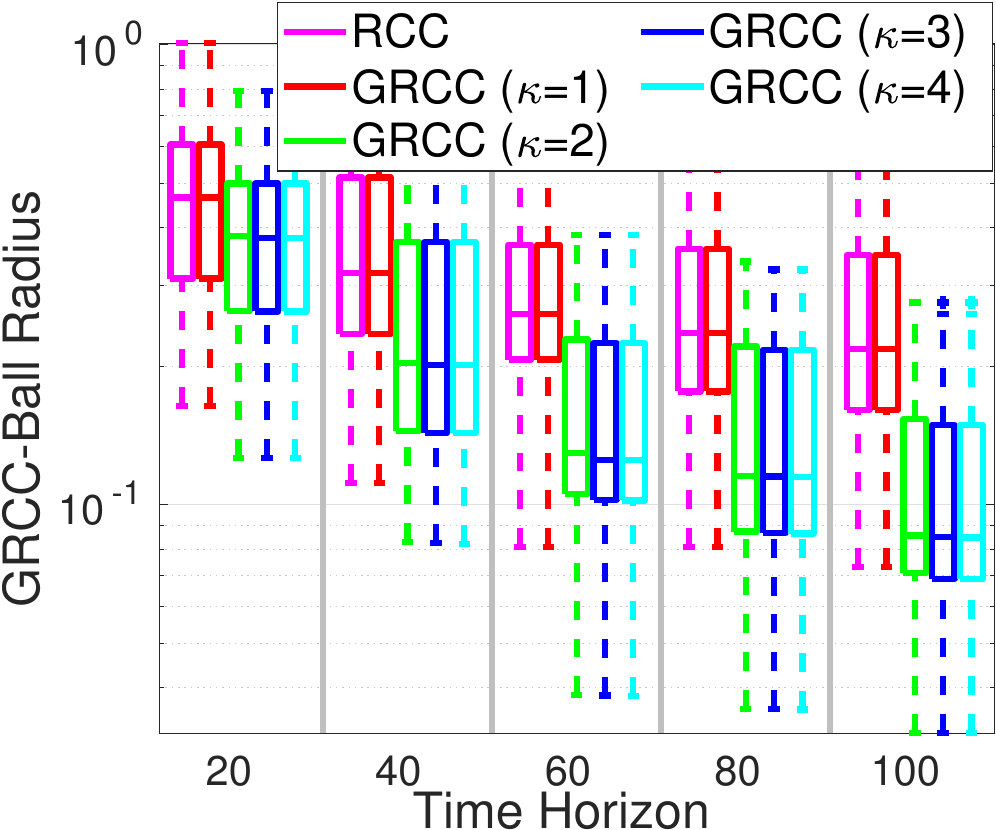}
			\end{minipage}\\
			\multicolumn{3}{c}{\smaller (b) Identification of random linear systems, from left to right: $\beta = 0.01$, $\beta=0.1$, $\beta=1.0$} \\
		\end{tabular}
	\end{center}
	\vspace{-7mm}
	\caption{Experimental results on system identification (Example~\ref{ex:sme-system-id}). 
    \label{fig:exp:system-id}}
	\vspace{-6mm}
\end{figure}

\subsection{System Identification (Example~\ref{ex:sme-system-id})}
\label{sec:exp:system-id}

\vspace{-2mm}
\subsubsection{Comparison with Least Squares Estimation}
\label{sec:exp:system-id:compare-lse}
\vspace{-2mm}
Consider a one-dimensional linear system $x_{k+1} = A_\star x_{k} + B_\star u_k + w_k$ with $A_\star = 0.8, B_\star = 1.0$, and $w_k \sim \calN(0,\sigma_w^2)$.\footnote{This is the same setup as in~\cite{li23arxiv-learning}.} Suppose we collected a system trajectory $x_0,u_0,\dots,x_{N}$, we wish to find a set $\calS$ such that we can guarantee $\mathbb{P}[(A_\star,B_\star) \in \calS ] \geq 80\%$. \cite[Theorem 1]{abbasi11colt-regret} provides an ellipse centered at the least-squares estimator (LSE) that satisfies the probabilistic coverage guarantee, which requires knowing a bound such that $A_\star^2 + B_\star^2 \leq K^2$. Let $K_\star = \sqrt{A_\star^2 + B_\star^2}$ be the best bound.

We can get much smaller coverage sets via the SME \eqref{eq:sme-system-id} and its enclosing ellipsoids without knowledge on $K_\star$. Let $\beta > 0$ be the smallest number such that $\mathbb{P}[w_k^2 \leq \beta] \geq (80\%)^{1/N}$, implying $\mathbb{P}[w_0^2 \leq \beta,\dots,w_{N-1}^2 \leq \beta] \geq 80\%$ (assume $w_k$'s are independent). Consequently, the SME \eqref{eq:sme-system-id} covers $(A_\star,B_\star)$ with probability at least $80\%$. Let $\calE_4$ be the enclosing ellipse of \eqref{eq:sme-system-id} computed via Theorem~\ref{thm:mee-sos} using $\kappa=4$ and $\calE_{\mathrm{LSE}}$ be the confidence ellipse from \cite{abbasi11colt-regret}. We compute $\vol{\calE_2}/\vol{\calE_{\mathrm{LSE}}}$ as the volume ratio, which indicates SME finds a smaller enclosing ellipse when it is below $1$. 

Fig.~\ref{fig:exp:system-id}(a) boxplots the volume ratio (in $\log$ scale) under different choices of $\sigma_w$ and $K$ with fixed $N=50$ (we perform $20$ random experiments in each setup). We observe that LSE outperforms SME only when (i) LSE has access to a precise upper bound $K_\star$ (which we think is often unrealistic) and (ii) the noise $w_k$ is large. In all other cases, the enclosing ellipse of SME is orders of magnitude smaller than that of LSE. A sample visualization is provided in Fig.~\ref{fig:exp:system-id}(a). {The parameters are $K = K^*, \sigma_w = 0.1$ and the number of sampled points is $5000$.

However, it is worth noting that in order to apply SME, we need to know the noise bound (or a high probability bound) $\beta$, while LSE does not neccessarily require such bounds or distributional assumptions unless one wants to estimate the uncertainty.}
\vspace{-3mm}
\subsubsection{Identification of Random Linear Systems}
\label{sec:exp:system-id:rand-linear}
\vspace{-2mm}
Consider a linear system $x_{k+1} = A_\star x_k + B u_k + w_k,k=0,\dots,N-1$. We generate random $A_\star \in \Real{n_x \times n_x},B \in \Real{n_x \times n_u}$ with each entry following a standard Gaussian distribution $\calN(0,1)$, then truncate the singular values of $A_\star$ that are larger than 1 to 1. $u_k$ follows $\calN(0,I)$, and $w_k$ is random inside a ball with radius $\beta$. We treat $B$ as known. We compute enclosing balls of the SME~\eqref{eq:sme-system-id} using two algorithms: the RCC algorithm in \cite{eldar2008minimax}, and our GRCC algorithm presented in Theorem~\ref{thm:grcc}. Fig.~\ref{fig:exp:system-id}(b) plots the radii of the enclosing balls with $N\in\{20,40,60,80,100\}$, $n_x = 2$, and $\beta \in \{0.01,0.1,1.0 \}$ (each boxplot summarizes $20$ random experiments). We observe that (i) GRCC ($\kappa=1$) leads to exactly the same result as RCC, verifying that our GRCC algorithm recovers the RCC algorithm with $\kappa=1$. (ii) The enclosing balls get much smaller with high-order relaxations (\ie larger $\kappa$). {More results with $n_x = 3$ and $n_x = 4$, also the runtime for the $n_x = 2$ experiment are presented in Appendix~\ref{app:sec:experiments}.}

\vspace{-3mm}
\subsubsection{Linear Systems that Are Hard to Learn}
\label{sec:exp:system-id:hard-to-learn}
\vspace{-2mm}
\cite{tsiamis2021linear} presented examples of linear systems that are hard to learn. We show that SME and its enclosing balls are \emph{adaptive to the hardness of identification}, \ie one gets large enclosing balls when the system is hard to learn.
Consider the system $x_{k+1} = A x_k + H w_k$ with 
$$
    A = \begin{bmatrix}
        0 & \theta_1 & 0 \\
        0 & 0 & \theta_2 \\
        0 & 0 & 0
    \end{bmatrix}, \quad H = \begin{bmatrix}
        1 & 0 \\
        0 & 0 \\
        0 & 1
    \end{bmatrix}
$$
where $w_k$ is a random disturbance with noise bound 0.1. Consider (i) an easy-to-learn system with $\theta_2=1$, and (ii) a hard-to-learn system with $\theta_2 = 10^{-5}$. Let $\theta_1 = 1$. Fig.~\ref{fig:exp:system-id}(c) boxplots the radii of the enclosing balls for the SME of $(\theta_1,\theta_2)$ with $\kappa=2$ and increasing $N$. Clearly, we observe that the radii for hard-to-learn systems are much larger than that for easy systems. 


\vspace{-3mm}
\subsubsection{Constraints Pruning in A Long Trajectory}
\label{sec:exp:system-id:pendulum-long-trajectory}
Consider the continuous-time dynamics of a simple pendulum
\vspace{-5mm}
$$
\begin{bmatrix}
x_1 \\
x_2
\end{bmatrix} =
\begin{bmatrix}
    x_2 \\ - \frac{\damping}{\mass \length^2} x_2 + \frac{1}{\mass \length^2} u - \frac{\gravity}{\length} \sin x_1
\end{bmatrix}
=
\begin{bmatrix}
    0 & 1 & 0 & 0  \\
    0 & -\frac{\damping}{\mass \length^2} & \frac{1}{\mass \length^2} & - \frac{\gravity}{\length}
\end{bmatrix}
\begin{bmatrix}
    x_1  \\
    x_2 \\
    u \\
    \sin x_1
\end{bmatrix}
$$
where $x=(x_1,x_2)$ is the state (angle and angular velocity), $\damping$ is the damping ratio, $\mass$ is the mass, $\length$ is the length of the pole, and $\gravity$ is the gravity constant. We wish to identify $\theta_1 = \frac{\damping}{\mass \length^2}$, $\theta_2 = \frac{1}{\mass \length^2}$ and $\theta_3 = \frac{\gravity}{\length}$. To do so, we discretize the dynamics using Euler method with $dt = 0.01$, add random disturbance $w_k$ that has bounded norm 0.1, and collect a single trajectory of length $N=1000$. Without constraint pruning, we can only run the GRCC algorithm with $\kappa=3$ 
. We then prune the constraint set using Algorithm~\ref{alg:prune-redundant-constraints}, which leads to a much smaller set of $19$ constraints (only $1.9\%$ of the original number of constraints). We can then increase the relaxation order of GRCC and the resulting enclosing ellipsoid and ball become much smaller as visualized in Fig.~\ref{fig:exp:system-id}(d). The volume of the enclosing ball with $\kappa=7$ is only $1.08 \times 10^{-6}$, indicating the SME has almost converged to a single point. The time result is provided in Appendix~\ref{app:sec:experiments}.

\vspace{-4mm}
\subsection{Object Pose Estimation (Example~\ref{ex:sme-object-pose}) }
\label{sec:exp:purse}
\vspace{-2mm}
We follow the same procedure as \cite{yang2023object}, \ie we use conformal prediction to calibrate the norm bounds $\beta_i$ for the noise vectors $\epsilon_i$ (\cf \eqref{eq:camera-measurement-model}) generated by the pretrained neural network in~\cite{pavlakos17icra-heatmap} and form the SME~\eqref{eq:sme-object-pose}. We then use the GRCC algorithm in Theorem~\ref{thm:grcc} with $\kappa=3$ to compute enclosing balls (for the rotation we apply GRCC in Theorem~\ref{thm:minimum-ball-SO3} to compute enclosing geodesic balls). We compare the radii of the enclosing balls obtained by GRCC with the radii of the enclosing balls obtained by RANSAG of \cite{yang2023object}.\footnote{RANSAG first estimates an average rotation and translation, then uses SDP relaxations to compute the inner ``$\max$'' problem in \eqref{eq:gcc}. In other words, RANSAG does not seek to find a better estimate with smaller error bounds.} Fig.~\ref{fig:exp:purse}(a) shows the empirical cumulative distribution function (CDF) of the translation bounds and rotation bounds, respectively (there are $7035$ translation problems and $6661$ rotation problems). We observe that the error bounds obtained by GRCC are smaller than those obtained by RANSAG. Examples of enclosing balls and ellipsoids for the translation SME are shown in Fig.~\ref{fig:exp:purse}(b), where we observe that the enclosing ellipsoid precisely captures the shape of the SME. Fig.~\ref{fig:exp:purse}(b) also plots the enclosing balls of the rotation SME using stereographic projection.

\vspace{-3mm}
\begin{figure}[h]
	\begin{center}
		\begin{tabular}{cccc}
            \hspace{-10mm}	
			\begin{minipage}{0.28\textwidth}
				\centering
				\includegraphics[width=\textwidth]{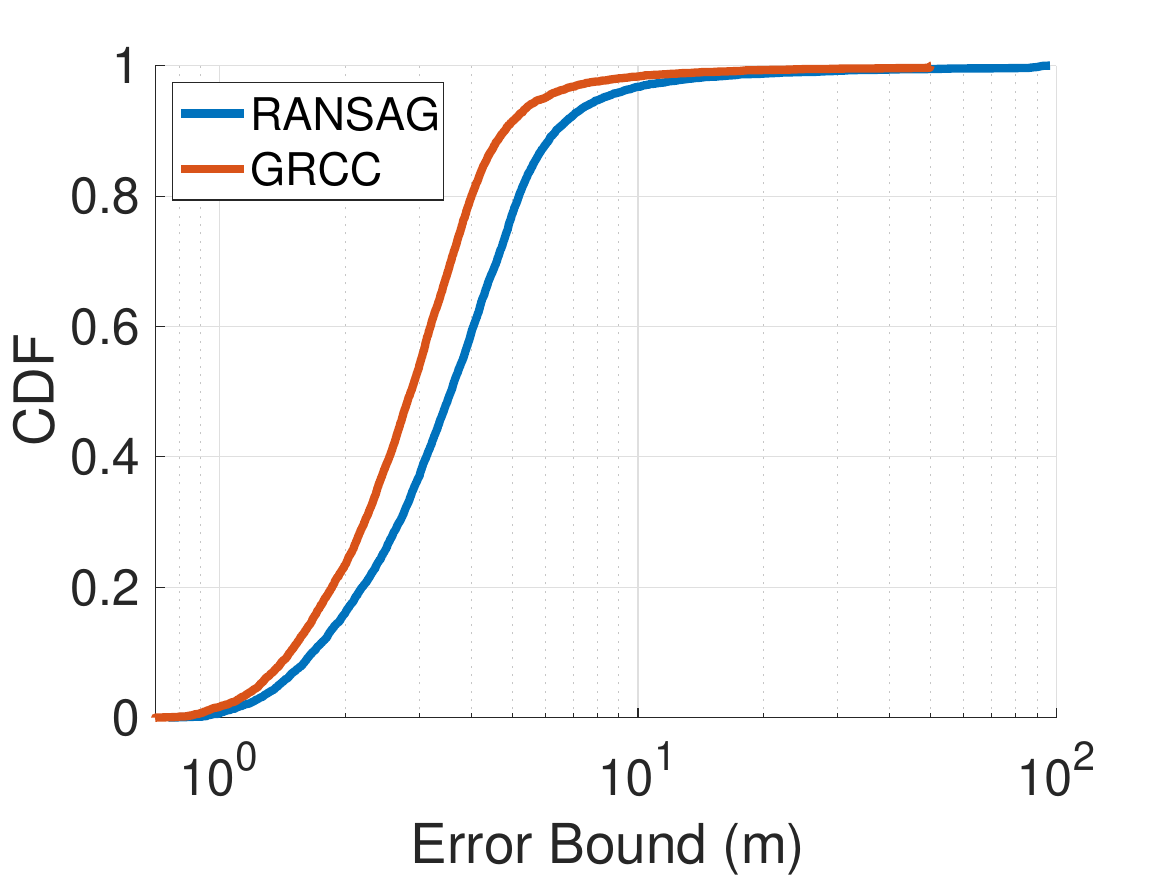}
			\end{minipage}
			&\hspace{-6mm}
			\begin{minipage}{0.28\textwidth}
				\centering
				\includegraphics[width=\textwidth]{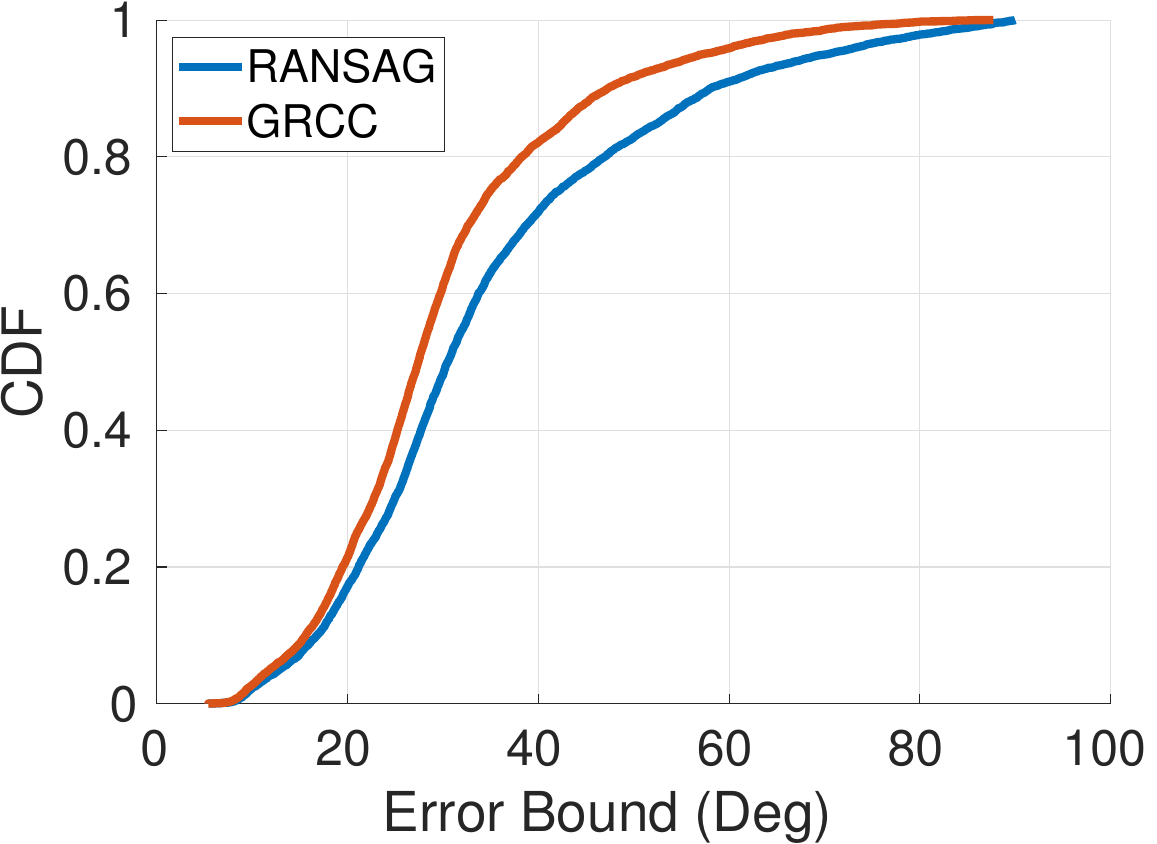}
			\end{minipage}
			&\hspace{-2mm}
			\begin{minipage}{0.23\textwidth}
				\centering
				\includegraphics[width=\textwidth]{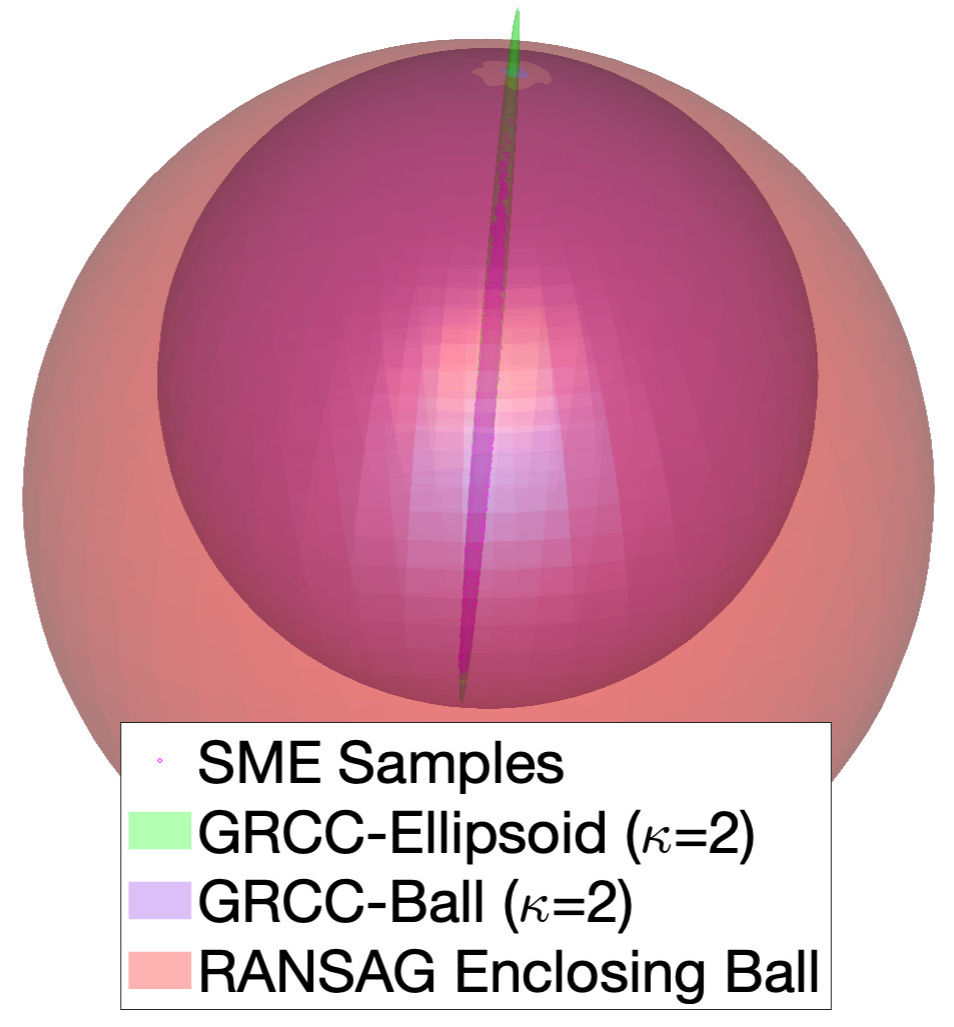}
			\end{minipage}
            &\hspace{-4mm}
			\begin{minipage}{0.23\textwidth}
				\centering
				\includegraphics[width=\textwidth]{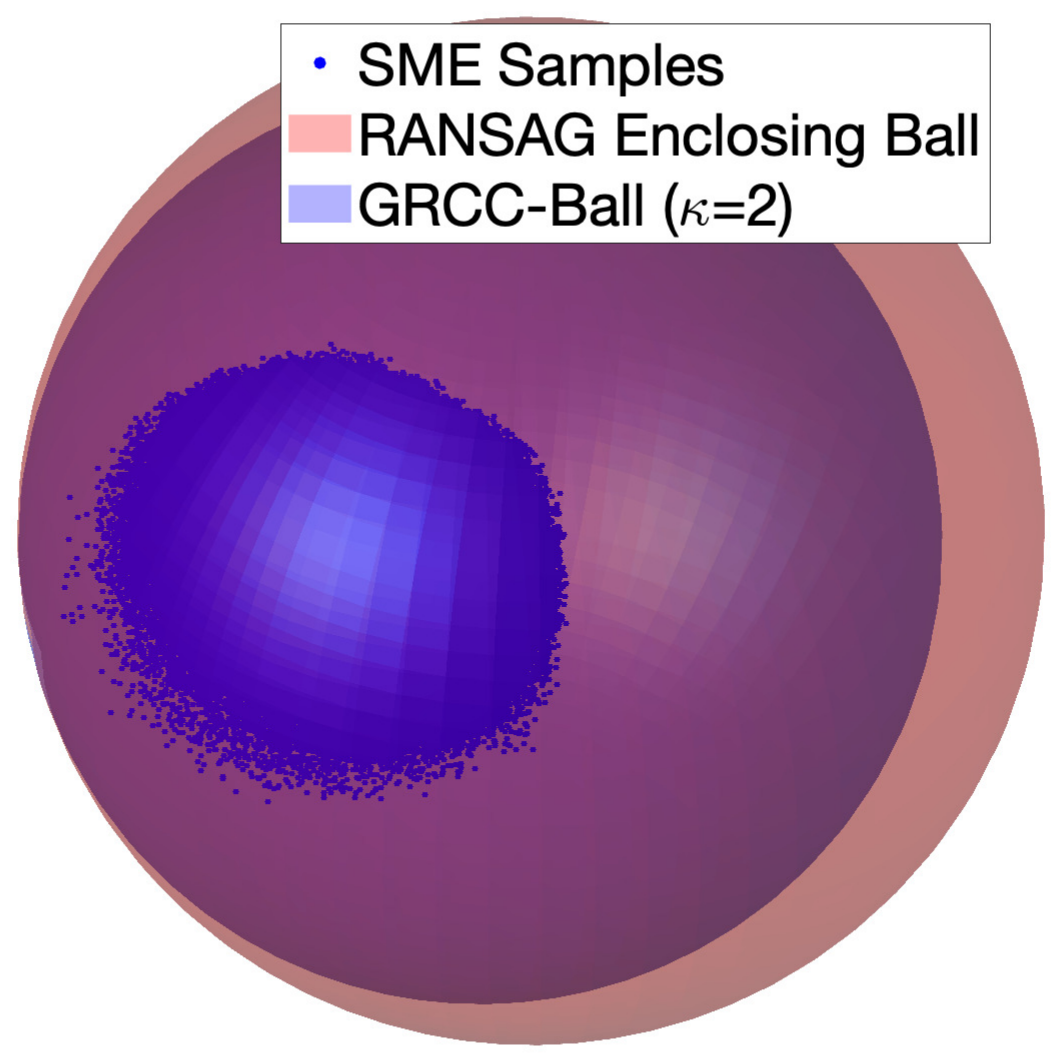}
			\end{minipage} \\
			\multicolumn{2}{c}{\smaller (a) CDF plots. Left: translation, right: rotation} & \multicolumn{2}{c}{\hspace{-4mm} \smaller (b) Enclosing balls/ellipsoids. Left: translation, right: rotation.} 
        \end{tabular}
	\end{center}
	\vspace{-6mm}
	\caption{Experimental results on object pose estimation (Example~\ref{ex:sme-object-pose}). 
    \label{fig:exp:purse}}
	\vspace{-5mm}
\end{figure}


For extra experiments, we refer the readers to the Appendix~\ref{app:sec:experiments}.

\vspace{-5mm}

\section{Conclusions}
\label{sec:conclusion}
We introduced a suite of computational algorithms based on semidefinite programming relaxations to compute minimum enclosing ellipsoids of set-membership estimation in system identification and object pose estimation. Three computational enhancements are highlighted, namely constraints pruning, generalized relaxed Chebyshev center, and handling non-Euclidean geometry. These algorithms are still limited to small- and medium-sized problems (though these problems are already interesting) due to computational challenges in semidefinite programming. Multiple future research directions are possible, \eg applying SME to system identification with partial observations, extending SME on object pose estimation to more perception problems, and integrating SME with adaptive control and reinforcement learning.


\clearpage
\appendix
\setcounter{equation}{0}
\setcounter{table}{0}
\setcounter{figure}{0}
\setcounter{theorem}{0}
\renewcommand{\theequation}{A\arabic{equation}}
\renewcommand{\theproposition}{A\arabic{proposition}}
\renewcommand{\thetheorem}{A\arabic{theorem}}
\renewcommand{\theassumption}{A\arabic{assumption}}
\renewcommand{\thefigure}{A\arabic{figure}}
\renewcommand{\thetable}{A\arabic{table}}

\section{Non-Gaussian Measurement Noise from Neural Networks}
\label{app:sec:non-Gaussian}

In this section, we provide, to the best of our knowledge, the first numerical evidence that the noise in measurements generated by neural networks in object pose estimation (Example~\ref{ex:sme-object-pose}) does not follow a Gaussian distribution. 

Our strategy is to compute the noise vector $\epsilon_i$ from \eqref{eq:camera-measurement-model} as
$$
\epsilon_i = y_i - \Pi(R_{\mathrm{gt}}Y_i + t_{\mathrm{gt}}),\quad i = 1,\dots,N,
$$
where $(\Rgt,\tgt)$ is the groundtruth camera pose, and $y_i$'s are neural network detections of the object keypoints. We use the LineMOD Occlusion (\lmo) dataset~\citep{brachmann14eccv-lmo} that includes 2D images picturing a set of household objects on a table from different camera perspectives. The groundtruth camera poses are annotated and readily available in the dataset. We use the 2D keypoint detector in~\cite{pavlakos17icra-heatmap} that was trained on the \lmo dataset. At test time, the trained model outputs a heatmap of the pixel location of each semantic keypoint, and following~\cite{pavlakos17icra-heatmap} we set $y_i$ as the peak (maximum likelihood) location in the heatmap for each keypoint.

This procedure produces a set of $N=83,347$\footnote{There are 8 different objects in the dataset, each with about 8 semantic keypoints. The dataset has 1214 images.} 2D noise vectors $\{ \epsilon_i \}_{i=1}^N$ and we want to test if they are drawn from a multivariate Gaussian (normal) distribution. To do so, we use the R package \mvn~\citep{korkmaz14R-mvn} that provides a suite of popular multivariate normality tests well established in statistics. To consider potential outliers in the noise vectors $\{\epsilon_i\}_{i=1}^N$ (\ie maybe only the smallest noise vectors satisfy a Gaussian distribution~\citep{antonante21tro-outlier}), we run \mvn on $\alpha\%$ of the noise vectors with smallest magnitudes, and we sweep $\alpha$ from $1$ (\ie keep only the $1\%$ smallest noise vectors) up to $100$ (\ie keep all noise vectors).

\begin{table}[h]
	\centering
    \begin{adjustbox}{width=1\textwidth}
	\begin{tabular}{ c|c|c|c|c|c|c } 
		\hline
		Percentage & Mardia Skewness & Mardia Kurtosis &  Henze-Zirkler & Royston &  Doornik-Hansen & Energy\\ 
		\hline
		1\% & YES & NO & NO & NO & NO & NO \\
		5\% & NO & NO & NO & NO & NO & NO \\
		10\% & NO & NO & NO & NO & NO & NO \\
		20\% & NO & NO & NO & NO & NO & NO \\
		40\% & NO & NO & NO & NO & NO & NO \\
		100\% & NO & NO & NO & NO & NO & NO \\
		\hline	
	\end{tabular}
\end{adjustbox}
    
	\caption{Results of \mvn tests on the noise vectors generated by neural network measurements in the \lmo dataset. We use the pretrained network from~\citep{pavlakos17icra-heatmap}.\label{tab:mvn-test-lmo}}
\end{table}

Table~\ref{tab:mvn-test-lmo} shows the test results. We can see that among the 36 tests performed, the data passed the test only once. This gives strong evidence that the noise vectors do not follow a Gaussian distribution, even after filtering potential outliers. Fig.~\ref{fig:mvn-test-lmo-perspective-plot} shows the perspective plots (top) and the Chi-square quantile-quantile (Q-Q) plots of the empirical density functions under different inlier ratios, in comparison to that of a Gaussian distribution. We can see that the empirical density functions deviate far away from a Gaussian distribution, and are difficult to characterize. This motivates the set membership estimation framework in Section~\ref{sec:intro}.

\begin{figure}[h]
    \hspace{-4mm}
	{\includegraphics[width=0.14\textwidth]{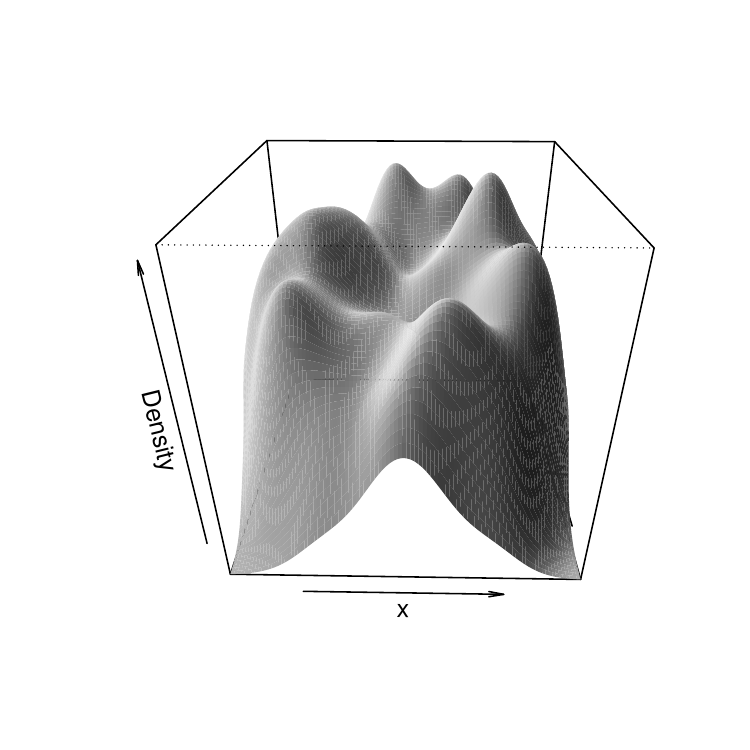}} 
	{\includegraphics[width=0.14\textwidth]{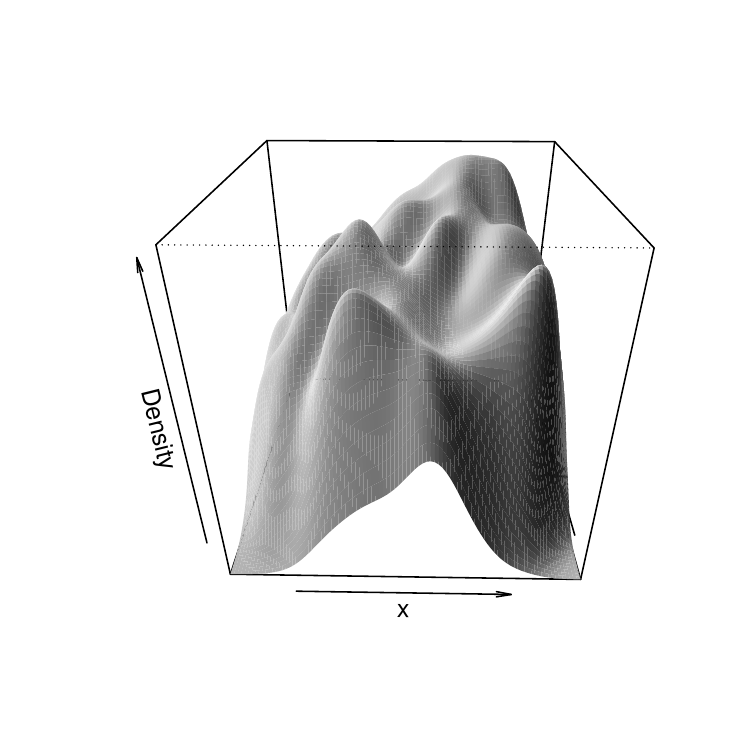}} 
	{\includegraphics[width=0.14\textwidth]{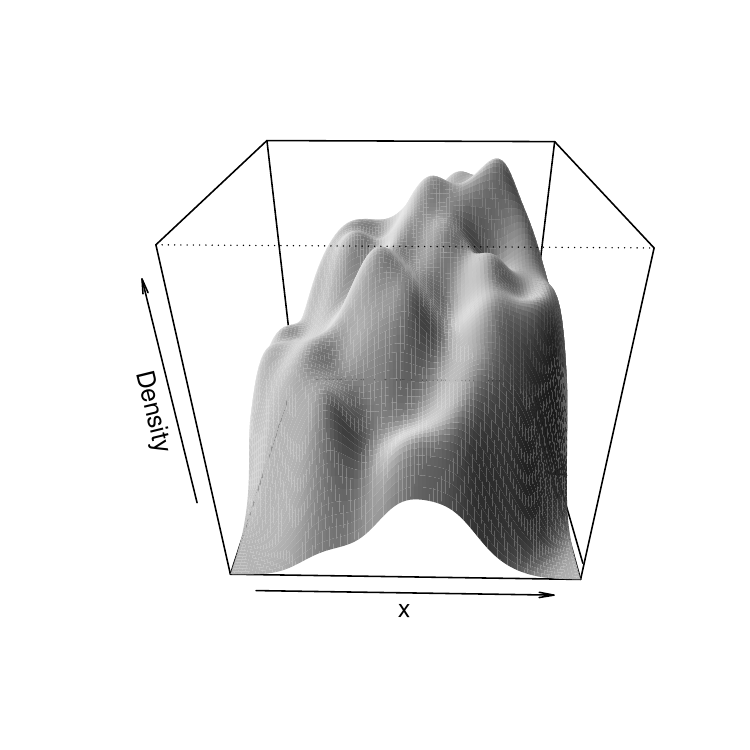}} 
	{\includegraphics[width=0.14\textwidth]{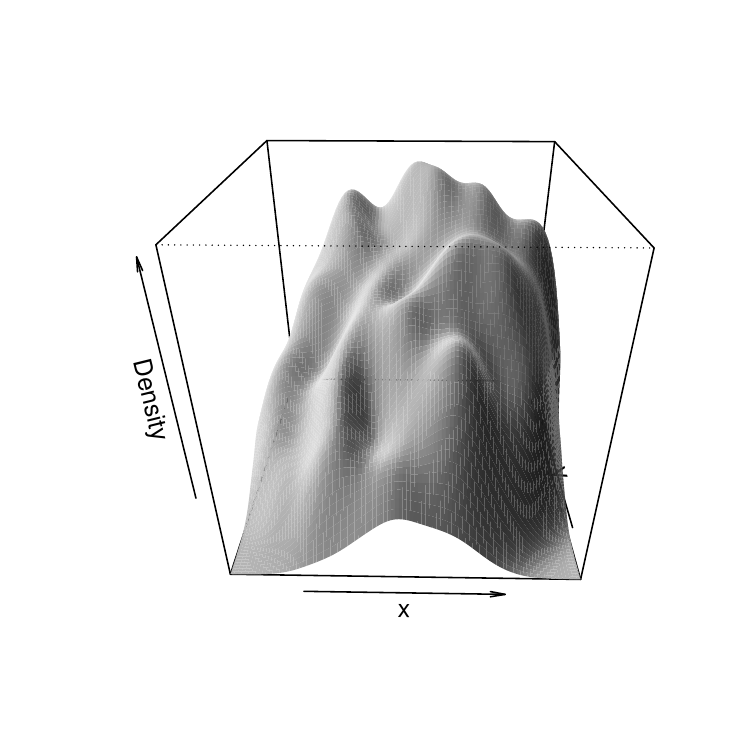}} 
	{\includegraphics[width=0.14\textwidth]{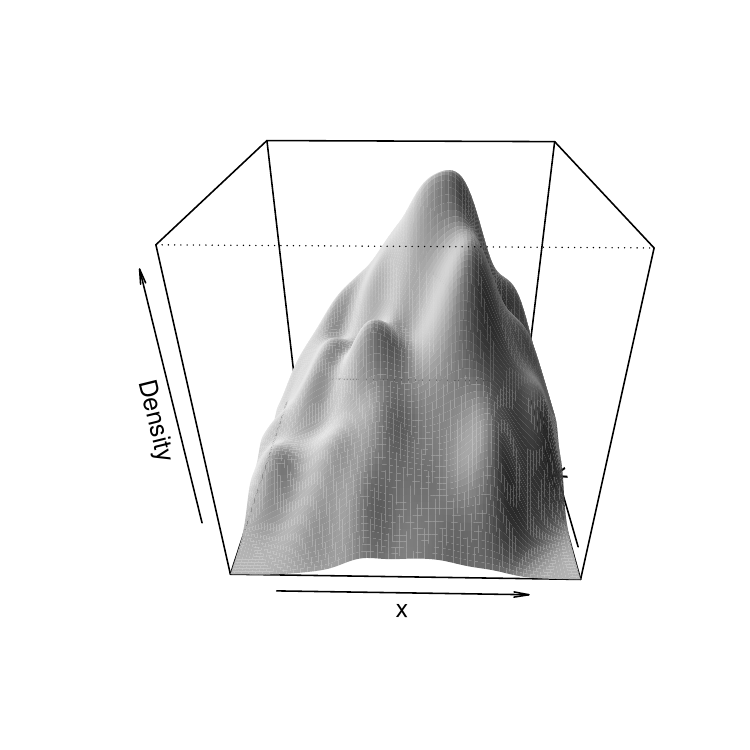}} 
	{\includegraphics[width=0.14\textwidth]{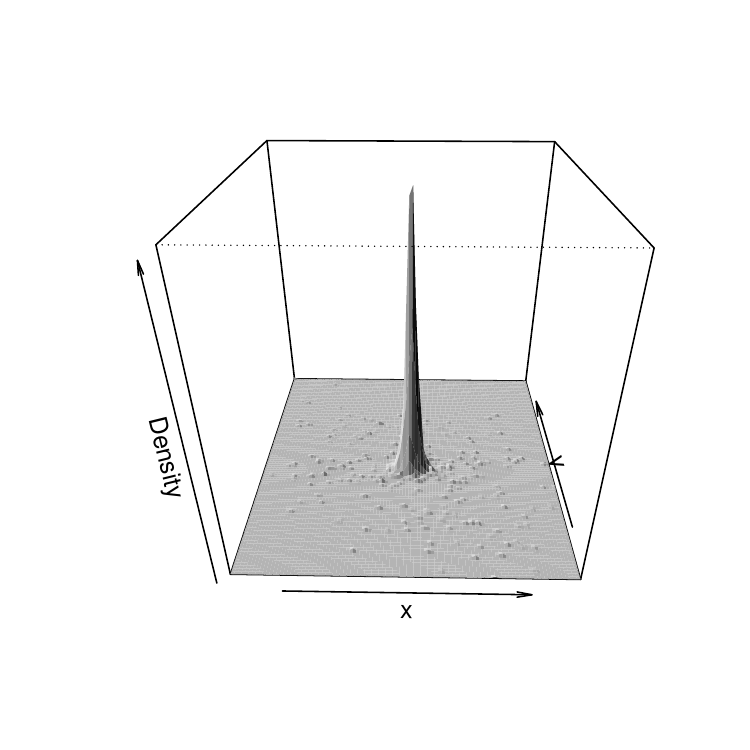}} 
	{\includegraphics[width=0.14\textwidth]{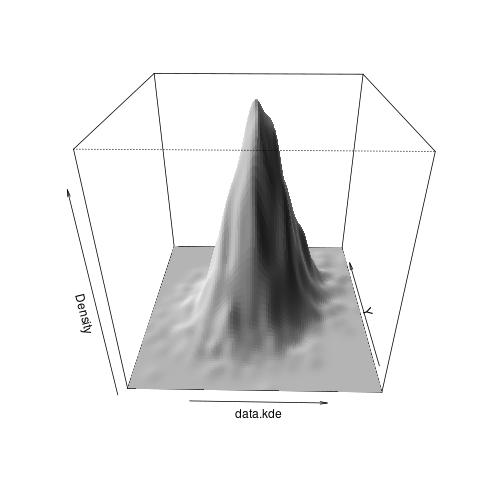}}\\
    \hspace{-4mm}
    {\includegraphics[width=0.14\textwidth]{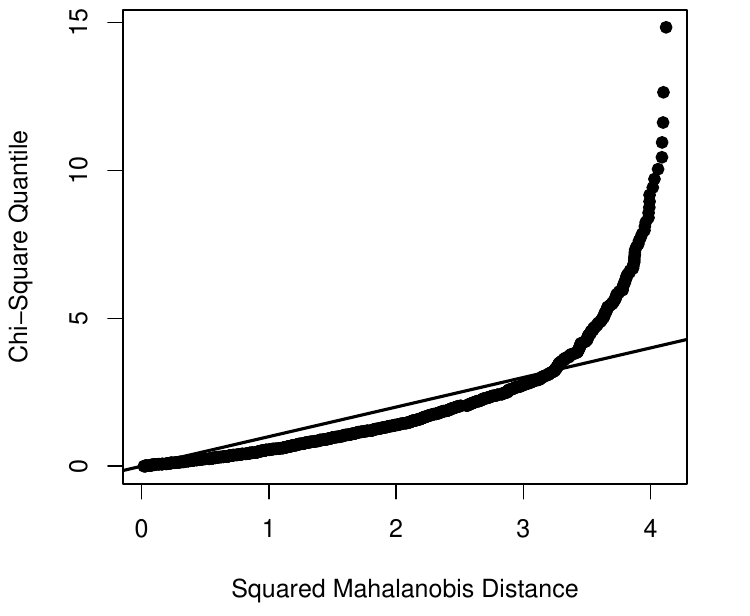}} \hspace{-2mm}
	{\includegraphics[width=0.14\textwidth]{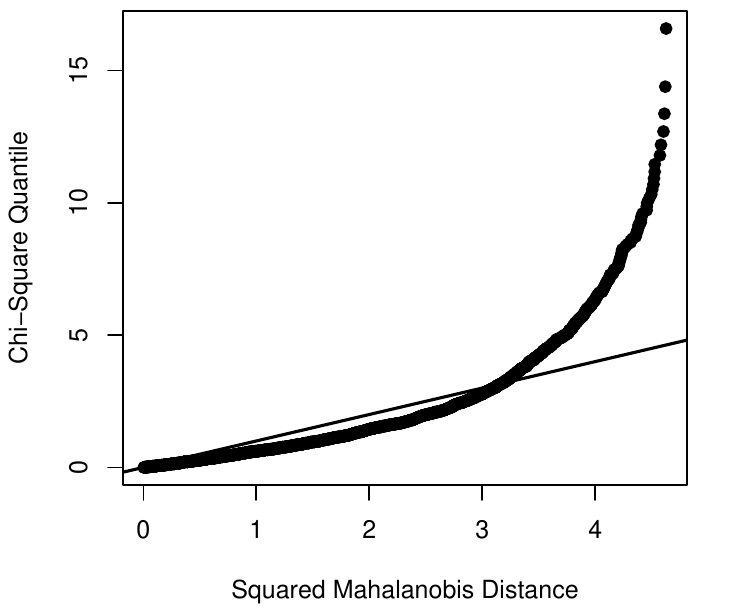}} \hspace{-2mm}
	{\includegraphics[width=0.14\textwidth]{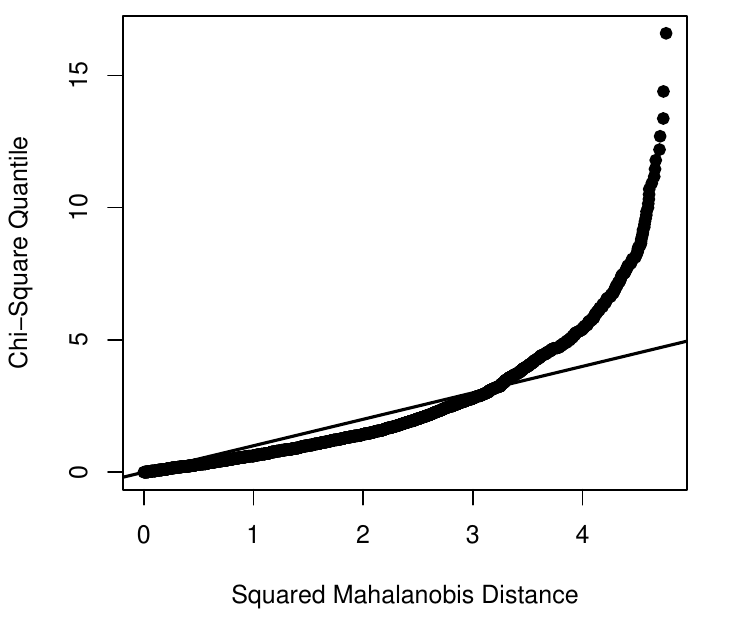}} \hspace{-2mm}
	{\includegraphics[width=0.14\textwidth]{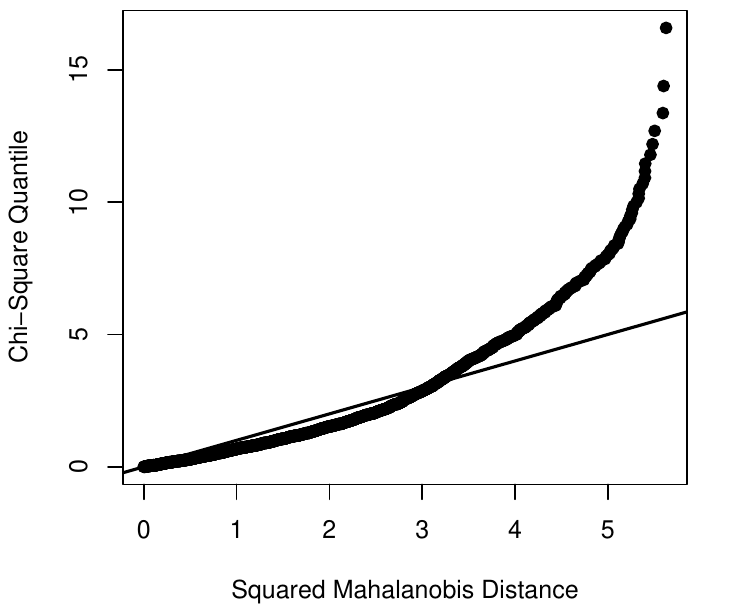}} \hspace{-2mm}
	{\includegraphics[width=0.14\textwidth]{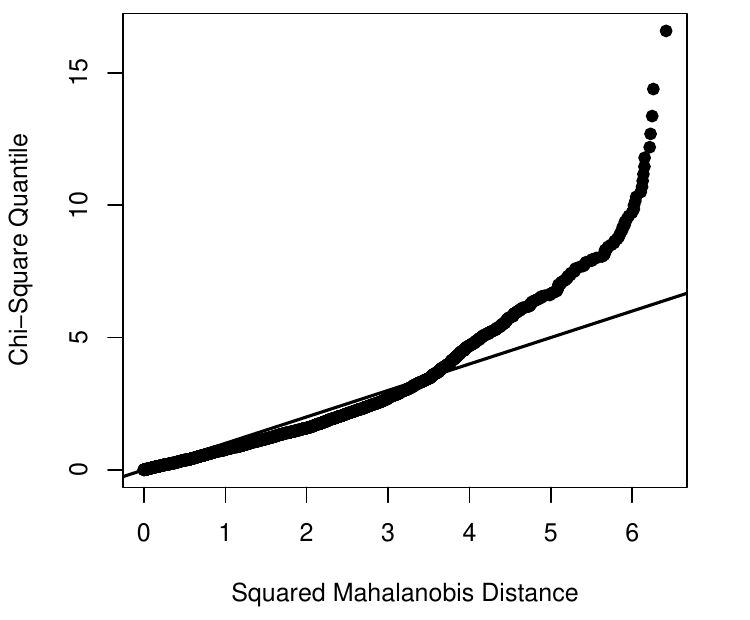}} \hspace{-2mm}
	{\includegraphics[width=0.14\textwidth]{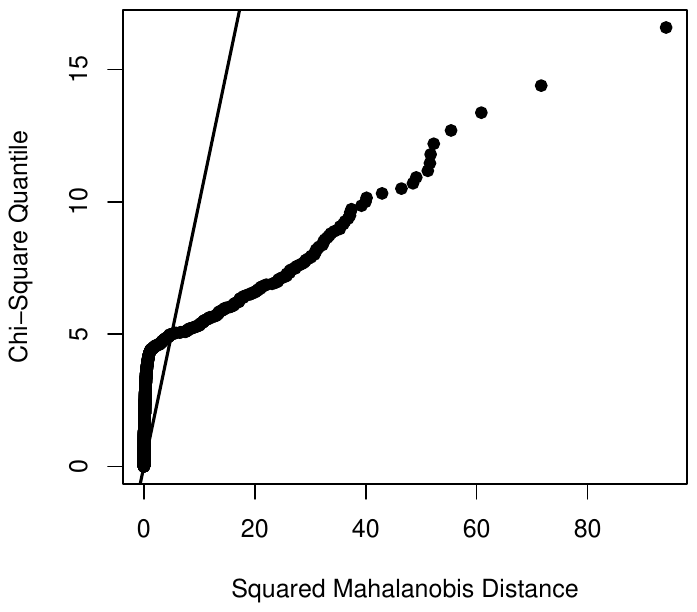}} \hspace{-2mm}
	{\includegraphics[width=0.14\textwidth]{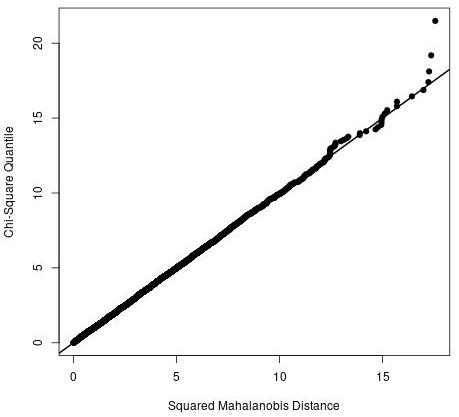}} 

	\caption{ Perspective plots (top) and Chi-square Q-Q plots (bottom) of the noise vectors generated by neural network measurements in the \lmo dataset. From left to right, they are respectively for $1\%$, $5\%$, $10\%$, $20\%$, $40\%$, $100\%$ noise vectors. The right-most graph shows the perspective plot for a Gaussian distribution that can be used as comparison.\label{fig:mvn-test-lmo-perspective-plot}}
\end{figure}


\section{SOS Relaxations and the Moment-SOS Hierarchy}
\label{app:sec:moment-sos-pop}

\subsection{Monomials, Polynomials, and Sum-of-Squares(SOS)}
\label{app:sec:moment-sos-pop:monomials-sos}

{\bf Monomials and Polynomials}.
Given $x = [x_1,x_2,\cdots,x_n]\in \mathbb{R}^n$, denote the monomial $\Pi_{i}x_i^{\alpha_i}$ as $x^\alpha$ where $\alpha = [\alpha_1,\cdots, \alpha_n]\in \mathbb{N}^n$. The degree of a monomial is defined as $\sum_{i = 1}^{n} \alpha_i$, and the degree of a polynomial, $\sum_{\alpha} c_\alpha x^\alpha$, is the maximum degree of all monomials. We use the notation $\mathbb{R}[x]_d$ to denote all the polynomials with degree less than or equal to $d$. 

{\bf Monomial Basis and Sum-of-Squares Polynomials}.
Denote the monomial basis $[x]_\kappa$ as the column vector of all monomials of degree up to $\kappa$. For example, if $n=2$ and $\kappa=2$, the corresponding monomial basis is $[x]_2 = [1;x_1;x_2;x_1^2;x_1x_2;x_2^2]$. If we define $s_n(d) :=\left( \substack{n+d \\ d} \right)$, the length of the monomial basis $[x]_\kappa$ is $s_n(\kappa)$.

A polynomial $g(x)$ is called a sum-of-squares (SOS) polynomial if it can be written as a sum of squares of polynomials, \ie $g(x) = \sum_{i = 1}^{m} q_i(x)^2$ for some $q_i(x)\in \mathbb{R}[x]$. Next proposition gives a necessary and sufficient condition for a polynomial to be SOS, which is more convenient for us to check whether a polynomial is a SOS polynomial.

\begin{proposition}[Condition for SOS Polynomials~\cite{lasserre2009moments}]\label{prop:sos-polynomial}
    A polynomial $g\in\mathbb{R}_{2d}[x]$ is a SOS polynomial if and only if there exists a symmetric and positive semidefinite matrix $Q\in\mathbb{R}^{s_n(d)\times s_n(d)}$ such that $g(x) = [x]_d\tran Q [x]_d$ where $[x]_d$ is the monomial basis.
\end{proposition}

\subsection{Nonnegativity Certificate}
\label{app:sec:moment-sos-pop:nonnegativity}

 We consider a basic semialgebraic set  $\calS:= \{x \in \Real{n}\mid g_i(x) \geq 0, i = 1,\dots,l_g,h_j(x)= 0,j = 1,\dots,l_h\}$, where $g_i(x),h_j(x)$ are polynomials. How to describe the nonnegativity of a polynomial on $\calS$ is a crucial technique in the following proofs. We first introduce a property called \emph{Archimedean}.

\begin{assumption}[Archimedean Property]\label{assumption:compact-level-set}
    For the basic semialgebraic set $\calS$, there exists $u\in \mathbb{R}[x]$ of the form:
    $$u = \sigma_0 + \sum_{j = 1}^{l_g} \sigma_j g_j + \sum_{k = 1}^{l_h} \lambda_k h_k$$
    where $\sigma_j \in \mathbb{R}[x]$ are sum of squares and $\lambda_k \in \poly{x}$ are arbitrary polynomials, such that the level set $\{x \in\mathbb{R}^n\mid  u(x) \geq 0\}$ is compact.
\end{assumption}

It's notable that, if the set is compact, then by adding an enclosing ball constraint, the Archimedean property can be satisfied trivially. If $\calS$ is Archimedean, then for strictly nonnegative polynomials on $\calS$, we have the following theorem.

\begin{theorem}[Putinar's Positivstellensatz~\citep{putinar1993positive}]\label{thm:putinar}
    Let $\set{S}$ be a basic semialgebraic set defined by $\{x\mid g_i(x)\geq 0, i = 1,...,l_g, h_j(x) = 0, j = 1,...,l_h \}$, and $f$ be a polynomial in $\mathbb{R}[x]$. If assumption 
    \ref{assumption:compact-level-set} holds, then $f$ is strictly nonnegative on $\set{S}$ if and only if there exists a sequence of sum-of-squares polynomials $\{\sigma_i\}_{i=0}^{l_g}$ and polynomials $\{\lambda_j\}_{j = 1}^{l_h}$, such that $f = \sigma_0 + \sum_{i=1}^{l_g} \sigma_ig_i + \sum_{j=1}^{l_h} \lambda_jh_j$.
\end{theorem}

There exist other types of nonnegativity certificates, for which we refer the interested reader to~\cite{blekherman12siam-semidefinite}.

\subsection{Polynomial Optimization and the Moment-SOS Hierarchy}
\label{app:sec:moment-sos-pop:pop-lasserre}

We introduce basic concepts of the moment-SOS hierarchy which are used in the derivation in the main text, specifically Theorem~\ref{thm:grcc}. For further information, we refer to \cite{lasserre2009moments,blekherman12siam-semidefinite}.

{\bf Polynomial optimization}. Consider the following polynomial optimization problem (POP):
\begin{subequations}\label{prob:POP}
    \bea
\min_x &  f(x) \\
\subject & g_i(x) \geq 0 \quad i = 1\cdots l_g \\
        & h_j(x)  = 0 \quad j = 1\cdots l_h
\eea
\end{subequations}
where $g_i(x),h_j(x)$ are polynomials. Denote the basic semialgebraic set (feasible set of \eqref{prob:POP}) $\calS:= \{x \in \Real{n}\mid g_i(x) \geq 0, i = 1,\dots,l_g,h_j(x)= 0,j = 1,\dots,l_h\}$. Generally speaking, \eqref{prob:POP} is a nonconvex problem, and it is NP-hard to solve. However, we can relax problem \eqref{prob:POP} into a sequence of convex semidefinite programming (SDP) problems, which guarantees to monotonely converge to the optimal value of \eqref{prob:POP}. The relaxation is called Lasserre's moment-SOS hierarchy~\citep{lasserre2001global}. Next we will introduce this technique briefly.

First, we can recast the polynomial optimization problem as a infinite dimensional linear programming problem. 
\begin{subequations}\label{prob:POP-measure}
    \bea
    \min_{\mu\in \mathcal{M}(\calS)_+}& \int f(x) d\mu \\
    \subject& \int_\calS d\mu = 1
\eea 
\end{subequations}
where $\mathcal{M}(\calS)_+$ is the space of all finite Borel measures supported on $\calS$. Then we introduce a linear functional $L_z(f)$
$$
f(x) = \sum_{\alpha\in \mathbb{N}^n} f_\alpha x^\alpha \to L_z(f) = \sum_{\alpha \in \mathbb{N}^n} f_\alpha z_\alpha,
$$
where $\{z_\alpha,\alpha\in \mathbb{N}^n\}$ is sequence of real numbers (often known as the pseudomoment vector). For a sequence $\{z_\alpha,\alpha\in \mathbb{N}^n\}$, we say it has a representing measure on $\calS$ if there exists a measure $\mu$ supported on $\calS$ satisfying: $z_\alpha = \int_\calS x^\alpha d\mu$.
Then the optimization problem~\eqref{prob:POP-measure} is equivalent to:
\begin{subequations}\label{prob:POP-measure-z}
    \bea
    \min_z& L_z(f) \\
    \subject& z_0 = 1 \\
            & z \text{ has a representing measure on } \calS
\eea
\end{subequations}
If we introduce a new notation of moment matrix and localizing matrix, we can cast a necessary and sufficient condition for the existence of representing measure as infinite semidefinite constraints.

{\bf Moment matrix}.
    Given a $s_n(2r)$-sequence $z = (z_\alpha)$, the moment matrix $M_r(z)$ is defined as follows:
    $$M_r(z)(\alpha,\beta) = L_z(x^{\alpha+\beta}) = z_{\alpha+\beta}\quad \forall \alpha,\beta\in \mathbb{N}^n_r$$
    where $\mathbb{N}^n_r$ is the set of all multi-indices $\alpha\in \mathbb{N}^n$ with $|\alpha| = \sum_{i=1}^{n}\alpha_i \leq r$.

{\bf Localizing matrix}.
Given a polynomial $g(x)\in \mathbb{R}[x]$, the localizing matrix with respect to $z,g$ is defined as:
$$
M_r(g,z)(\alpha,\beta) = L_z(x^{\alpha+\beta}g(x)) = \sum_{\gamma\in \mathbb{N}^n} z_{\alpha+\beta+\gamma}g_\gamma.
$$


If the set satisfies the Archimedean property, we have the following theorem.

\begin{theorem}[Representing measure {\cite[Theorem 3.8]{lasserre2009moments}} ]
    Let $z = (z_\alpha)_{\alpha\in\mathbb{N}^n}$ be a given infinite sequence in $\mathbb{R}$ and assume the basic semialgebraic set $\calS$ is compact. Then $z$ has a representing measure on $\calS$ if and only if for all $r\in\mathbb{N}$
    \begin{align*}
        M_r(z) &\succeq 0 \\
        M_r(g_iz)&\succeq 0, \quad i = 1\cdots l_g \\
        M_r(h_jz)& = 0, \quad j = 1\cdots l_h
    \end{align*}
\end{theorem}

This theorem, however, is not practical since it requires infinite number of constraints. In practice we solve the relaxation of the above theorem by only using constraints with a finite $r$.

Thus the final problem for Lasserre's hierarchy, with a relaxation order $\kappa$ is:
\begin{subequations}
    \bea
\min_z & L_z(f) \\
\subject& z_0 = 1\\
& M_\kappa(z) \succeq 0\\
& M_\kappa(g_iz)\succeq 0, \quad i = 1\cdots l_g \\
& M_\kappa(h_jz) = 0, \quad j = 1\cdots l_h
\eea
\end{subequations}
Using monomial basis $[x]_\kappa$, which is the column vector of all monomials of degree up to $\kappa$, the last part constraints can also be interpreted as:
$$
L_z(h_j [\theta]_{2\kappa-\deg{h_j}}) = 0, \quad j=1,\dots,l_h.
$$


\section{Proof of Theorem~\ref{thm:mee-sos}}
\label{app:sec:proof-mee-sos}
\begin{proof}
    First let's prove (i). We show that the following problem is equivalent to \eqref{eq:mee}
    \begin{subequations}\label{eq:mee-equiv}
        \bea 
        \max_{E,b,c} & \log\det E  \\
        \subject & 1 - (\xi\tran E \xi + 2 b\tran \xi + c) \geq 0, \ \forall \theta \in \calS \\
         & c - b\tran E^{-1}b = 0
        \eea
    \end{subequations}
    because with the last equality, we can see that the first inequality constraint is equivalent to:
    $$
    (\xi + E^{-1}b)\tran E(\xi + E^{-1}b) \leq 1.
    $$
    With $\mu := - E\inv b$, problem~\eqref{eq:mee-equiv} recovers \eqref{eq:mee}.

    Next, we only need to show the following problem is equivalent to problem~\eqref{eq:mee-equiv}.
    \begin{subequations} \label{prob: ellip-equivalent}
    \bea 
    \max_{E,b,c} & \log\det E  \\
    \subject & 1 - (\xi\tran E \xi + 2 b\tran \xi + c) \geq 0, \ \forall \theta \in \calS \label{cons: sos-enclosing-2}\\
     & \begin{bmatrix}
        E & b \\ b\tran & c 
     \end{bmatrix} \succeq 0 \label{cons:sos-lifting-2}
    \eea
    \end{subequations}
    The first constraint \eqref{cons: sos-enclosing-2} indicates enclosure. We can equivalently write it as:
    $$
    (\xi + E^{-1}b)\tran E(\xi + E^{-1}b) \leq 1 - (c - b\tran E^{-1}b), \ \forall \theta \in \calS
    $$
    By $E$ is positive definite, we can conclude that $c - b\tran E^{-1}b < 1$.
    The second constraint~\eqref{cons:sos-lifting-2} implies $c - b\tran E\inv b \geq 0$ by the Schur complement theorem. We will show that in fact we always have $c - b\tran E\inv b = 0$ at the optimal solution of \eqref{prob: ellip-equivalent}. 
    
    Consider the ellipsoid representation $\xi\tran E \xi + 2 b\tran \xi + c \leq 1$. Note that one ellipsoid can correspond to different pairs of $(E,b,c)$ due to 
    $$
    \xi\tran E \xi + 2 b\tran \xi + c \leq 1 \Leftrightarrow m(\xi\tran E \xi + 2 b\tran \xi + c ) \leq m, \forall m > 0.
    $$
    All the equivalent class can be written as: $(mE,mb,mc - m + 1))$ for any $m>0$.

    Suppose now we have a feasible solution $(E_0,b_0,c_0)$ which satisfies $c_0 - b_0\tran E_0^{-1}b_0 > 0$. We claim that this solution must not be optimal, because we can construct a better solution by maintaining the same ellipsoid, as follows. 
    
    For $m > 1$, consider $(E_1,b_1,c_1) = (mE_0,mb_0,mc_0 - (m-1))$. Clearly, if $(mE_0,mb_0,mc_0 - (m-1))$ is a feasible solution, then it will always be a better solution than $(E_0,b_0,c_0)$ due to $\log\det (mE_0) > \log\det (E_0)$ when $m >1$.

    Now we show that indeed we can make $(E_1,b_1,c_1)$ feasible for \eqref{prob: ellip-equivalent}. It suffices to construct $(E_1,b_1,c_1)$ feasible for \eqref{cons:sos-lifting-2}. To do so, we write 
    \begin{align*}
         c_1 - b_1\tran E_1^{-1}b_1 &= mc_0 - (m-1) - (m b_0)\tran (mE_0)^{-1}(mb_0) \\
         &= m(c_0 - b_0\tran E_0^{-1}b_0) - (m-1) \\
         &= (m-1)(c_0 - b_0\tran E_0^{-1}b_0 - 1) + c_0 - b_0\tran E_0^{-1}b_0.
    \end{align*}
    Note that $c_0 - b_0\tran E_0^{-1}b_0 > 0$ by our hypothesis, and $c_0 - b_0\tran E_0^{-1}b_0 - 1 < 0$ by enclosure. Therefore, we can always choose some $m-1>0$ that is sufficiently small such that $c_1 - b_1\tran E_1^{-1}b_1 \geq 0$ and hence $(E_1,b_1,c_1)$ can be made feasible. 
    
    Therefore, any optimal solution of problem \eqref{prob: ellip-equivalent} must satisfy \eqref{eq:mee-equiv}. 

    Then we leverage the sum-of-squares technique to relax the problem. By substituting constraints \eqref{cons: sos-enclosing-2} with \eqref{eq:mee-sos-equality},\eqref{eq:mee-sos-sos-multiplier} and \eqref{eq:mee-sos-poly-multiplier}. Because of the certificate of nonnegativity (\cf Section~\ref{app:sec:moment-sos-pop:nonnegativity}), we can always get an enclosing ellipsoid. Also from the derivation above, we shall always have $c - b\tran E^{-1}b = 0$. Therefore, the ellipsoid we get can be written as: $(\xi + E^{-1}b)\tran E(\xi + E^{-1}b) \leq 1$.

    Then we prove (ii). First, it's easy to see that when $\kappa$ increases, the feasible region of the problem~\eqref{eq:mee-sos} gets larger, so the optimal value of the optimization problem will get larger, which means that Vol($\calE$) decreases. 
    
    Suppose the optimal solution of the minimum enclosing ellipsoid problem \eqref{eq:mee} is: $\calE^\star\coloneqq (\xi-\mu^\star)\tran E^\star (\xi-\mu^\star) \leq 1$. Then for any sufficiently small $\varepsilon > 0$ making sure that $E^\star - \varepsilon \eye_n\succeq 0$, we can have $1 - (\xi-\mu^\star)\tran (E^\star - \varepsilon \eye_n)(\xi-\mu^\star)$ strictly positive on $\calS$. Then leveraging the Putinar's Positivstellensatz (\cf Section~\ref{app:sec:moment-sos-pop:nonnegativity}), we know that there exists sum of squares polynomials $\sigma_i(\theta)$ satisfies:
    $$
    1 - ((\xi-\mu^\star)\tran (E^\star - \varepsilon \eye_n)(\xi-\mu^\star)) =  \sum_{i=0}^{l_g} \sigma_i(\theta) g_i(\theta) + \sum_{j=1}^{l_h} \lambda_j(\theta) h_j(\theta),
    $$ 
    which shows that the point $(E^\star - \varepsilon \eye_n, \mu^\star)$ is a feasible point of the problem \eqref{eq:mee-sos} for some $\kappa$. 

    From the continuity property of the determinant function, we can know that $ \det(A - \varepsilon \eye_n) = \det(A) + O(\varepsilon)$. By letting $\varepsilon\to 0$, we can show that the optimal value of the problem \eqref{eq:mee-sos} converges to the optimal value of the problem \eqref{eq:mee} when $\kappa \to \infty$.
\end{proof}

\section{Certificate of MEE}
\label{app:sec:proof-certificate-mee}

We present a result to check the convergence of the SOS-based MEE algorithm in Theorem~\ref{thm:mee-sos} when the set $\calS$ is convex.

\begin{proposition}[Certificate of MEE]
    \label{prop:certificate-mee}
    Let $\calS_\xi$ be convex, and $(E_\star,b_\star,c_\star)$ be an optimizer of \eqref{eq:mee-sos} with certain $\kappa$. Consider the following (nonconvex) polynomial optimization problem 
    \bea\label{eq:pop-certificate}
    \min_{\xi \in \calS_\xi} & 1 - (\xi\tran E_\star \xi + 2 b_\star\tran  \xi + c_\star).
    \eea 
    Denote the optimal value of the $\nu$-th order moment-SOS relaxation of \eqref{eq:pop-certificate} as $\rho^\star_\nu$ (\cf Section~\ref{app:sec:moment-sos-pop:pop-lasserre}). Suppose at a certain relaxation order $\nu$, we have (i) $\rho^\star_\nu = 0$, (ii) a finite number of global minimizers of \eqref{eq:pop-certificate} can be extracted as $\{ \xi_{\star}^1,\dots,\xi_\star^K \}$, and (iii) there exist positive $(\alpha_1,\dots,\alpha_K)$ such that 
    \bea \label{eq:contact-points}
    \sum_{k=1}^K \alpha_k ( \xi_{\star}^k + E_\star^{-1}b_\star) = 0, \quad \sum_{k=1}^K \alpha_k ( \xi_{\star}^k + E_\star^{-1}b_\star)( \xi_{\star}^k + E_\star^{-1}b_\star)\tran = E_\star^{-1},
    \eea 
    then $\calE = \{ \xi \in \Real{d} \mid \xi\tran E_\star \xi + 2 b_\star\tran \xi + c_\star \leq 1 \}$ is the unique solution of \eqref{eq:mee}.
\end{proposition}

Before presenting the proof of Proposition~\ref{prop:certificate-mee}, we briefly describe the intuition. By~\cite[Theorem 3.2.5]{xie2016thesis}, $\calE$ is the minimum enclosing ellipsoid of $\calS_\xi$ if and only if there exists a finite number of \emph{contact points} at the intersection of $\calE$ and $\calS_\xi$ that satisfy \eqref{eq:contact-points}. Since the contact points lie on the boundary of $\calE$, they must attain zero value in \eqref{eq:pop-certificate}, and hence must be its global minimizers. Therefore, we can use the moment-SOS hierarchy to find the contact points. Section~\ref{app:sec:exp:toy-examples} gives a concrete numerical example of computing the certificate of convergence.

\begin{proof}
We first present John's theorem of minimum enclosing ellipsoid.

\begin{lemma}[John's Theorem]\label{lem:johnstheorem}
    Let $C\subseteq\Real{n}$ be a convex body and $C\subseteq B_H(\bar{x},1)$, where $B_H(\bar{x},1)\coloneqq \{x\in \Real{n}| (x-\bar{x})\tran H (x-\bar{x})\leq 1\},H \succ 0$. Then the following statements are equivalent:
    \begin{itemize}
        \item $B_H(\bar{x},1)$ is the minimum volume ellipsoid enclosing $C$.
        \item There exists some $r\in\left\{n,n+1,\cdots,\frac{n(n+3)}{2}\right\}$ contact points $p_1,\cdots,p_r\in \partial(C)\cap \partial(B_H(\bar{x},1))$ and $\alpha\in \mathbb{R}^r_{++}$ such that
        \bea \label{eq:johnstheorem-contact-points}
        H^{-1} = \sum_{i=1}^{r} \alpha_i(p_i-\bar{x})(p_i-\bar{x})\tran \text{ and } \sum_{i=1}^{r}\alpha_i(p_i - \bar{x}) = 0.
        \eea
    \end{itemize}
\end{lemma}

Then we leverage this Lemma to prove Proposition~\ref{prop:certificate-mee}.

    From Theorem~\ref{thm:mee-sos}, the bounding ellipsoid of $\calS_\xi$ is $\calE = (\xi + E_\star^{-1}b_\star)\tran E_\star(\xi + E_\star^{-1}b_\star) \leq 1$. Based on Lemma~\ref{lem:johnstheorem}, if the ellipsoid is the minimum volume enclosing ellipsoid, there must exists contact points satisfying~\eqref{eq:johnstheorem-contact-points}. Then the optimal value of the optimization problem~\eqref{eq:pop-certificate} is $0$ and the solutions are the contact points. 
    
    Suppose by solving the optimization problem~\eqref{eq:pop-certificate} at certain relaxation order $\nu$, we get the optimal value $\rho^\star_\nu = 0$, then the solutions are the contact points. Suppose we can extract the optimal solutions $\xi_{\star}^1,\dots,\xi_\star^K \in \partial(\calS_\xi)\cap \partial(\calE)$ using the result in~\cite{henrion05-detecting},
    Then from John's theorem, if there exist $\alpha\in \mathbb{R}^K_{++}$ such that,
    \bea\label{eq:johnstheorem-contact-points-restate}
    \sum_{k=1}^K \alpha_k (\xi_{\star}^k + E_\star^{-1}b_\star) = 0, \quad \sum_{k=1}^K \alpha_k (\xi_{\star}^k + E_\star^{-1}b_\star)(\xi_{\star}^k + E_\star^{-1}b_\star)\tran = E_\star^{-1}
    \eea
    the ellipsoid $\calE$ is the minimum volume enclosing ellipsoid. Note that \eqref{eq:johnstheorem-contact-points-restate} is a system of linear equations in $\alpha$ and can be easily solved.
\end{proof}

\section{Proof of Theorem~\ref{thm:grcc}}
\label{app:sec:proof-grcc}


We first show that there exists $R > 0$ such that the optimization problem~\eqref{eq:gcc} is equivalent to the following problem
\begin{equation}\label{eq:equ-problem}
    \eta^\star =  \min_{\mu \in B_R(O)} \max_{\xi \in \calS_\xi} \quad \Vert \mu - \xi  \Vert_Q^2,
\end{equation} 
where $B_R(O)$ denotes a ball of radius $R$ centered at $O$.

\begin{theorem}[Equivalence of the problem]\label{thm:equivalence_problem}
    There exists $R>0$ such that the original problem
    \begin{equation*}
        \eta^\star =  \min_{\mu \in \mathbb{R}^{d}} \max_{\xi \in \calS_\xi} \quad \Vert \mu - \xi  \Vert_Q^2
    \end{equation*} 
    is equivalent to the problem:
    \begin{equation*}
        \eta^\star =  \min_{\mu \in B_R(O)} \max_{\xi \in \calS_\xi} \quad \Vert \mu - \xi  \Vert_Q^2
    \end{equation*} 
\end{theorem}
    
\begin{proof}
    By \cite{sain2016chebyshev}, the Chebyshev center of the bounded set in $\mathbb{R}^n$ exists and is unique. We only need to show that the solution lies inside some ball $B_R(O)$. 
    
    Denote the Chebyshev center as $\mu^\star$ and the corresponding radius as $\eta^\star$. Due to the assumption of boundedness, there exists $R'$ such that for all $\xi \in \calS_\xi$, 
    $$
    \|\mu^* - O\|\leq \|\mu^* - s\| + \|s - O\|\leq \eta^\star + R',
    $$
    where $s$ is the point in $\calS_\xi$ that attains the maximum distance to $\mu^\star$, and $R'$ is the distance from $s$ to $O$. Thus we can choose $R = R'+\eta^\star$ to complete the proof.
\end{proof}

Then using Lasserre's hierarchy, for relaxtation order $\kappa$ satisfying $2\kappa \geq \max\{ \deg{g_1},\dots,\deg {g_{l}},2 \}$, we can relax the inner maximization problem
\bea
\label{eq:innermax}
\eta (\mu) = \max_{\theta \in \calS} & \mu\tran Q \mu - 2 \mu\tran Q P\theta + \theta\tran P\tran Q P\theta
\eea
Denote the pseudomoment vector of degree up to $2\kappa$ as $[z]_{2\kappa}$ and shorthand it as $z$, and we denote the moment matrix with respect to $z$ of order $\kappa$ as $M_\kappa(z)$, then the moment relaxation of \eqref{eq:innermax} at order $\kappa$ reads:
\bea \label{eq:innermaxrelaxted}
\eta_\kappa(\mu) = \max_z &  \mu\tran Q \mu - 2 \mu\tran QP L_z(\theta) + \inprod{C}{M_\kappa(z)} \\
\subject & M_\kappa(z) \succeq 0, \quad M_0(z) = 1 \\
& M_{\kappa - \lceil \deg{g_i}/2 \rceil}(g_i z) \succeq 0, \quad i = 1,\dots,l_g \\
& L_z(h_j [\theta]_{2\kappa-\deg{h_j}}) = 0, \quad j=1,\dots,l_h
\eea 
where $L_z(\theta)$ denotes the order-one moment vector in $z$ (\ie the sub-vector of $z$ corresponding to the order-one monomials $\theta$), $ M_{\kappa - \lceil \deg{g_i}/2 \rceil}(g_i z)$ are the localizing matrices corresponding to inequalities $\{ g_i \}_{i=1}^{l}$, $M_0(z) = 1$ enforces the zero-order pseudomoment to be $1$, and the matrix $C$ is 
$$
C = \bmat{ccc}
0 & \zero & \zero \\
\zero & A & \zero \\
\zero & \zero & \zero
\emat
$$
where $A$ satisfies
$$ A_{ij}=\left\{
\begin{aligned}
    & (P\tran QP)_{ij}\quad \text{if}\quad i=j \\
    & (P\tran QP)_{ij}/2\quad \text{else}
\end{aligned}
\right.
$$
Thus, $\inprod{C}{M_\kappa(z)}$ corresponds to the moment relaxation of $ \theta\tran P\tran Q P\theta $.

Denote $\Omega_\kappa$ as the feasible set of the relaxation \eqref{eq:innermaxrelaxted}. Note that $\Omega_\kappa$ does not depend on $\mu$.


Now that we have relaxed the inner ``$\max$'' as a convex problem, we bring back the outer ``$\min$''
\bea \label{eq:minmaxrelaxed}
\eta_\kappa^\star = \min_{\mu\in B_R(O)} \left\{ \eta_\kappa(\mu) = \max_{z \in \Omega_\kappa} \left\{ \mu\tran Q \mu - 2 \mu\tran QP L_z(\theta) + \inprod{C}{M_\kappa(z)} \right\} \right\}
\eea 
from Lasserre's hierarchy, we have $\eta_\kappa (\mu) \geq \eta(\mu)$ for any $\mu$ and any relaxation order $\kappa$, thus we have 
$$
\eta_\kappa^\star = \min_{\mu\in B_R(O)} \eta_\kappa(\mu) \geq \min_{\mu\in B_R(O)} \eta (\mu) = \eta^\star
$$
\ie the optimum of \eqref{eq:minmaxrelaxed} always provides an upper bound for $\eta^\star$.

Notice that now the objective function of the relaxed problem \eqref{eq:minmaxrelaxed} is affine in $z$ and convex in $\mu$. And note the $B_R(O)$ is a compact set. So that we can apply Sion's minimax theorem to switch the order of ``$\min$'' and ``$\max$''.

\begin{theorem}[Sion's minimax theorem~{\cite[Corollary 3.3]{sion1958general}}]
    Let $M,N$ be convex spaces one of which is compact, and $f$ a function on $M\times N$, quasi-concave-convex and upper semicontinuous - lower semicontinuous. Then $\inf_{x\in N}\sup_{y\in M}f(x,y) = \sup_{y\in M}\inf_{x\in N}f(x,y)$.
\end{theorem}

Then we get the equivalent problem:
\bea \label{eq:minmaxrelaxed-switch}
\eta_\kappa^\star = \max_{z \in \Omega_\kappa} \min_{\mu\in B_R(O)} & \mu\tran Q \mu - 2 \mu\tran QP L_z(\theta) + \inprod{C}{M_\kappa(z)}
\eea
Note that by taking the derivative of the objective function with respect to $\mu$, now the inner minimization problem has an explicit solution:
$$
\mu^\star_\kappa = PL_z(\theta)
$$

Next we will show that: we can choose the ${R_1}$ big enough such that for all $\kappa$, $\mu_\kappa^\star \in B_{R_1}(O)$.

\begin{theorem}[Attainability of the minimum]\label{thm:attain_min}
There exists ${R_1}>0$ which doesn't rely on the relaxation order $\kappa$, such that $\mu_\kappa^\star \in B_{R_1}(O)$.

\begin{proof}
    $P$ does not rely on $\kappa$, thus it suffices to show that $L_z(\theta)$ is bounded independent of $\kappa$.

    Due to the assumption of compactness of $\set{S}$, there exists ${R_1}$ such that $\set{S} \subset B_{R_1}(O)$, which is independent from the optimization procedure. Thus, we can add an extra inequality constraint $g_{l+1}(\theta) = -\theta\tran \theta +{R_1}^2 \geq 0$ without affecting the original problem. So without loss of generality, we can assume that we include the constraint $g_{l+1}(\theta) = -\theta\tran \theta +{R_1}^2 \geq 0$ in the original problem.

    Then we examine the constraint of the localizing matrix. The corresponding constraint for the inequality $g_{l+1}(\theta) = -\theta\tran \theta +{R_1}^2 \geq 0$ is $M_{\kappa-1}(g_{l+1}y)\succeq 0$. The principal minors of a positive semidefinite matrix is nonnegative, so the first principal minor of $M_{\kappa-1}(g_{l+1}y)$ is nonnegative, which implies $L_z(g_{l + 1})\geq 0$. And this is equivalent to the constraint that $\inprod{T}{M_\kappa(z)} \leq {R_1}^2$ where
    $$
    T = \bmat{ccc}
    0 & \zero & \zero \\
    \zero & \eye_n & \zero \\
    \zero & \zero & \zero
    \emat
    $$

    Then let's pass this constraint to $L_z(\theta)$ via the moment matrix. Denote the $(n+1)\times (n+1)$ principal submatrix as:
    $$
    \bmat{cc}
    1 & L_z(\theta)\tran \\
    L_z(\theta) & A
    \emat
    $$ where $\trace{A} = \inprod{T}{M_\kappa(z)} \leq {R_1}^2$.

    The moment matrix $M_\kappa(z)\succeq 0$, so that the $(n+1)\times (n+1)$ principal submatrix is also positive semidefinite. By Schur complement, we know that $A - L_z(\theta) L_z(\theta)\tran \succeq 0$. Thus ${R_1}^2\geq \trace{A}\geq \trace{L_z(\theta)L_z(\theta)\tran} = L_z(\theta)\tran L_z(\theta)$. This implies $L_z(\theta)$ lies inside $B_{R_1}(O)$, which is independent of the choice of $\kappa$.
\end{proof}
\end{theorem}

Thus we can know that we can solve the inner minimization exactly. Plugging the $\mu_\kappa^\star$ back into the objective function, we can get another equivalent problem:
\bea \label{eq:minmaxrelaxed-switch-singlelevel}
\eta_\kappa^\star = \max_{z \in \Omega_\kappa} &  - L_z(\theta)\tran P\tran QP L_z(\theta) + \inprod{C}{M_\kappa(z)}
\eea
or equivalently in minimization form 
\bea \label{eq:minmaxrelaxed-switch-singlelevel-min}
- \eta_\kappa^\star = \min_{z \in \Omega_\kappa} & L_z(\theta)\tran P\tran QP L_z(\theta) - \inprod{C}{M_\kappa(z)}
\eea
This is already a convex problem. But it is a convex quadratic SDP problem. By leveraging lifting technique, we can also convert it into a standard form SDP as follows:
\bea \label{eq:minmaxrelaxed-switch-singlelevel-min-SDP}
- \eta_\kappa^\star = \min_{z \in \Omega_\kappa} &  t - \inprod{C}{M_\kappa(z)} \\
\subject &\bmat{cc}
Q^{-1} & PL_z(\theta) \\
L_z(\theta)\tran P\tran & t
\emat \succeq 0
\eea

{\bf Proof of convergence}. Because the basic semialgebraic set $\calS$ is compact, we can always assume that the Assumption \ref{assumption:compact-level-set} holds. Then we will introduce two theorems which are the main ingredients in our proof of the convergence result.

\begin{theorem}[Degree bound of Putinar's Positivstellensatz ~{\cite{nie2007complexity}}] \label{thm: degree bound}
    Let $\set{K}$ be a compact basic semialgebraic set defined by:$\{x:g_i(x)\geq 0 \quad i = 1...m \}$ satisfying the Archimedean property. Assume $\emptyset \neq \set{K} \subset (-1,1)^n$. Then there is some $c>0$ such that for all $f\in \mathbb{R}[x]$ of degree $d$, and positive on $\set{K}$ (i.e. such that $f^*\coloneqq \min{f(x):x\in\set{K}})$), the representation in Putinar's Positivstellensatz \ref{thm:putinar} holds with $$\deg {f_jg_j} \leq cexp((d^2n^d\frac{\|f\|_0}{f^*})^c)$$
    where $\|f\|_0\coloneqq \max_\alpha \frac{f_\alpha}{C_{|\alpha|}^\alpha}$ and $C_{|\alpha|}^\alpha = \frac{|\alpha|!}{\alpha_1!...\alpha_n!}$.
\end{theorem}

We have the following theorem to guarantee the convergence of the Lasserre's hierarchy in this problem, which is the extension of the proof in \cite{lasserre2009moments}. Before the proof of our main theorem, we shall first show the following lemma which summarizes the relationship between optimization problems which we will leverage in the proof of the main theorem.
\begin{lemma}[Relationship between optimization problems ~{\cite{lasserre2009moments}}] \label{lemma: equivalent-prob}
    Given a basic semialgebraic set $\set{S}\coloneqq \{x\in\mathbb{R}^n|\ g_i(x) \geq 0\quad i = 1\cdots l_g\}$, and a polynomial $f(x)$, there are four related optimization problems:

    1. Moment problem: 
    \bea\label{prob:moment-prob}
        \eta_{\text{mom}} = \max_{x}& f(x)\\
        \subject & x\in \set{S}
    \eea

    2. Dual problem to the moment problem:
    \bea\label{prob:moment-dual_prob}
        \eta_{\text{pop}} =\min_{\lambda}& \lambda\\
        \subject & \lambda - f(x) \geq 0 \quad \forall x\in \calS
    \eea

    3. Primal relaxation problem

    (with relaxation order $\kappa$ satisfying $2\kappa\geq \max\{\deg {g_i}, \deg f\}$):
    \bea\label{prob:primal-relax-prob}
        \eta_\kappa = \max_{z}& L_z(f)\\
        \subject & z_0 = 1 \\
        & M_\kappa(z)\succeq 0 \\
        & M_{\kappa-\lceil \deg{g_i}/2 \rceil}(g_iz)\succeq 0 
    \eea

    4. Dual relaxation problem:

    (with relaxation order $\kappa$ satisfying $2\kappa\geq \max\{\deg {g_i}, \deg f\}$):
    \bea\label{prob:dual-relax-prob}
    \eta_\kappa^* = \min_{\lambda}& \lambda\\
    \subject & \lambda - f = \sigma_0 + \sum_{j=1}^{l_g} \sigma_jg_j   \\
    & \sigma_j\text{ is sum of squares polynomial} \\
    &  \deg {\sigma_j}\leq 2\kappa - \deg {g_j}, j = 1\cdots l_g 
    \eea
    We can conclude that :
    \bea
    \eta_{\text{mom}} = \eta_{\text{pop}} \leq \eta_\kappa \leq \eta_\kappa^* \quad \forall \kappa \ \text{satisfying} \  2\kappa\geq \max\{\deg {g_i}, \deg f\}
    \eea

\end{lemma}

We prove the Theorem~\ref{thm:grcc} by proving the following equivalent theorem:

\begin{theorem}
    Let $\set{S}$ be a compact basic semialgebraic set $\set{S}\coloneqq \{g_i(x) \geq 0, i = 1\cdots l_g, h_j(x) = 0, j = 1\cdots l_h\}$ which has non-empty interior, such that the Assumption \ref{assumption:compact-level-set} holds. Then the sequence of optimal values of the Lasserre's hierarchy \eqref{eq:minmaxrelaxed-switch-singlelevel-min-SDP} converges to the optimal value of the original problem when the relaxation order $\kappa\to\infty$. Plus, for any relaxation order $\kappa$ satisfying $2\kappa\geq \max\{\deg {g_i}, \deg {h_i}, 2\}$, we have $ \eta_\kappa^\star \geq \eta_{\kappa+1}^\star \geq \eta^\star$
\end{theorem}

\begin{proof}
    First we fixed $\mu$ in the outer minimization and only focus on the inner maximization problem. Here we denote $f(\theta) = \Vert \mu - P\theta  \Vert_Q^2$ as the objective function by holding $\mu$ as a constant. From lemma \ref{lemma: equivalent-prob}, we can know that when the relaxation order $\kappa$ satisfies $2\kappa \geq \max\{\deg {g_i}, \deg {h_i}, 2\}$, we have:
    $$
    \eta_{\text{mom}} = \eta_{\text{pop}} \leq \eta_i \leq \eta_i^*
    $$

    Using the definition of the dual problem \eqref{prob:moment-dual_prob}, if we fix an arbitrary $\varepsilon >0$, we can know that there exists a feasible point for problem \eqref{prob:moment-dual_prob} namely $\lambda$ which satisfies:
    $$
    \eta_{\text{mom}} = \eta_{\text{pop}} \leq \lambda \leq \eta_{\text{mom}} + \varepsilon = \eta_{\text{pop}} + \varepsilon
    $$

    Next, consider $\bar{\lambda} = \lambda + \varepsilon$. Thus we can know that $\bar{\lambda} - f \geq \varepsilon > 0 \ \forall x\in \set{S}$, then we can use Putinar's Positivstellensatz \ref{thm:putinar}:
    $$
    \bar{\lambda} - f = \sigma_0 + \sum_{i=1}^{l_g} \sigma_ig_i +\sum_{j = 1}^{l_h} \lambda_jh_j \quad \sigma_j\text{ is sum of squares polynomial}
    $$
    which shows that once $2\kappa \geq \max\{\deg {\sigma_jg_j},\deg {\lambda_jh_j}\}$, we can have
    $$
        \eta_{\text{pop}} = \eta_{\text{mom}}\leq \eta_\kappa\leq \eta_\kappa^* \leq \bar{\lambda} = \lambda + \varepsilon\leq \eta_{\text{mom}} + 2\varepsilon
    $$
    and as $\varepsilon$ is arbitrary, we can let $\varepsilon\to 0$.

    In the following discussion we will only focus on the inequality constraints, for all the equality constraints can be interpreted as two inequality constraints.

    The only barrier we will encounter when we want to prove the minimax version is that, the bound of degree of $\sigma_ig_i$ may not be the same for different $\mu$. So we need a uniform convergence result with respect to $\mu$ (the outer min). Thus we want to apply the degree bound of Putinar's Positivstellensatz. Remind that we have shown the equivalence of problem by restricting the feasible set of $\mu$ to a ball $B_{R_1}(O)$ for all $\kappa$. Also consider the ball of radius $R$ contains $\calS$. Then $\calS$ is also contained in a box $[-R,R]^n$. 

    Consider the polynomial $\tilde{f}(\tilde{\theta}) = \bar{\lambda} - f(\tilde{\theta}\cdot R)$, where $\tilde{\theta}\in \calS/R\subseteq[-1,1]^n$. Suppose $\lambda_{\text{max}}$ is the maximum eigenvalue of $Q$, then we apply Theorem \ref{thm: degree bound}. The parameters are as follows: $d = 2$, $n$ is the dimension of vector $\theta$. Note that: $\|f\|_0 = \max\{2\lambda_{\text{max}}R_1\tilde{\theta}(i),\lambda_{\text{max}}R^2\}$, which can be bounded by a constant independent of $\mu$. Also note that: $f^* = \min\{f(\theta):\theta\in\calS\} > \epsilon$, thus we can choose $\epsilon$ to bound $f^*$. Now we as long as $2\kappa \geq  cexp((d^2n^d\frac{\|f\|_0}{f^*})^c)$, we can guarantee all the $\sigma_ig_i$ will have degree less than $2\kappa$. Thus we can guarantee the inner maximization is exact, and the convergence of the minimax problem follows.
\end{proof}


\section{Proof of Theorem~\ref{thm:minimum-ball-SO3}}
\label{app:sec:proof-ball-so3}

\begin{proof} We first introduce some basics about unit quaternions.

{\bf Unit quaternion}. There are multiple equivalent representations of 3D rotations~\citep{barfoot17book-state}. The most popular choice is a 3D rotation matrix $R \in \SOthree$. Another popular choice is a \emph{unit quaternion}~\citep{yang19iccv-quaternion}. A unit quaternion is a 4-D vector with unit norm, \ie $q \in S^3 := \{ q \in \Real{4} \mid q\tran q = 1 \}$. Given a unit quaternion $q = (q_0,q_1,q_2,q_3)$, its corresponding rotation matrix reads
\bea \label{eq:quaternion-to-rotation}
\calR(q) = \bmat{ccc}
2(q_0^2+q_1^2)-1 & 2(q_1q_2-q_0q_3) & 2(q_1q_3+q_0q_2) \\
2(q_1q_2+q_0q_3) & 2(q_0^2+q_2^2)-1 & 2(q_2q_3-q_0q_1) \\
2(q_1q_3 - q_0q_2) & 2(q_2q_3 + q_0q_1) & 2(q_0^2+q_3^2) - 1
\emat.
\eea
Clearly, $q$ and $-q$ corresponds to the same rotation matrix. The inverse of a rotation matrix $R$ is $R\inv = R\tran$, and the inverse of a unit quaternion $q$ is $q\inv$, which simply flips the sign of $q_1,q_2,q_3$
\bea\label{eq:quaternion-inverse}
q\inv = [q_0,-q_1,-q_2,-q_3]\tran.
\eea
The product of two quaternions $q_c = q_a q_b$ is defined as 
\bea\label{eq:quaternion-product}
q_c = q_a q_b = \begin{bmatrix}
    q_{a,0} & -q_{a,1} & -q_{a,2} & -q_{a,3} \\
    q_{a,1} & q_{a,0} & -q_{a,3} & q_{a,2} \\
    q_{a,2} & q_{a,3} & q_{a,0} & -q_{a,1} \\
    q_{a,3} & -q_{a,2} & q_{a,1} & q_{a,0} 
\end{bmatrix}q_b.
\eea

The next proposition characterizes the equivalence of distance metrics on $\SOthree$ and $S^3$.

\begin{proposition}[Equivalent expression of distance on $\SOthree$ and $S^3$]
    Given two rotations $R,Q$, the geodesic distance is defiend as:
    $$
    \theta_1 = \arccos\left(\frac{\trace{R\tran Q} - 1}{2}\right).
    $$
    Let $q_R$ and $q_Q$ be the unit quaternions corresponding to $R$ and $Q$, the geodesic distance on $S^3$ is defined as 
    $$
    \theta_2 = 2\arccos(| \langle q_R, q_Q\rangle|),
    $$
    where $\langle \cdot,\cdot\rangle $ denotes the inner product viewing $q_R, q_Q$ as vectors in $\mathbb{R}^4$. It holds that $\theta_1 = \theta_2$.
\end{proposition}

\begin{proof}
    Consider the rotation $A = R\tran Q = R^{-1}Q$.
    The rotation angle of $A$ is 
    $$
    \theta_1 = \arccos\left(\frac{\trace{A} - 1}{2}\right).
    $$
    Let $q_A = q_R^{-1}q_Q$ ($q_R^{-1}$ is the inverse quaternion), by direct computation using \eqref{eq:quaternion-inverse} and \eqref{eq:quaternion-product}, we know 
    $$
    q_{A,0} = \langle q_R, q_Q\rangle.
    $$
    where $\langle \cdot,\cdot\rangle $ denotes the inner product viewing $q_R, q_Q$ as vectors in $\mathbb{R}^4$. For $A$, $\cos{\frac{\theta_2}{2}} = q_{A,0}$. By convention, we want $\theta_2 \in [0,\pi)$, so we need to carefully select the quaternions to guarantee the inner product is positive. For simplicity, we can just take the absolute value for $q_{A,0}$. Thus, the rotation angle of $q_A$ is:
    $$
    \theta_2 = 2\arccos(| \langle q_R, q_Q\rangle|).
    $$
    These two angle refers to the same rotation, thus $\theta_1 = \theta_2$, which completes the proof.
\end{proof}

The fact that for quaternions, $q$ and $-q$ represent the same rotation is the reason why we need to take the absolute value of the inner product. Geometrically, this is saying the same rotation $R$ corresponds to two quaternions on $S^3$ and these two quaternions are actually \emph{antipodal}. When considering the distance of two rotations on $S^3$, we shall always pick the nearest pair of points. 

Now think of the set $\calS_R$, it will get mapped to two antipodal sets on $S^3$. When the set $\calS_R$ is reasonablly small, we can get rid of the absolute value when calculating the minimum enclosing ball.



Precisely speaking, if there exists a rotation $\bar{R}$ such that $d_{\SOthree}(\bar{R},R)\leq \frac{\pi}{2}, \forall R\in\calS_R$, we can know that on $S^3$, there exists a quaternion $\bar{q}$ such that the maximum angle difference from $\bar{q}$ to the set is smaller than $\frac{\pi}{4}$. Thus by triangle inequality, in each branch on the $S^3$, the angle between any two points in the set will be less than $\frac{\pi}{2}$.
Under this assumption, we can completely get rid of the absolute value when we are calculating the distance of any two quaternions.

Thus the problem of finding the minimum enclosing ball of $\calS_R$ on $SO(3)$ can be conveniently translated into finding a minimum enclosing ball for one branch of the counterpart of $\calS_R$ on $S^3$. We claim finding the minimum enclosing ball on $S^3$ is equivalent to finding the minimum enclosing ball in $\mathbb{R}^4$.

\begin{proposition}[Finding miniball on $S^3$ is equivalent to finding miniball in $\mathbb{R}^4$]\label{prop:miniball-rotation}
    Given a set of rotations $\calS_R$ with angle span less than $\frac{\pi}{2}$ defined on $S^3$, the miniball of the set on $S^3$ is just the restriction of the miniball of the set in $\mathbb{R}^4$ onto $S^3$.
\end{proposition}

An analogy of Proposition~\ref{prop:miniball-rotation} is shown in Fig.~\ref{fig:analogy-miniball-rotation}. The underlying geometric intuition is that: the center of the minimum enclosing ball in $\Real{4}$ must line in the affine space defined by the contact points. And the minimum enclosing ball on $S^3$ is just the equator of the minimum enclosing ball in $\Real{4}$. Thus by projecting the corresponding center of the minimum enclosing ball in $\Real{4}$ onto $S^3$, we can get the center of the minimum enclosing ball on $S^3$.

\begin{figure}
    \centering
    \includegraphics[width=0.5\linewidth]{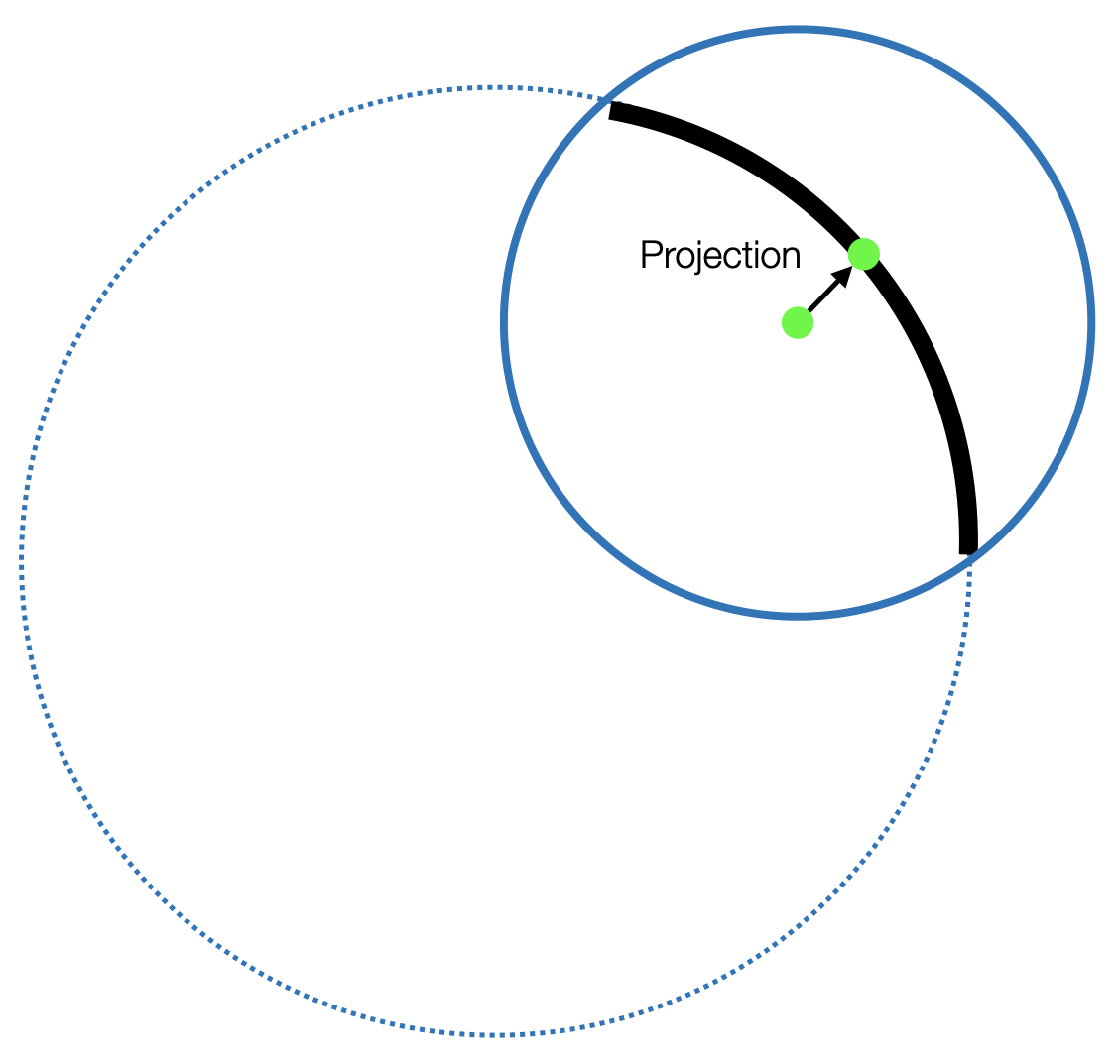}
    \caption{Analogy of Proposition~\ref{prop:miniball-rotation} in $\Real{2}$. The black line segment is a subset of $S^1$ (the dotted circle). The solid blue circle is the minimum enclosing ball of the black line segment in $\Real{2}$. The center of the solid blue circle, projected onto $S^1$, is the center of the minimum enclosing geodesic ball on $S^1$. \label{fig:analogy-miniball-rotation}}
\end{figure}

\begin{proof}
    We'll show that the restriction of the miniball of the set in $\mathbb{R}^4$ is the miniball on $S^3$.
    Suppose $B$ is the miniball of the set in $\mathbb{R}^4$, then denote $C = B\cap S^3$. Clearly, $C$ is a geodesic ball that encloses the set $\calS_R$. We claim that $C$ is also the smallest enclosing geodesic ball.
    
    
    A key observation is that, the center of $B$ must lie in the affine space, else the radius wouldn't be the smallest. The miniball must have contact points, which must lie on the intersection of $S^3$ and $B$, \ie $C$.

    Now Suppose $C$ is not the miniball in $S^3$ and $C_1$ is. Then the radius of $C_1$ on its affine space is smaller than the radius of $C$ on $T$. Then consider the ball $B_1$ in $\mathbb{R}^4$ whose equator is $C_1$. 
    
    We only need to show that $B_1$ contains the set. If we prove this, then $B_1$ a smaller miniball in $\mathbb{R}^4$ than $B$, which causes a contradiction.

    We only need to show that $B_1$ contains the ball crown enclosed by $C_1$ (in $\Real{4}$). Suppose $a_1\in \mathbb{R}^4$ is the center of $C_1$, then the radius of $C_1$ is $r_1 = \sqrt{1 - \|a_1\|^2}$. The farthest distance from the ball crown enclosed by $C_1$ to the $a_1$ is $1 - \|a_1\|$. By simple mathematics, we know that $1 - \|a_1\| < \sqrt{1 - 
    \|a_1\|^2}$ for $0<\|a_1\|<1$, so $B_1$ contains the ball crown, which completes the proof.
\end{proof}

Since the set $\calS_R$ has two branches on $S^3$ that are antipodal to each other, we need to get only one branch and compute its minium enclosing ball.
This is easily done by adding $\bar{q}\tran q\geq 0$ to restrict the quaternions to one branch. This leads to the problem in \eqref{eq:quaternion-minimax} and concludes the proof.
\end{proof}

\section{Extra Experiments}
\label{app:sec:experiments}

We provide extra experiments here to supplement the main experiments in Section~\ref{sec:experiments}.

\subsection{Simple Geometric Sets}
\label{app:sec:exp:toy-examples}
    In this section, we apply the GRCC algorithm~\eqref{eq:grcc-relax} to simple geometric shapes. We choose $Q=I$ so that we are only interested in computing the Chebyshev center.
    \begin{itemize}
        \item Sphere: $\{x\in\mathbb{R}^3| \Vert x \Vert = r\}$. The origin is the Chebyshev center and $r$ is the radius.
        \item Ball: $\{x\in\mathbb{R}^3|\|x\|_2\leq r\}$. The origin is the Chebyshev center and $r$ is the radius.
        \item Cube: $\{x\in\mathbb{R}^3|\|x\|_\infty\leq r\}$. The origin is the Chebyshev center and $\sqrt{3}r$ is the radius.
        \item Boundary of ellipsoid: $\{x\in\mathbb{R}^3|x^T P^{-1} x = 1, P\in \mathbb{S}_{++}\}$. The origin is the Chebyshev center and the root suare of the largest eigenvalue of $P$ is the radius.
        \item TV screen: $\{x\in\mathbb{R}^3|x_1^4+x_2^4+x_3^4\leq 1\}$. The origin is the Chebyshev center and $3^{1/4}$ is the radius.
    \end{itemize}

    It is worth mentioning that, even if the shapes above are simple, the RCC algorithm~\citep{eldar2008minimax} actually can't handle them. For example, some of the sets are not convex (\eg sphere and boundary of ellipsoid) or not defined by quadratic constraints (\eg cube, TV screen).
    
    We choose $r = 3$ for the experiment, and the relaxation order $\kappa=1$ ($\kappa = 2$ for the TV screen problem). And we choose $P$ to be a random positive definite matrix. The results are shown in Table~\ref{tab:toy}. Center difference is the Frobenius norm of the difference between the real Chebyshev center and the result of the SDP relaxation. Radius difference is the difference between the solution of our algorithms and the ground truth radius.

    \begin{table}[htbp] 
        \centering
        \begin{tabular}{ c|c|c|c } 
            \hline
        Example  & Can RCC handle?& Center difference & Radius difference \\ 
        \hline
        Sphere & NO &4.7924e-17 &  6.2866e-10  \\ 
        Ball & YES &9.9467e-18 &  6.4283e-09  \\
        Cube  & YES &1.0080e-17 &  2.6002e-09  \\
        Ellipsoid  & NO & 1.1486e-16 & 9.6512e-10  \\
        TV screen  & NO & 1.8918e-17 & 4.1497e-09  \\
        \hline
        \end{tabular} 
        \caption{Computing the Chebyshev center of simple geometric shapes.}
        \label{tab:toy}
    \end{table}
    We can see that for simple geometrical shapes, the GRCC relaxation is tight even when $\kappa=1$. 

    {\bf Certificate for the finite convergence of the MEE}.
    Here we present the example of TV Screen and leverage the certificate of MEE~\eqref{prop:certificate-mee} proposed in Appendix~\ref{app:sec:proof-certificate-mee} to show the finite convergence. For this symmetric shape, the results of GRCC and MEE coincide.

\begin{figure} 
    \centering
    \includegraphics[width=0.5\linewidth]{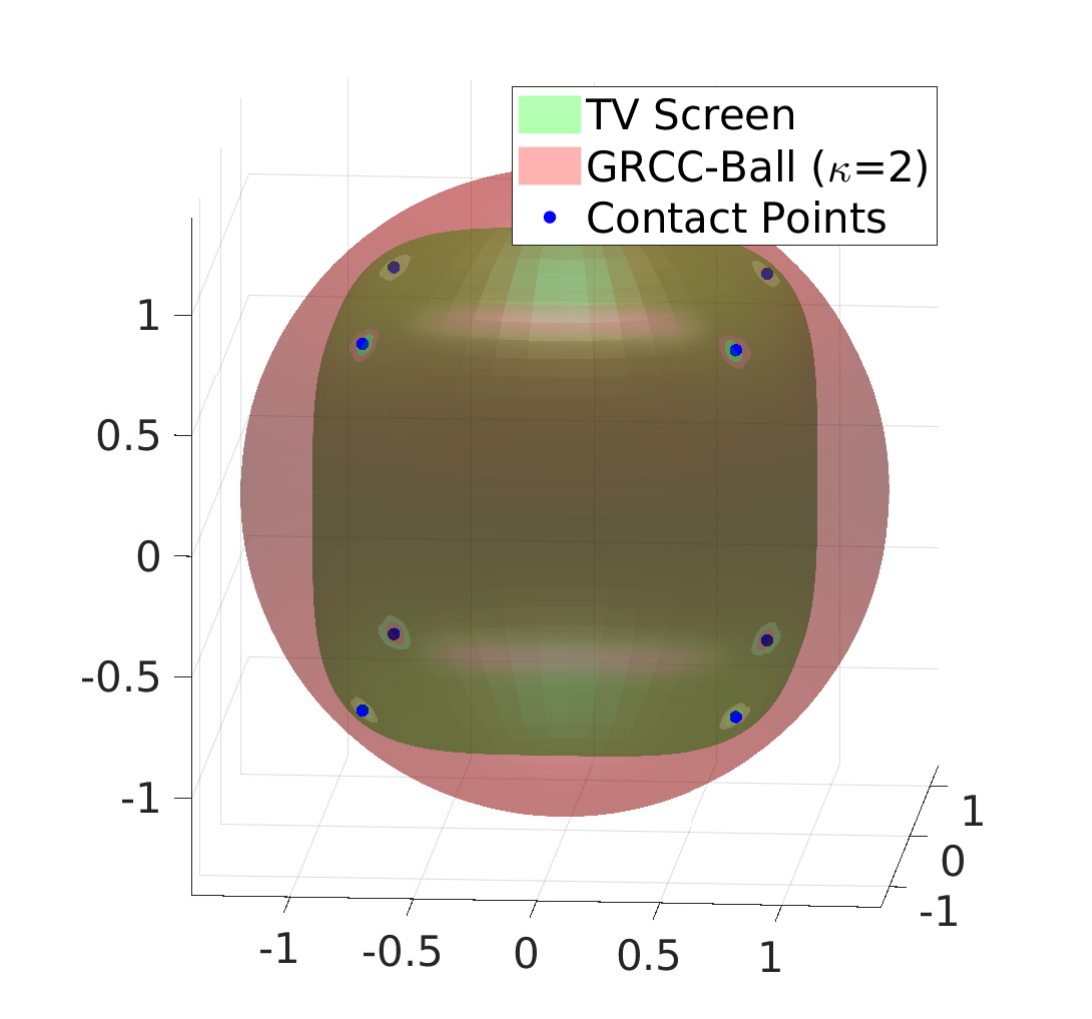}
    \caption{The certificate of finite convergence for TV Screen Example. \label{fig:toy-example-touch}}
\end{figure}

We first get the radius of the ball by letting $\kappa=2$. Then by solving the nonconvex optimization problem~\eqref{eq:pop-certificate}, we can extract eight contact points:

\begin{align*}
    &\bmat{c} 
0.7598\\
0.7598\\
0.7598
\emat, \bmat{c} 
-0.7598\\
0.7598\\
0.7598
\emat, \bmat{c} 
-0.7598\\
-0.7598\\
-0.7598
\emat, \bmat{c} 
0.7598\\
-0.7598\\
-0.7598
\emat, \\
&\bmat{c} 
0.7598\\
0.7598\\
-0.7598
\emat, \bmat{c} 
-0.7598\\
-0.7598\\
0.7598
\emat, \bmat{c} 
0.7598\\
-0.7598\\
0.7598
\emat,\bmat{c} 
-0.7598\\
0.7598\\
-0.7598
\emat
\end{align*}

By solving the feasibility problem~\eqref{eq:contact-points} with a slackness of $10^{-15}$, we can get the following $\alpha$:
$$0,0.75,0.75,0,0.75,0,0.75,0$$
which certifies the finite convergence of the MEE algorithm. The figure is shown in Figure~\ref{fig:toy-example-touch}.

\subsection{Random Linear system}

To supplement Section~\ref{sec:exp:system-id:rand-linear}, here we provide more results on random linear system identification where the matrix $A_\star$ has size $n_x = 3$ and $n_x=4$.

\begin{figure}[h]
	\begin{center}
		\begin{tabular}{ccc}
            \hspace{-10mm}	
            \begin{minipage}{0.34\textwidth}
				\centering
				\includegraphics[width=\textwidth]{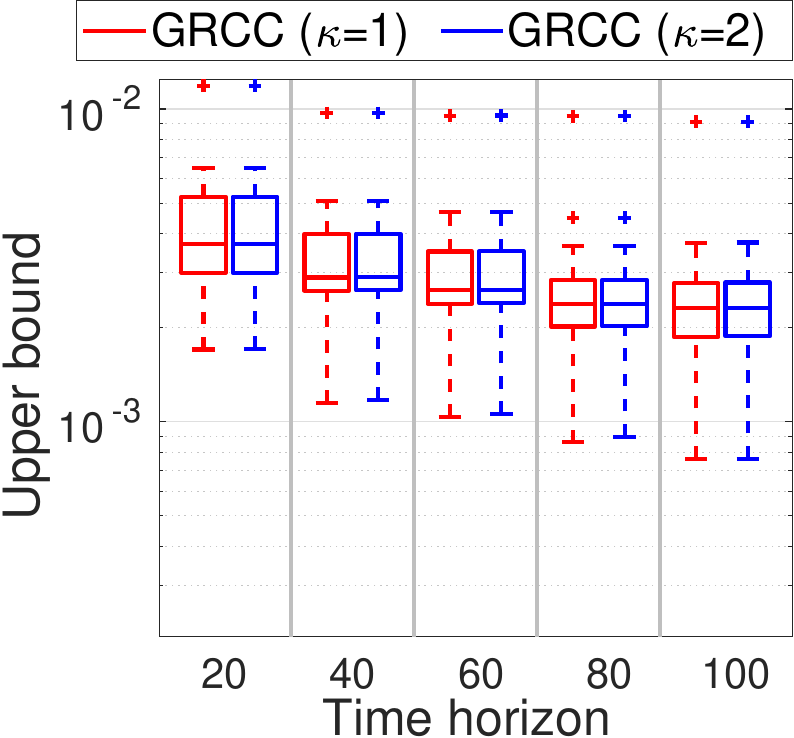}
			\end{minipage}
			&\hspace{-3mm}
			\begin{minipage}{0.34\textwidth}
				\centering
				\includegraphics[width=\textwidth]{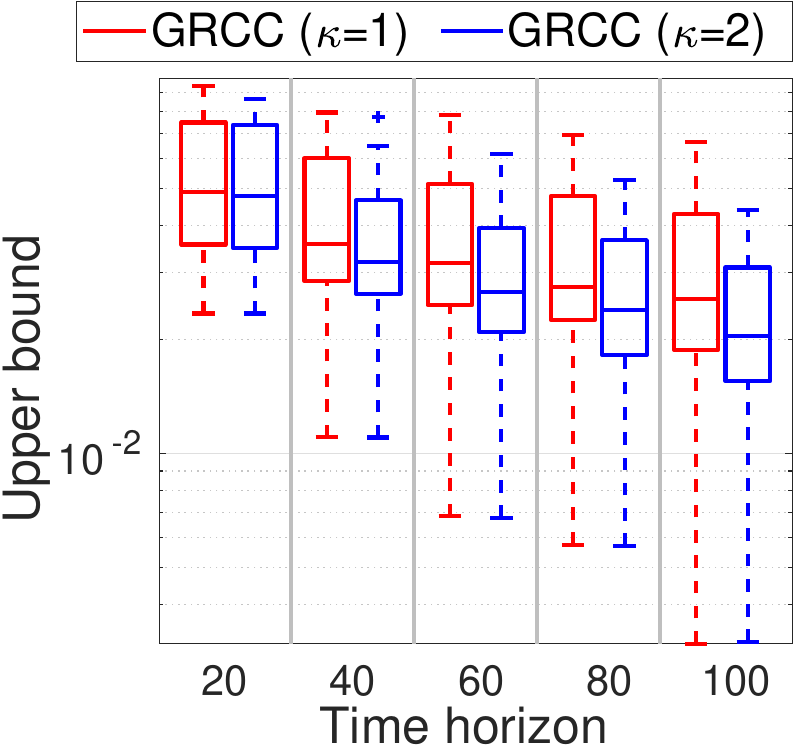}
			\end{minipage}
			&\hspace{-3mm}
			\begin{minipage}{0.34\textwidth}
				\centering
				\includegraphics[width=\textwidth]{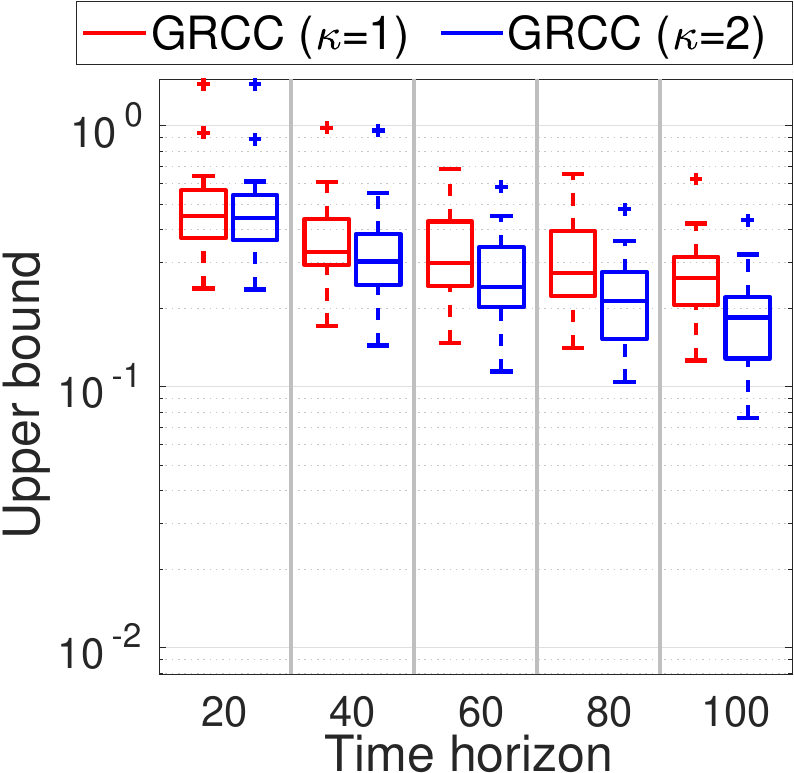}
			\end{minipage}\\
			\multicolumn{3}{c}{(a) $n_x$ = 3}
		\end{tabular}

        \begin{tabular}{ccc}
            \hspace{-10mm}	
            \begin{minipage}{0.34\textwidth}
				\centering
				\includegraphics[width=\textwidth]{figures/linear-n-4-alpha-0.01.pdf}
			\end{minipage}
			&\hspace{-3mm}
			\begin{minipage}{0.34\textwidth}
				\centering
				\includegraphics[width=\textwidth]{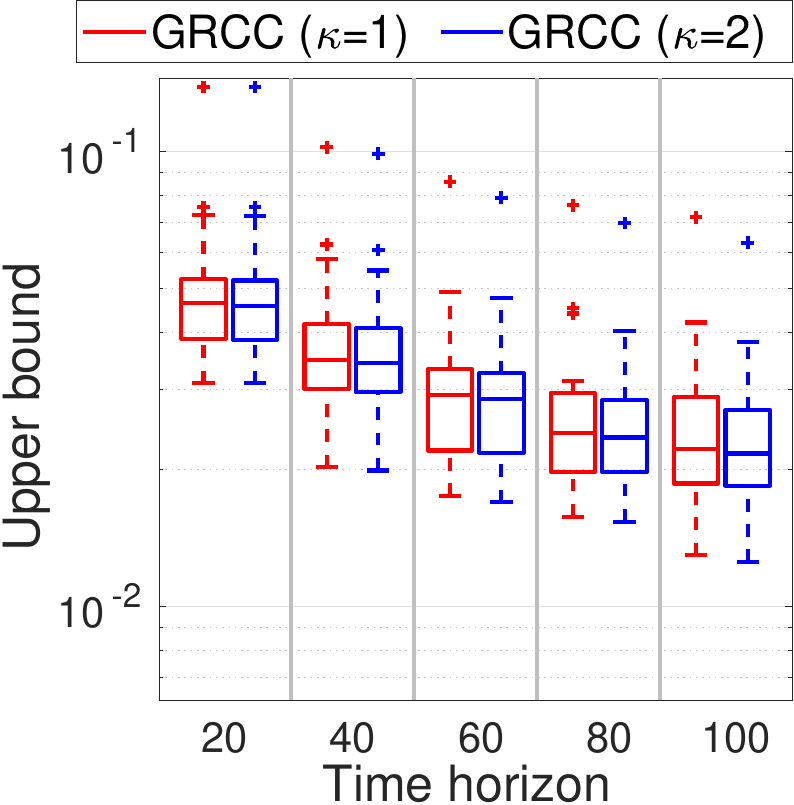}
			\end{minipage}
			&\hspace{-3mm}
			\begin{minipage}{0.34\textwidth}
				\centering
				\includegraphics[width=\textwidth]{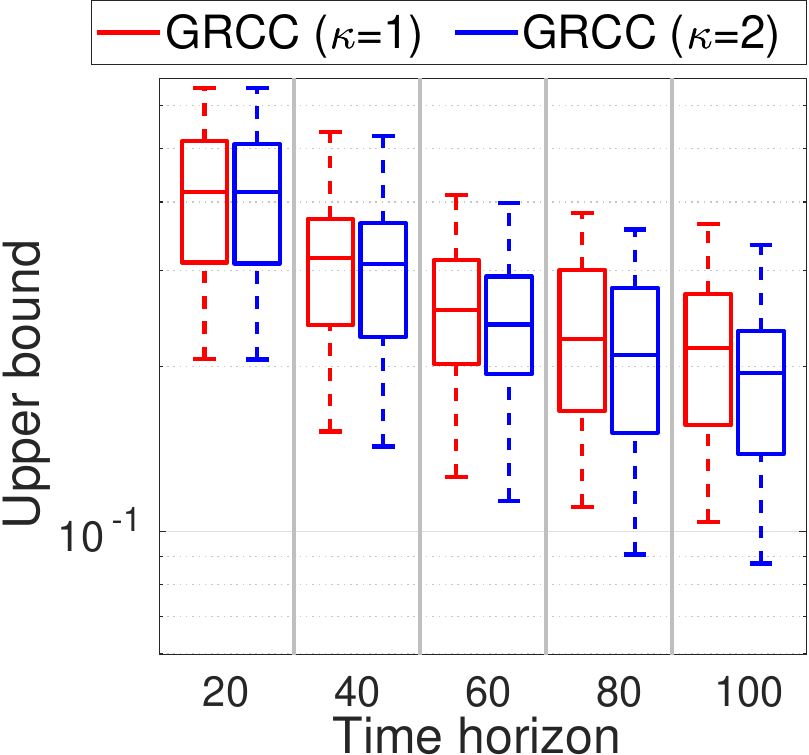}
			\end{minipage}\\
			\multicolumn{3}{c}{(a) $n_x$ = 4}
		\end{tabular}
	\end{center}
	\vspace{-6mm}
	\caption{Identification of random linear systems with $n_x = 3$ and $n_x =4$. From left to right: increasing noise bounds $\beta = 0.01$, $\beta=0.1$, $\beta=1.0$. See Fig.~\ref{fig:exp:system-id}(b) in the main text for results with $n_x=2$. \label{fig:random-linear-3-4}} 
\end{figure}

{
\subsection{The Estimation Quality of the Sampling-based Method}

To supplement the estimated $Q$ matrix in Equation~\ref{eq:estimate-Q}, we present analysis for four different cases.

\begin{itemize}
    \item Rectangle: $\{(x,y)\in\Real{2}\mid -2\leq x\leq 2, -1\leq y \leq 1\}$
    \item Triangle: $\{(x,y)\in\Real{2}\mid x\geq -2,  y \leq 1, y\geq 0.5x\}$
    \item Triangle + Parabola: $\{(x,y)\in\Real{2}\mid y\leq x+1, y\leq -x+1, y\geq x^2 - 1\}$
    \item Random Linear system: The system same as in Section~\ref{sec:exp:system-id:rand-linear}, with $N = 15$.
\end{itemize}

We first use MEE method with $\kappa = 4$ to obtain ground truth $Q_0$ matrix for these four simple examples. Then we do the sampling in the set and estimate the $Q$ matrix by PCA. We normalize $Q_0$ and $Q$ and take their Frobenius distance as the estimation error. We choose the sampling number from $10^2$ to $10^7$ in log scale and perform $20$ times of sampling for each sampling number. The result is shown in Fig~\ref{fig:estimation-Q}.

\begin{figure}[h]
	\begin{center}
		\begin{tabular}{cc}
            \hspace{-10mm}	
            \begin{minipage}{0.5\textwidth}
				\centering
				\includegraphics[width=\textwidth]{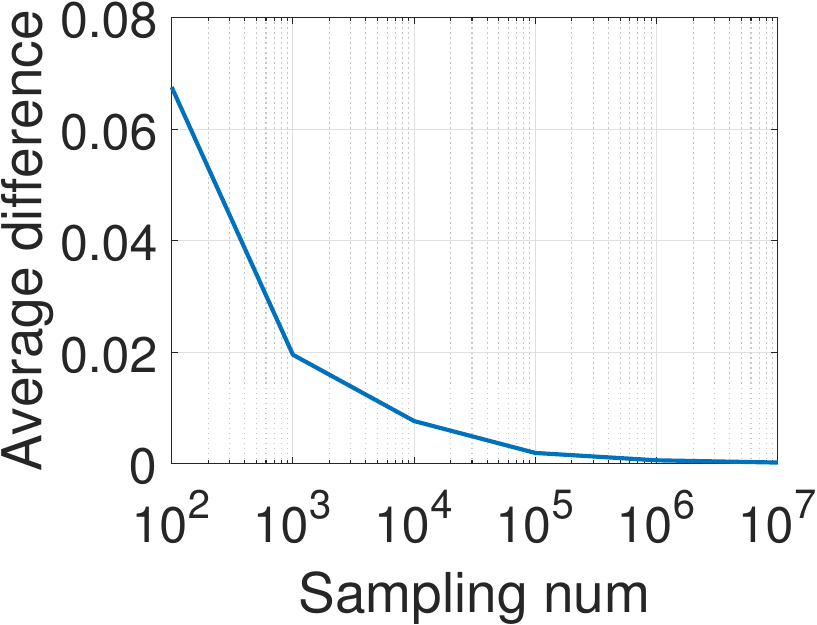}
                \caption*{Example 1}
			\end{minipage}
			&\hspace{-3mm}
			\begin{minipage}{0.5\textwidth}
				\centering
				\includegraphics[width=\textwidth]{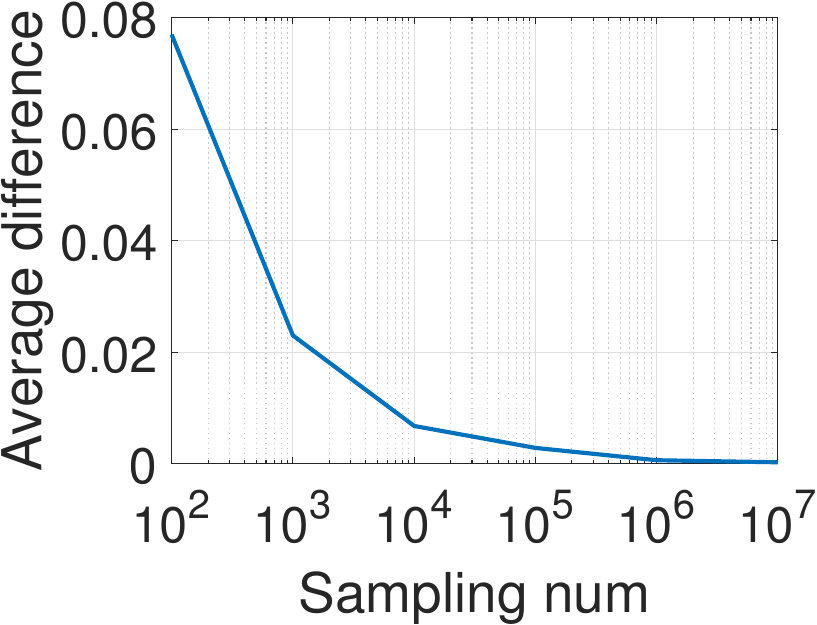}
                \caption*{Example 2}
			\end{minipage}
		\end{tabular}

        \begin{tabular}{cc}
            \hspace{-10mm}	
            \begin{minipage}{0.5\textwidth}
				\centering
				\includegraphics[width=\textwidth]{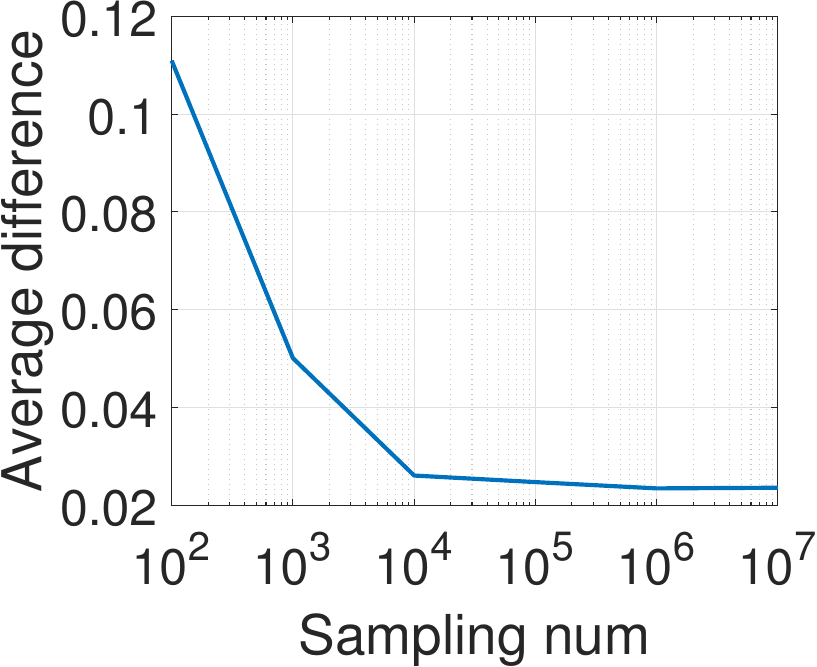}
                \caption*{Example 3}
			\end{minipage}
			&\hspace{-3mm}
			\begin{minipage}{0.5\textwidth}
				\centering
				\includegraphics[width=\textwidth]{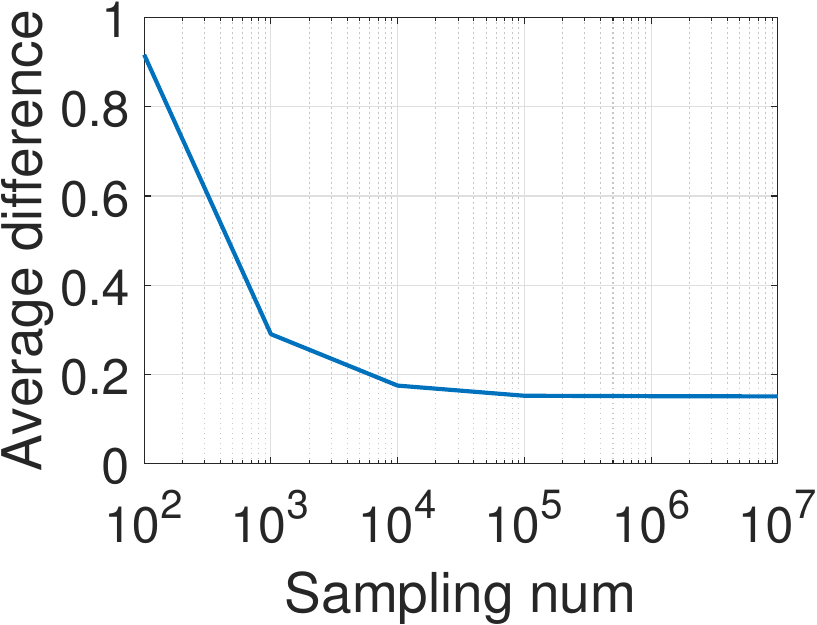}
                \caption*{Example 4}
			\end{minipage}
		\end{tabular}
	\end{center}
    \vspace{-6mm}
	\caption{Average difference between normalized ground truth $Q_0$ matrix and estimated $Q$ \label{fig:estimation-Q}}
\end{figure}

{\bf Interpretation of the result}.

In the first two examples, under sufficiently large sampling number, the difference between the normalized $Q_0$ and $Q$ converges to zero, which means that the shape estimated by the sampling and PCA method do coincide with the ground truth one.

In the third examples, we can see that although $Q$ doesn't converge to $Q_0$ eventually, the difference is small.

For the last example, we can see that there is a gap between $Q$ and $Q_0$. But looking at the data in Table~\ref{tab:add-exp-table}, we can see that the volume difference between MEE and GRCC is not that large. Thus, we can find that $Q$ estimated by sampling is indeed a good estimator of the ground truth shape $Q_0$.

\subsection{Time results for the random linear system experiment in Section~\ref{sec:exp:system-id:rand-linear}}

    The runtime for the $n_x = 2$ experiment is provided as follows in Table~\ref{tab:rand-lin-time-001},~\ref{tab:rand-lin-time-01} and~\ref{tab:rand-lin-time-1}:
    \begin{table}
        \centering
        \begin{tabular}{|c|c|c|c|c|c|}
            \hline 
            & RCC & $\kappa = 1$ & $\kappa = 2$ & $\kappa = 3$ & $\kappa = 4$ \\
            \hline
            $N=20$ &$$0.1493$$ &$$0.0654$$  &$$0.1593$$  &$3.7697$  &$44.7019$  \\
            \hline
            $N=40$ &$0.1185$  &$0.1033$   &$0.865$   &$11.5882$   &$235.4550$   \\
            \hline
            $N=60$ &$0.1327$  &$0.1562$   &$0.6808$   &$17.8974$   &$652.5686$   \\
            \hline
            $N=80$ &$0.1505$  &$0.2146$   &$1.0232$   &$28.3246$   &$1415.6274$   \\
            \hline
            $N=100$ &$0.1733$  &$0.2893$  &$1.4937$  &$43.2002$  &$2561.3750$  \\
            \hline
           \end{tabular}
           \caption{Time result for $\alpha = 0.01$ \label{tab:rand-lin-time-001}}
    \end{table}
    \begin{table}
        \centering
        \begin{tabular}{|c|c|c|c|c|c|}
            \hline 
            & RCC & $\kappa = 1$ & $\kappa = 2$ & $\kappa = 3$ & $\kappa = 4$ \\
            \hline
            $N=20$ &$$0.0964$$ &$$0.0509$$  &$$0.1448$$  &$3.7781$  &$43.3341$  \\
            \hline
            $N=40$ &$0.1123$  &$0.1008$   &$0.3721$   &$10.8165$   &$204.2520$   \\
            \hline
            $N=60$ &$0.1301$  &$0.1554$   &$0.6363$   &$16.5717$   &$630.0205$   \\
            \hline
            $N=80$ &$0.1466$  &$0.2129$   &$1.0177$   &$29.4593$   &$1378.5854$   \\
            \hline
            $N=100$ &$0.1646$  &$0.2913$  &$1.5479$  &$45.0319$  &$2353.4675$  \\
            \hline
           \end{tabular}
           \caption{Time result for $\alpha = 0.1$ \label{tab:rand-lin-time-01}}
    \end{table}
    \begin{table}
        \centering
        \begin{tabular}{|c|c|c|c|c|c|}
            \hline 
            & RCC & $\kappa = 1$ & $\kappa = 2$ & $\kappa = 3$ & $\kappa = 4$ \\
            \hline
            $N=20$ &$$0.1499$$ &$$0.0653$$  &$$0.1288$$  &$2.4671$  &$31.4998$  \\
            \hline
            $N=40$ &$0.1206$  &$0.1048$   &$0.3052$   &$8.9280$   &$181.6477$   \\
            \hline
            $N=60$ &$0.1346$  &$0.1572$   &$0.5957$   &$16.3963$   &$570.1058$   \\
            \hline
            $N=80$ &$0.1508$  &$0.2158$   &$0.9260$   &$28.2859$   &$1239.9052$   \\
            \hline
            $N=100$ &$0.1708$  &$0.2946$  &$1.3095$  &$46.2027$  &$2286.8749$  \\
            \hline
           \end{tabular}
           \caption{Time result for $\alpha = 1$ \label{tab:rand-lin-time-1}}
    \end{table}

}

\subsection{Time results for the constraints pruning experiment in Section~\ref{sec:exp:system-id:pendulum-long-trajectory}}

{The runtime of this experiment is provided as follows: Pruning time for $1000$ constraints is $39.47$s, sampling time is $39.67$s for $100000$ trials. The acceptance rate is $0.27\%$ because the uncertainty set is extremely small. The runtime of GRCC ellipsoid with $\kappa=6$ is $419.94$s, and the GRCC ball with $\kappa=7$ is $38553.65$s.}

\subsection{{Analysis of Three Computational Improvements}}

{
    Consider a random linear system same as the one in Section~\ref{sec:exp:system-id:rand-linear}. We choose $n_x = 2, N = 50, 1000, \beta = 0.5$. We perform experiments combining three components: (i) MEE in Theorem~\ref{thm:mee-sos}, (ii) Constraints pruning in Section~\ref{sec:prune-redundant} and (iii) GRCC algorithm in Theorem~\ref{thm:grcc}. 
    
    More specifically, we perform experiments under following settings: (i) MEE (with $\kappa = 2,4$), (ii) Constraints Pruning + MEE (with $\kappa = 2,4$), (iii) GRCC (with $\kappa = 2,4$) and (iv) Constraints Pruning + GRCC (with $\kappa = 2,4$). For the GRCC algorithm, we first solve the enclosing ball with $\kappa = 2$ and then do the rejection sampling in the enclosing ball to estimate the shape $Q$. We didn't perform the GRCC algorithm with $\kappa = 4$ when $N = 1000$ because the problem can't be handled by MOSEK. The results are summarized in Table~\ref{tab:add-exp-table}.







    {\bf Interpretation of the results}.

    \begin{enumerate}
        \item Comparing (i) and (ii), we observe that the pruning algorithm can significantly reduce the runtime for MEE algorithm. It can also increase the numerical precision. Also, when the trajectory is short, the MEE can get tighter bound comparing to the GRCC algorithm.
        \item Comparing (i) and (iii), we can see that the original GRCC algorithm is slightly more conservative when $N=50$ and more time-consuming when the trajectory is long. However, we need to note that the MEE algorithm is not numerically stable for we get UNKNOWN for (i) when $N=1000$. 
        \item Comparing (iii) and (iv), we can see that the pruning algorithm can significantly reduce the runtime for GRCC algorithm. Also, from the results when $N=1000$, we can see that the GRCC algorithm get tighter bound by the pruning algorithm due to better numerical performance.
        \item Comparing (i), (ii) and (iv), we can see that pruned GRCC is more conservative than MEE algorithm. However, it is more numerically stable and it is close to the MEE results, which shows that the estimated $Q$ is a relatively good estimation.
        \item Comparing the results of $N = 50$ and $N=1000$, we can see that the MEE algorithm is not numerically stable enough when the trajectory becomes longer, while the GRCC algorithm is more stable with respect to the trajectory length.
    \end{enumerate}


\begin{table}
    \centering
    \begin{tabular}{|c|c|c||c|c|}
        \hline 
                & \multicolumn{2}{c||}{$N=50$} & \multicolumn{2}{c|}{$N=1000$}\\
        Methods & Runtime & Volume & Runtime & Volume \\
        \hline 
        Pruning &0.9450 &$\backslash$  &42.3442  &$\backslash$    \\
        Sampling w/o Pruning &0.5343  &$\backslash$  &90.4302 &$\backslash$   \\
        Sampling w/ Pruning &0.1398  &$\backslash$  &0.6129  &$\backslash$   \\
        \hline
        MEE ($\kappa=2$) &2.1273 &1.3983e-04  &58.1813  &UNKNOWN   \\
        MEE ($\kappa=4$) &22.6902  &1.2382e-04  &7963.7482  &UNKNOWN   \\
        \hline 
        MEE w/ Pruning ($\kappa=2$) &0.4851  &1.3973e-04  &0.5927  &UNKNOWN    \\
        MEE w/ Pruning ($\kappa=4$) &5.9770  &1.2301e-04  &6.2782  &UNKNOWN   \\ 
        \hline 
        GRCC ($\kappa=2$) &0.4745  &1.4975e-04  &99.0342  &1.1996e-07    \\
        GRCC ($\kappa=4$) &477.0226  &1.4065e-04  &N/A  &N/A   \\ 
        \hline 
        GRCC w/ Pruning ($\kappa=2$) &0.0903  &1.4975e-04  &0.0983  &9.3510e-09    \\
        GRCC w/ Pruning ($\kappa=4$) &28.2055  &1.4035e-04  &19.2076  &8.7840e-10  \\
        \hline  
       \end{tabular}
       \caption{Summary table. \label{tab:add-exp-table}}
\end{table}

}

{
    \section{Summary}\label{app:sec:summary}
    We hereby summarize the three computational enhancements and give an overview of the complete workflow for computing the MEE of the SME in Example~\ref{ex:sme-system-id} and Example~\ref{ex:sme-object-pose}.


    \begin{enumerate}[leftmargin=18mm, label= Step \arabic*:]
        \item Constraints Pruning. If the number of constraints in the original problem is too large, we can prune redundant constraints by solving cheap moment-SOS relaxations using Algorithm~\ref{alg:prune-redundant-constraints}. This can not only reduce the runtime significantly, it can also benefit the numerical performance of the optimization algorithm.
        \item Sample Points and Estimate $Q$. In this step, if a sampling algorithm is already available, we can directly use the sampling algorithm. If not, we first solve the GRCC problem with $Q=I$ to obtain a bounding ball of the set, then we do rejection sampling in the ball. At last, we leverage Equation~\eqref{eq:estimate-Q} (\ie PCA) to estimate the $Q$ matrix using the sampled points.
        \item Solve the GRCC problem using estimated $Q$. For normal problems in Euclidean space, we choose an relaxation order and solve GRCC. For problems in the \SOthree space, we leverage Theorem~\ref{thm:minimum-ball-SO3} to solve the Chebyshev.
    \end{enumerate}
    \vspace{-5mm}
}

\section{Related Work}
\label{sec:related-work}

We provide a brief review of set-membership estimation (SME) in control and perception (mainly related to system identification and object pose estimation), as well as existing algorithms that seek simpler enclosing sets of SME. For related work about system identification, we refer to \cite{oymak2019non,bakshi23arxiv-new}, for related work about object pose estimation, we refer to \cite{yang2023object} and references therein.

As mentioned in Section~\ref{sec:intro}, set-membership estimation (SME) is a framework for model estimation which doesn't place any assumptions on the statistical distribution of the noise. It only assumes the noises are unknown but bounded. SME dates back to at least~\cite{milanese91automatica-optimal} and~\cite{kosut92tac-set} in control theory where the main interest was to estimate unknown parameters of dynamical systems. \cite{li23arxiv-learning} proves that the SME converges to a single point estimate (\ie the groundtruth) as the system trajectory becomes long enough when the noise satisifies bounded Gaussian distributions. Recently, \cite{yang2023object} applied SME to a computer vision example known as object pose estimation. 

Because an SME is defined by complex constraints and the number of constraints increases as the horizon of the system gets longer, an important task in set-membership estimation is to find simpler and more intuitive sets that enclose the original sets~\citep{walter1994characterizing}. The main idea of applying SME framework to system identification is to maintain an easily-computed outer approximation of the uncertainty set. There are two main ways to get the outer approximation, namely \emph{batch computation} or \emph{recursive computation}.

For batch calculations, we observe a sequence of inputs and outputs of the system, and the uncertainty set is all the possible parameters that can generate the sequence under the assumption of the bounded noise. \cite{cerone2011set} propose an algorithm leveraging Lasserre's hierarchy to calculate the bounding interval on every dimension of the parameters. In robotics, similar bounding intervals have been proposed in the context of simultaneous localization and mapping~\citep{ehambram22icra-interval,mustafa18ras-guaranteed}. \cite{casini2014feasible} propose relaxation algorithms to approximate minimum volume enclosing orthotope. An exstention to bounding polytope is also shown by calculating the interval bounds on different directions. \cite{eldar2008minimax} propose a relaxation algorithm based on Shor's semidefinite relaxation to calculate the enclosing ball of the sets defined by quadratic constraints. \cite{durieu2001multi} consider the ellipsoidal approximation of the intersection of $K$ ellipsoids, and solved the minimum enclosing ellipsoid in a certain parametric class. \cite{dabbene13ecc-set} proposed an asymptotically convergent algorithm using sublevel sets of polynomials to enclose the uncertainty set.

For recursive calculations, we sequencially observe new observations and update the approximation according to new observations. \cite{fogel1982value} propose a recursive formula on how to update the bounding ellipsoid minimizing the volume and trace after each observation. \cite{walter1989exact} propose the algorithm to update the polyhedral description of the feasible parameter set. \cite{valero2019ifac-recurparall} propose how to update the parallelotope by intersecting the current one with the observation strips, and then find the minimum one.
In favor of computational efficiency, recursive methods find an outer approximation of the feasible set at every time step, thus are always not as exact as the batch methods. 

{\bf Position of our algorithms}. Our SOS-MEE algorithm (Theorem~\ref{thm:mee-sos}) and GRCC algorithm (Theorem~\ref{thm:grcc}) both belong to batch computation algorithms (although it is easy to see that our algorithms can also be used in a recursive fashion). Our GRCC formulation is more general and contains the minimum enclosing interval studied in \cite{cerone2011set} by choosing $\xi$ as just one dimension of $\theta$. \cite{durieu2001multi} calculates the minimum enclosing ellipsoids only in a certain parametric set, so it could be sub-optimal, while our SOS-MEE algorithm has asymptotic convergence to the true MEE and the convergence can be certified. Moreover it can only handle the case of intersections of ellipsoids while we can handle arbitrary basic semialgebraic sets. Our GRCC algorithm strictly generalizes \citep{eldar2008minimax}, as it's just the simplest case ($\kappa=1$) of our algorithm. Lastly, although computing the enclosing ellipsoids is not as as general as computing polynomial level sets as in~\cite{dabbene13ecc-set}, our algorithms are more computational efficient and implementable, as demonstrated by numerial experiments in Section~\ref{sec:experiments}. Our contributions lie in the three computational enhancements (Section~\ref{sec:improvement}) that make the SOS-based framework tractable and applicable to modern system identification and object pose estimation problems.

\bibliography{refs.bib}

\begin{thebibliography}{59}
\providecommand{\natexlab}[1]{#1}
\providecommand{\url}[1]{\texttt{#1}}
\expandafter\ifx\csname urlstyle\endcsname\relax
  \providecommand{\doi}[1]{doi: #1}\else
  \providecommand{\doi}{doi: \begingroup \urlstyle{rm}\Url}\fi

\bibitem[Abbasi-Yadkori and Szepesv{\'a}ri(2011)]{abbasi11colt-regret}
Yasin Abbasi-Yadkori and Csaba Szepesv{\'a}ri.
\newblock Regret bounds for the adaptive control of linear quadratic systems.
\newblock In \emph{Proceedings of the 24th Annual Conference on Learning
  Theory}, pages 1--26. JMLR Workshop and Conference Proceedings, 2011.

\bibitem[Angelopoulos and Bates(2021)]{angelopoulos21arxiv-conformal}
Anastasios~N Angelopoulos and Stephen Bates.
\newblock A gentle introduction to conformal prediction and distribution-free
  uncertainty quantification.
\newblock \emph{arXiv preprint arXiv:2107.07511}, 2021.

\bibitem[Antonante et~al.(2021)Antonante, Tzoumas, Yang, and
  Carlone]{antonante21tro-outlier}
Pasquale Antonante, Vasileios Tzoumas, Heng Yang, and Luca Carlone.
\newblock Outlier-robust estimation: Hardness, minimally tuned algorithms, and
  applications.
\newblock \emph{IEEE Transactions on Robotics}, 38\penalty0 (1):\penalty0
  281--301, 2021.

\bibitem[ApS(2019)]{aps19mosek}
Mosek ApS.
\newblock Mosek optimization toolbox for matlab.
\newblock \emph{User's Guide and Reference Manual, Version}, 4:\penalty0 1,
  2019.

\bibitem[Arnaudon and Nielsen(2013)]{arnaudon13cg-approximating}
Marc Arnaudon and Frank Nielsen.
\newblock On approximating the riemannian 1-center.
\newblock \emph{Computational Geometry}, 46\penalty0 (1):\penalty0 93--104,
  2013.

\bibitem[Bakshi et~al.(2023)Bakshi, Liu, Moitra, and Yau]{bakshi23arxiv-new}
Ainesh Bakshi, Allen Liu, Ankur Moitra, and Morris Yau.
\newblock A new approach to learning linear dynamical systems.
\newblock \emph{arXiv preprint arXiv:2301.09519}, 2023.

\bibitem[Barfoot(2017)]{barfoot17book-state}
Timothy~D Barfoot.
\newblock \emph{State estimation for robotics}.
\newblock Cambridge University Press, 2017.

\bibitem[B{\'e}lisle et~al.(1993)B{\'e}lisle, Romeijn, and
  Smith]{belisle93mor-hit}
Claude~JP B{\'e}lisle, H~Edwin Romeijn, and Robert~L Smith.
\newblock Hit-and-run algorithms for generating multivariate distributions.
\newblock \emph{Mathematics of Operations Research}, 18\penalty0 (2):\penalty0
  255--266, 1993.

\bibitem[Blekherman et~al.(2012)Blekherman, Parrilo, and
  Thomas]{blekherman12siam-semidefinite}
Grigoriy Blekherman, Pablo~A Parrilo, and Rekha~R Thomas.
\newblock \emph{Semidefinite optimization and convex algebraic geometry}.
\newblock SIAM, 2012.

\bibitem[Brachmann et~al.(2014)Brachmann, Krull, Michel, Gumhold, Shotton, and
  Rother]{brachmann14eccv-lmo}
Eric Brachmann, Alexander Krull, Frank Michel, Stefan Gumhold, Jamie Shotton,
  and Carsten Rother.
\newblock Learning 6d object pose estimation using 3d object coordinates.
\newblock In \emph{European Conf. on Computer Vision (ECCV)}, pages 536--551.
  Springer, 2014.

\bibitem[Caron et~al.(1989)Caron, McDonald, and Ponic]{caron89jota-degenerate}
RJ~Caron, JF~McDonald, and CM~Ponic.
\newblock A degenerate extreme point strategy for the classification of linear
  constraints as redundant or necessary.
\newblock \emph{Journal of Optimization Theory and Applications}, 62\penalty0
  (2):\penalty0 225--237, 1989.

\bibitem[Casini et~al.(2014)Casini, Garulli, and Vicino]{casini2014feasible}
Marco Casini, Andrea Garulli, and Antonio Vicino.
\newblock Feasible parameter set approximation for linear models with bounded
  uncertain regressors.
\newblock \emph{IEEE Transactions on Automatic Control}, 59\penalty0
  (11):\penalty0 2910--2920, 2014.

\bibitem[Cerone et~al.(2011)Cerone, Piga, and Regruto]{cerone2011set}
Vito Cerone, Dario Piga, and Diego Regruto.
\newblock Set-membership error-in-variables identification through convex
  relaxation techniques.
\newblock \emph{IEEE Transactions on Automatic Control}, 57\penalty0
  (2):\penalty0 517--522, 2011.

\bibitem[Cotorruelo et~al.(2020)Cotorruelo, Kolmanovsky, Ram{\'\i}rez, Limon,
  and Garone]{cotorruelo20arxiv-elimination}
Andres Cotorruelo, Ilya Kolmanovsky, Daniel~R Ram{\'\i}rez, Daniel Limon, and
  Emanuele Garone.
\newblock Elimination of redundant polynomial constraints and its use in
  constrained control.
\newblock \emph{arXiv preprint arXiv:2006.14957}, 2020.

\bibitem[Dabbene and Henrion(2013)]{dabbene13ecc-set}
Fabrizio Dabbene and Didier Henrion.
\newblock Set approximation via minimum-volume polynomial sublevel sets.
\newblock In \emph{European Control Conference}, pages 1114--1119. IEEE, 2013.

\bibitem[Durieu et~al.(2001)Durieu, Walter, and Polyak]{durieu2001multi}
C{\'e}cile Durieu, E~Walter, and Boris Polyak.
\newblock Multi-input multi-output ellipsoidal state bounding.
\newblock \emph{Journal of optimization theory and applications}, 111:\penalty0
  273--303, 2001.

\bibitem[Ehambram et~al.(2022)Ehambram, Voges, Brenner, and
  Wagner]{ehambram22icra-interval}
Aaronkumar Ehambram, Raphael Voges, Claus Brenner, and Bernardo Wagner.
\newblock Interval-based visual-inertial lidar slam with anchoring poses.
\newblock In \emph{IEEE Intl. Conf. on Robotics and Automation (ICRA)}, pages
  7589--7596. IEEE, 2022.

\bibitem[Eldar et~al.(2008)Eldar, Beck, and Teboulle]{eldar2008minimax}
Yonina~C Eldar, Amir Beck, and Marc Teboulle.
\newblock A minimax chebyshev estimator for bounded error estimation.
\newblock \emph{IEEE transactions on signal processing}, 56\penalty0
  (4):\penalty0 1388--1397, 2008.

\bibitem[Fogel and Huang(1982)]{fogel1982value}
Eli Fogel and Yih-Fang Huang.
\newblock On the value of information in system identification—bounded noise
  case.
\newblock \emph{Automatica}, 18\penalty0 (2):\penalty0 229--238, 1982.

\bibitem[G{\"a}rtner(1999)]{gartner99esa-fast}
Bernd G{\"a}rtner.
\newblock Fast and robust smallest enclosing balls.
\newblock In \emph{European symposium on algorithms}, pages 325--338. Springer,
  1999.

\bibitem[Hartley and Zisserman(2003)]{hartley03book-multiple}
Richard Hartley and Andrew Zisserman.
\newblock \emph{Multiple view geometry in computer vision}.
\newblock Cambridge university press, 2003.

\bibitem[Henk(2012)]{henk12-lowner}
Martin Henk.
\newblock l{\"o}wner-john ellipsoids.
\newblock \emph{Documenta Math}, 95:\penalty0 106, 2012.

\bibitem[Henrion and Lasserre(2005)]{henrion05-detecting}
Didier Henrion and Jean-Bernard Lasserre.
\newblock Detecting global optimality and extracting solutions in gloptipoly.
\newblock In \emph{Positive polynomials in control}, pages 293--310. Springer,
  2005.

\bibitem[Huber(2004)]{huber04book-robust}
Peter~J Huber.
\newblock \emph{Robust statistics}, volume 523.
\newblock John Wiley \& Sons, 2004.

\bibitem[Kojima and Yamashita(2013)]{kojima13mp-enclosing}
Masakazu Kojima and Makoto Yamashita.
\newblock Enclosing ellipsoids and elliptic cylinders of semialgebraic sets and
  their application to error bounds in polynomial optimization.
\newblock \emph{Mathematical Programming}, 138\penalty0 (1-2):\penalty0
  333--364, 2013.

\bibitem[Korkmaz et~al.(2014)Korkmaz, G{\"o}ks{\"u}l{\"u}k, and
  Zararsiz]{korkmaz14R-mvn}
Selcuk Korkmaz, Din{\c{c}}er G{\"o}ks{\"u}l{\"u}k, and G{\"O}KMEN Zararsiz.
\newblock Mvn: An r package for assessing multivariate normality.
\newblock \emph{R JOURNAL}, 6\penalty0 (2), 2014.

\bibitem[Kosut et~al.(1992)Kosut, Lau, and Boyd]{kosut92tac-set}
Robert~L Kosut, Ming~K Lau, and Stephen~P Boyd.
\newblock Set-membership identification of systems with parametric and
  nonparametric uncertainty.
\newblock \emph{IEEE Transactions on Automatic Control}, 37\penalty0
  (7):\penalty0 929--941, 1992.

\bibitem[Lasserre(2001)]{lasserre2001global}
Jean~B Lasserre.
\newblock Global optimization with polynomials and the problem of moments.
\newblock \emph{SIAM Journal on optimization}, 11\penalty0 (3):\penalty0
  796--817, 2001.

\bibitem[Lasserre(2015)]{lasserre15mp-generalization}
Jean~B Lasserre.
\newblock A generalization of l{\"o}wner-john's ellipsoid theorem.
\newblock \emph{Mathematical Programming}, 152:\penalty0 559--591, 2015.

\bibitem[Lasserre(2023)]{lasserre23crm-pell}
Jean~B Lasserre.
\newblock Pell's equation, sum-of-squares and equilibrium measures on a compact
  set.
\newblock \emph{Comptes Rendus. Math{\'e}matique}, 361\penalty0 (G5):\penalty0
  935--952, 2023.

\bibitem[Lasserre(2009)]{lasserre2009moments}
Jean~Bernard Lasserre.
\newblock \emph{Moments, positive polynomials and their applications},
  volume~1.
\newblock World Scientific, 2009.

\bibitem[Li et~al.(2023)Li, Yu, Conger, and Wierman]{li23arxiv-learning}
Yingying Li, Jing Yu, Lauren Conger, and Adam Wierman.
\newblock Learning the uncertainty sets for control dynamics via set
  membership: A non-asymptotic analysis.
\newblock \emph{arXiv preprint arXiv:2309.14648}, 2023.

\bibitem[Lofberg(2004)]{lofberg04yalmip}
Johan Lofberg.
\newblock Yalmip: A toolbox for modeling and optimization in matlab.
\newblock In \emph{2004 IEEE international conference on robotics and
  automation}, pages 284--289. IEEE, 2004.

\bibitem[Magnani et~al.(2005)Magnani, Lall, and Boyd]{magnani05cdc-tractable}
Alessandro Magnani, Sanjay Lall, and Stephen Boyd.
\newblock Tractable fitting with convex polynomials via sum-of-squares.
\newblock In \emph{Proceedings of the 44th IEEE Conference on Decision and
  Control}, pages 1672--1677. IEEE, 2005.

\bibitem[Milanese and Vicino(1991)]{milanese91automatica-optimal}
Mario Milanese and Antonio Vicino.
\newblock Optimal estimation theory for dynamic systems with set membership
  uncertainty: An overview.
\newblock \emph{Automatica}, 27\penalty0 (6):\penalty0 997--1009, 1991.

\bibitem[Moshtagh et~al.(2005)]{moshtagh05copt-minimum}
Nima Moshtagh et~al.
\newblock Minimum volume enclosing ellipsoid.
\newblock \emph{Convex optimization}, 111\penalty0 (January):\penalty0 1--9,
  2005.

\bibitem[Mustafa et~al.(2018)Mustafa, Stancu, Delanoue, and
  Codres]{mustafa18ras-guaranteed}
Mohamed Mustafa, Alexandru Stancu, Nicolas Delanoue, and Eduard Codres.
\newblock Guaranteed slam—an interval approach.
\newblock \emph{Robotics and Autonomous Systems}, 100:\penalty0 160--170, 2018.

\bibitem[Nie and Demmel(2005)]{nie05jgo-minimum}
Jiawang Nie and James~W Demmel.
\newblock Minimum ellipsoid bounds for solutions of polynomial systems via sum
  of squares.
\newblock \emph{Journal of Global Optimization}, 33\penalty0 (4):\penalty0
  511--525, 2005.

\bibitem[Nie and Schweighofer(2007)]{nie2007complexity}
Jiawang Nie and Markus Schweighofer.
\newblock On the complexity of putinar's positivstellensatz.
\newblock \emph{Journal of Complexity}, 23\penalty0 (1):\penalty0 135--150,
  2007.

\bibitem[Nocedal and Wright(1999)]{nocedal99book-numerical}
Jorge Nocedal and Stephen~J Wright.
\newblock \emph{Numerical optimization}.
\newblock Springer, 1999.

\bibitem[Oymak and Ozay(2019)]{oymak2019non}
Samet Oymak and Necmiye Ozay.
\newblock Non-asymptotic identification of lti systems from a single
  trajectory.
\newblock In \emph{2019 American control conference (ACC)}, pages 5655--5661.
  IEEE, 2019.

\bibitem[Parrilo(2003)]{parrilo03mp-semidefinite}
Pablo~A Parrilo.
\newblock Semidefinite programming relaxations for semialgebraic problems.
\newblock \emph{Mathematical programming}, 96:\penalty0 293--320, 2003.

\bibitem[Paulraj et~al.(2010)Paulraj, Sumathi,
  et~al.]{paulraj10mpe-comparative}
Sumathi Paulraj, P~Sumathi, et~al.
\newblock A comparative study of redundant constraints identification methods
  in linear programming problems.
\newblock \emph{Mathematical Problems in Engineering}, 2010, 2010.

\bibitem[Pavlakos et~al.(2017)Pavlakos, Zhou, Chan, Derpanis, and
  Daniilidis]{pavlakos17icra-heatmap}
Georgios Pavlakos, Xiaowei Zhou, Aaron Chan, Konstantinos~G Derpanis, and
  Kostas Daniilidis.
\newblock 6-dof object pose from semantic keypoints.
\newblock In \emph{IEEE Intl. Conf. on Robotics and Automation (ICRA)}, pages
  2011--2018. IEEE, 2017.

\bibitem[Pineda et~al.(2022)Pineda, Fan, Monge, Venkataraman, Sodhi, Chen,
  Ortiz, DeTone, Wang, Anderson, et~al.]{pineda22neurips-theseus}
Luis Pineda, Taosha Fan, Maurizio Monge, Shobha Venkataraman, Paloma Sodhi,
  Ricky~TQ Chen, Joseph Ortiz, Daniel DeTone, Austin Wang, Stuart Anderson,
  et~al.
\newblock Theseus: A library for differentiable nonlinear optimization.
\newblock \emph{Advances in Neural Information Processing Systems},
  35:\penalty0 3801--3818, 2022.

\bibitem[Putinar(1993)]{putinar1993positive}
Mihai Putinar.
\newblock Positive polynomials on compact semi-algebraic sets.
\newblock \emph{Indiana University Mathematics Journal}, 42\penalty0
  (3):\penalty0 969--984, 1993.

\bibitem[Sain et~al.(2016)Sain, Kadets, Paul, and Ray]{sain2016chebyshev}
Debmalya Sain, Vladimir Kadets, Kallol Paul, and Anubhab Ray.
\newblock Chebyshev centers that are not farthest points, 2016.

\bibitem[Sion(1958)]{sion1958general}
Maurice Sion.
\newblock On general minimax theorems.
\newblock 1958.

\bibitem[Stengel(1994)]{stengel94book-optimal}
Robert~F Stengel.
\newblock \emph{Optimal control and estimation}.
\newblock Courier Corporation, 1994.

\bibitem[Telgen(1983)]{telgen83ms-identifying}
Jan Telgen.
\newblock Identifying redundant constraints and implicit equalities in systems
  of linear constraints.
\newblock \emph{Management Science}, 29\penalty0 (10):\penalty0 1209--1222,
  1983.

\bibitem[Toh et~al.(1999)Toh, Todd, and T{\"u}t{\"u}nc{\"u}]{toh99oms-sdpt3}
Kim-Chuan Toh, Michael~J Todd, and Reha~H T{\"u}t{\"u}nc{\"u}.
\newblock Sdpt3—a matlab software package for semidefinite programming,
  version 1.3.
\newblock \emph{Optimization methods and software}, 11\penalty0 (1-4):\penalty0
  545--581, 1999.

\bibitem[Tsiamis and Pappas(2021)]{tsiamis2021linear}
Anastasios Tsiamis and George~J Pappas.
\newblock Linear systems can be hard to learn.
\newblock In \emph{2021 60th IEEE Conference on Decision and Control (CDC)},
  pages 2903--2910. IEEE, 2021.

\bibitem[Valero and Paulen(2019)]{valero2019ifac-recurparall}
Carlos~E Valero and Radoslav Paulen.
\newblock Effective recursive set-membership state estimation for robust linear
  mpc.
\newblock \emph{IFAC-PapersOnLine}, 52\penalty0 (1):\penalty0 486--491, 2019.

\bibitem[Vandenberghe et~al.(1998)Vandenberghe, Boyd, and
  Wu]{vandenberghe98simaa-determinant}
Lieven Vandenberghe, Stephen Boyd, and Shao-Po Wu.
\newblock Determinant maximization with linear matrix inequality constraints.
\newblock \emph{SIAM journal on matrix analysis and applications}, 19\penalty0
  (2):\penalty0 499--533, 1998.

\bibitem[Walter and Piet-Lahanier(1989)]{walter1989exact}
E~Walter and H{\'e}lene Piet-Lahanier.
\newblock Exact recursive polyhedral description of the feasible parameter set
  for bounded-error models.
\newblock \emph{IEEE Transactions on Automatic Control}, 34\penalty0
  (8):\penalty0 911--915, 1989.

\bibitem[Walter and Pronzato(1994)]{walter1994characterizing}
Eric Walter and Luc Pronzato.
\newblock Characterizing sets defined by inequalities.
\newblock \emph{IFAC Proceedings Volumes}, 27\penalty0 (8):\penalty0 325--336,
  1994.

\bibitem[{Xie, Miaolan}(2016)]{xie2016thesis}
{Xie, Miaolan}.
\newblock Inner approximation of convex cones via primal-dual ellipsoidal
  norms.
\newblock Master's thesis, 2016.
\newblock URL \url{http://hdl.handle.net/10012/10474}.

\bibitem[Yang and Carlone(2019)]{yang19iccv-quaternion}
Heng Yang and Luca Carlone.
\newblock A quaternion-based certifiably optimal solution to the wahba problem
  with outliers.
\newblock In \emph{Intl. Conf. on Computer Vision (ICCV)}, pages 1665--1674,
  2019.

\bibitem[Yang and Pavone(2023)]{yang2023object}
Heng Yang and Marco Pavone.
\newblock Object pose estimation with statistical guarantees: Conformal
  keypoint detection and geometric uncertainty propagation.
\newblock In \emph{Proceedings of the IEEE/CVF Conference on Computer Vision
  and Pattern Recognition}, pages 8947--8958, 2023.

\end{thebibliography}

\end{document}